\documentclass[11pt]{article}
\usepackage{graphicx} % Required for inserting images
\usepackage{enumitem}
\usepackage{amssymb}
\title{Path-Dependent Ergodic Optimal Control \\ and Backward Stochastic Differential Equations}
\author{Xuyang Lin \thanks{New York University, Tandon School of Engineering, Dept. \hspace{-7pt} of Financial and Risk Engineering}~\footnote{Email: xuyang.lin@nyu.edu} \and Mathieu Lise\footnotemark[1]~\footnote{Email: mathieu.lise@nyu.edu} \and Nizar Touzi \footnotemark[1]~\footnote{This author is partially supported by NSF grant $\#$DMS-2508581. Email: nizar.touzi@nyu.edu}}
\date{}

\usepackage[backend=biber, maxnames=99]{biblatex}

\AtBeginBibliography{\fontsize{9.8}{11.8}\selectfont}

\usepackage{geometry}
\usepackage{hyperref}
\hypersetup{
    colorlinks=true,
    linkcolor=blue,
    citecolor=magenta,
    filecolor=magenta,      
    urlcolor=cyan,
    pdftitle={Overleaf Example},
    pdfpagemode=FullScreen,
    }

\usepackage{amsthm}
\usepackage{setspace}
\usepackage{amsmath}
\usepackage{amsfonts}
\usepackage{mathtools}
\usepackage{mathrsfs}
\usepackage{amssymb}
\usepackage{tabularx}
\usepackage{algorithm}
\usepackage{array}
\usepackage{graphicx}
\usepackage{multirow}
\usepackage{amsthm}
\usepackage{algorithmic}
\usepackage{xcolor}
\usepackage{bm}
\usepackage{float}
\usepackage[overload]{empheq}
\usepackage{hyperref}
\usepackage{tabularx}
\usepackage{bbm}

\newcommand{\R}{\mathbb{R}}
\newcommand{\E}{\mathbb{E}}
\newcommand{\C}
{\mathcal{C}}

\newcommand{\X}{\mathcal{X}}
\newcommand{\K}{\mathcal{K}}
\newcommand{\HH}{\mathbb{H}^2}

\newtheorem{theorem}{Theorem}
\newtheorem{lemma}[theorem]{Lemma}
\newtheorem{proposition}[theorem]{Proposition}
\newtheorem{definition}[theorem]{Definition}

\theoremstyle{definition}
\newtheorem{remark}[theorem]{Remark}
\newtheorem{example}[theorem]{Example}
\newtheorem{assumption}{Assumption}

\numberwithin{equation}{section}
\numberwithin{theorem}{section}

\newcounter{assumptiontag}
\renewcommand{\theassumptiontag}{{\rm A\arabic{assumption}.\roman{assumptiontag}}}

\begin{document}

\maketitle

\setlength{\parindent}{0pt}
\begin{spacing}{1.15}

\begin{abstract}
    We investigate a new class of infinite-horizon backward stochastic differential equations for ergodic optimal control where the cost and state dynamics are time and path-dependent. The state process is defined on an unbounded underlying domain and satisfies an extended dissipativity condition. In contrast with the time-homogeneous Markovian setting, the optimal ergodic cost in our framework is characterized by the asymptotic behavior of a deterministic function, rather than by a single real constant. We obtain well-posedness, verification and stability properties, which extend the previous results in the literature on the Markov case.
\end{abstract}

%\tableofcontents

\vspace{5pt}


\vspace{5pt}
\section{Introduction}

\setlength{\parindent}{20pt}

Ergodic optimal control has received an important attention during the last four decades and has proved relevant in several areas of applied mathematics. A classical problem consists in minimizing over all admissible controls $\alpha \in \mathcal{A}$ the long run average cost functions:
\begin{equation}\label{ergodic_cost_intro}
\begin{array}{ccc}
      \displaystyle \underline{J}(x, \alpha) = \liminf_{T \rightarrow \infty} \frac{1}{T} \mathbb{E}^{\alpha} \int_0^T f(X^{x}_{t}, \alpha_t) \, dt   &  \text{or} & \displaystyle \overline{J}(x, \alpha) = \limsup_{T \rightarrow \infty} \frac{1}{T} \mathbb{E}^{\alpha} \int_0^T f(X^{x}_{t}, \alpha_t) \, dt \, , 
\end{array}
\end{equation}
with the state process $X$ following the dynamics:
\begin{equation*}\label{state_dynamics_intro}
    \begin{array}{ccc}
         dX_t^x = b(X^{x}_{t}) dt + (dW_t^{\alpha} + \alpha_t dt) & \text{and} & X^{x}_0 = x \in \mathbb{R}^d \, ,
    \end{array} 
\end{equation*}
where $\mathbb{P}^{\alpha}$ is a probability measure associated to $\alpha$ and $W^{\alpha}$ is a $\mathbb{P}^{\alpha}$-Brownian motion. This problem has been studied in depth through analytical techniques. In particular, under some conditions on $f$ and $b$, the optimal cost $ \lambda = \inf_{\alpha} \underline{J}(x, \alpha) = \inf_{\alpha} \overline{J}(x, \alpha)$ can be shown to be independent of the initial state $x$ and is the only real number for which the ergodic Hamilton-Jacobi-Bellman (HJB) equation:

\vspace{-5pt}

\begin{equation}\label{ergodic_HJB_intro}
\begin{array}{cc}
    \displaystyle  -\frac{1}{2} \Delta v - b \cdot \nabla v - \inf_{a} \big\{  a \cdot \nabla v + f(x, a) \big\} + \lambda = 0  \, , & x \in \mathbb{R}^d \, ,
\end{array}
\end{equation}
has a solution $(v, \lambda)$ in some appropriate sense. We refer the reader to the early papers of \citeauthor{Arisawa_Lions} \cite{Arisawa_Lions} for the case where the coefficients are periodic in $x$, \citeauthor{Bensoussan_Frehse} \cite{Bensoussan_Frehse} for a drift-controlled problem in $\mathbb{R}^d$, \citeauthor{Lasry_Lions_state_constraints} \cite{Lasry_Lions_state_constraints} for a model with state constraints and \citeauthor{borkar} \cite{borkar} for a variant with an almost sure long time average cost. Note that we can also consider the approach of ergodic control through infinite-horizon discounted cost functions and obtain the same ergodic value $\lambda$, following the well-known Tauberian Theorems. 

In the more recent literature, ergodic problems have been successfully described by backward stochastic differential equations (backward SDEs or BSDEs for short) since the pioneering paper of \citeauthor{furhman_ergodic} \cite{furhman_ergodic}. In their framework, $b$ satisfies a dissipativity condition which forces the ergodicity of $X$ in $\mathbb{R}^d$. This condition is connected to the Lyapounov function approach, see \textit{e.g.} \citeauthor{ergodic_control_diffusion_processes} \cite{ergodic_control_diffusion_processes}. The HJB equation \eqref{ergodic_HJB_intro} can then be replaced by a family of infinite-horizon backward SDEs $(Y^x, Z^x, \lambda)_{x \in \mathbb{R}^d}$ with:

\vspace{-15pt}

\begin{equation} \label{Ergodic_BSDE_intro}
    \begin{array}{cccc}
        \displaystyle Y_t^x = Y_T^x - \lambda (T-t) + \int_t^T H(X_s^x, Z_s^x) ds - \int_t^T Z_s^x \cdot dW_s \, , & 0 \leq t \leq T < \infty & \text{and} & x \in \mathbb{R}^d \, ,
    \end{array}
\end{equation}
where $H({x}, z) \coloneqq \inf_a \big\{ f(x, a) +a \cdot z  \big\}$ is assumed to be Lipschitz continuous in both the $x$ and $z$ variables. This framework is natural for stochastic optimal control problems, as the backward SDE \eqref{Ergodic_BSDE_intro} directly reformulates the dynamic programming principle. Formally, for all $x \in \mathbb{R}^d$, the triplet $(Y^x, Z^x, \lambda)$ is obtained as the limit when $\rho \downarrow 0$ of $(Y^{\rho, x} - Y^{\rho, 0}, Z^{\rho, x}, \rho Y^{\rho, x})$, where $(Y^{\rho, x}, Z^{\rho, x})$ solves an infinite-horizon backward SDEs with discount $\rho > 0$:
\begin{equation*}\label{infinite_horizon_Backwards SDE}
    \begin{array}{cc}
        \displaystyle Y_t^{\rho, x} = Y_T^{\rho, x} + \int_t^T \big(H(X_s^x, Z_s^{\rho, x}) - \rho Y^{\rho, x}_s \big) ds - \int_t^T Z_s^{\rho, x} \cdot dW_s \, , & 0 \leq t \leq T < \infty \, .
    \end{array}
\end{equation*}
Several variants of the ergodic backward SDE \eqref{Ergodic_BSDE_intro} have been established since \cite{furhman_ergodic}. We can mention for example the case with a weaker dissipativity assumption on $b$ by \citeauthor{ergodic_weak_dissipative} \cite{ergodic_weak_dissipative}, the problem with Neumann boundary conditions in a bounded domain of \citeauthor{ergodic_BSDE_neumann} \cite{ergodic_BSDE_neumann} and the papers of \citeauthor{ergodic_performance_process} \cite{ergodic_performance_process} and \citeauthor{ergodic_bsde_superquadratic} \cite{ergodic_bsde_superquadratic} which extend \eqref{Ergodic_BSDE_intro} to more general Hamiltonians $H$. 
Throughout these papers, one central assumption systematically remains: the problem is entirely time-homogeneous Markovian, \textit{i.e.} the coefficients $f$ and $b$ depend only on the current state $X_t^x$. 

Our objective in this paper is to extend this backward SDE approach to a time and path dependent framework, building on the recent progress regarding the ergodic properties of time and path dependent diffusion processes, see \citeauthor{Ergodic_FSDE} \cite{Ergodic_FSDE} for the setting of functional (or delayed) stochastic differential equations, \citeauthor{Evolution_system_measures} \cite{Evolution_system_measures} and \citeauthor{OU_periodic} \cite{OU_periodic} for the notion of \textit{evolution system of measures} as a substitute of the classical invariant measure for Markov time-inhomogeneous processes. We also refer to \citeauthor{Benaim_ergodicity_inhomogeneous} \cite{Benaim_ergodicity_inhomogeneous} and \citeauthor{entrance_measure_inhomogeneous} \cite{entrance_measure_inhomogeneous}, and the corresponding PDE approach through the eigenvalue problem for linear parabolic PDEs, see \textit{e.g.} \hspace{-7pt} \citeauthor{Harnack_inequalities} \cite{Harnack_inequalities} and \citeauthor{principal_eigenvalues_parabolic} \cite{principal_eigenvalues_parabolic}. We shall borrow the framework and the generalized dissipativity condition introduced in \citeauthor{FSDE_infinite_delay} in \cite{FSDE_infinite_delay} and extended by \citeauthor{Infinite_delay_nonhomogeneous} \cite{Infinite_delay_nonhomogeneous} for inhomogeneous systems with infinite delay. 

Optimal control of delayed systems was examined under various angles, see \textit{e.g.}\hspace{-2pt} \citeauthor{HJB_delay} \cite{HJB_delay} and \citeauthor{Feo_control_delay} \cite{Feo_control_delay}. The study of long-term average cost control in a non-Markovian setting was conducted by \citeauthor{LQ_italiennes} \cite{LQ_italiennes} in the linear-quadratic case, taking advantage of the explicit characterization of LQ control problems through the corresponding backward stochastic Riccati equations. For Markov time-periodic control problems under general criteria, ergodic backward SDEs were used in \citeauthor{cohen_periodic_EBSDE} \cite{cohen_periodic_EBSDE}. To the best of our knowledge, the study of ergodic control for general Markov time-inhomogeneous ergodic control was initiated by \citeauthor{zero_sum_non_stationary_ergodic} \cite{zero_sum_non_stationary_ergodic} and \citeauthor{Guo_Huang_Zhang} \cite{Guo_Huang_Zhang} in the context of stochastic games and Markov decision processes. These articles were an inspiration for our work. 

In the present paper, we adopt the framework of backward SDEs to study optimal ergodic control in continuous time, where the coefficients $b$ and $H$ (or $f$) are time and path-dependent. In particular, the initial condition $\xi$ of the path-state process $X^\xi$ is a continuous function, and the cost function $f(t,\xi,a)$ depends on both time and the path. Under mild assumptions on $b$ and $H$, we construct a new ergodic backward SDE, where the linear map $t \mapsto \lambda t$ in \eqref{Ergodic_BSDE_intro} is replaced by a function $t \mapsto \Lambda(t)$ from $\R_+$ to $\mathbb{R}$ independent of the initial data. 

For the corresponding ergodic control problem, we define similar performance maps to \eqref{ergodic_cost_intro} under path-dependent dynamics of the state process $X$ and the cost $f$. We obtain the following two-sided verification result:
\begin{equation}\label{verif_argument_intro}
    \begin{array}{ccc}
         \displaystyle \inf_{\alpha} \underline{J} = \liminf_{T \rightarrow \infty} \frac{\Lambda(T)}{T}  & \text{and} & \displaystyle \inf_{\alpha} \overline{J} = \limsup_{T \rightarrow \infty} \frac{\Lambda(T)}{T} \, .
    \end{array}
\end{equation}
Moreover, the function $\Lambda$ is sufficient to characterize a wide class of ergodic optimal control problems arising from the system $(b,f)$, in which the time averaging is replaced by an averaging with respect to a general non-increasing density $K:\mathbb R_+\to\mathbb R_+$ by:
\begin{equation*}
    J^{K_\rho}(\xi,\alpha)=\mathbb{E}^{\alpha} \int_0^{\infty} K_\rho(t) f(t,(X^{\xi}_s)_{s\leq t}, \alpha_t) \, dt \, ,
~~\mbox{with}~~
K_\rho(t):=\rho K(\rho t) \, ,~t\ge 0 \, .
\end{equation*}
In particular, we prove that $\Lambda$ can be chosen to be differentiable and that we have the following characterization of the ergodic control problem:
\begin{equation*}
\begin{aligned}
\inf_{\alpha}\liminf_{\rho \downarrow 0} J^{K_\rho}(\xi,\alpha)
&=
\liminf_{\rho \downarrow 0}\inf_{\alpha} J^{K_\rho}(\xi,\alpha)
=
\liminf_{\rho \downarrow 0}
\int_0^{\infty} K_\rho(t)\Lambda'(t)\,dt \,, \\
\inf_{\alpha}\limsup_{\rho \downarrow 0} J^{K_\rho}(\xi,\alpha)
&=
\limsup_{\rho \downarrow 0}\inf_{\alpha} J^{K_\rho}(\xi,\alpha)
=
\limsup_{\rho \downarrow 0}
\int_0^{\infty} K_\rho(t)\Lambda'(t)\,dt \,.
\end{aligned}
\end{equation*}
Note that the choice $K(t)=\mathbf{1}_{[0,1)}(t)$ recovers the uniform averaging \eqref{verif_argument_intro}, while $K(t)=e^{-t}$ recovers the exponential averaging, which is also standard in the literature; see \citeauthor{Arisawa_Lions}~\cite{Arisawa_Lions}. 

This relationship \eqref{verif_argument_intro} is consistent with the recent literature on Floquet bundles and generalized principal eigenvalue for second-order parabolic operators. Indeed, if $X$ is homogeneous Markov and is constrained to stay in a bounded domain, the optimal ergodic cost $\lambda$ is related to the principal eigenvalue of the elliptic operator $\mathcal{L} = \frac{1}{2} \Delta + b \cdot \nabla$, see \textit{e.g.} \citeauthor{Lasry_Lions_state_constraints} \cite{Lasry_Lions_state_constraints} and \citeauthor{quasitrationary_ergodic} \cite{quasitrationary_ergodic}. The notion of principal eigenvalue was extended to second-order parabolic operators in bounded domains and reflects the same structure as \eqref{verif_argument_intro}, see \textit{e.g.} \cite{principal_floquet_bundle, Harnack_inequalities} and \citeauthor{principal_eigenvalues_parabolic} \cite{principal_eigenvalues_parabolic} for a recent review. Note that if the system is time-periodic, this notion coincides with the natural principal eigenvalue for elliptic operators. Similarly in our case we obtain that $\liminf \frac{\Lambda(T)}{T} = \limsup \frac{\Lambda(T)}{T}$ when $(b, H)$ is time-periodic, which is consistent with \cite{cohen_periodic_EBSDE}.

The paper is organized as follows. After introducing the general framework in Section~\ref{sec:general_framework} and presenting our main results in Section~\ref{sec:main_results}, Sections~\ref{sec:forward equation} and~\ref{sec:ebsde} establish the well-posedness of our path-dependent ergodic backward SDE $(Y,Z,\Lambda)$. In Section~\ref{sec:stability}, we prove some stability properties with respect to the coefficients $b$, $H$ and$f$. We finally provide further results in special cases in Section \ref{sec:special_cases}, in particular the uniqueness of the process $Y$ in a suitable sense and a canonical representation of $\Lambda$.

\section{General framework} \label{sec:general_framework}

\subsection{Notations}

Given an integer $d \geq 1$ and a real number $\tau \in \R$, we introduce the following notations:
\begin{itemize}
    
    \item Given a real number $\theta > 0$, $\mathcal{C}_{\theta}$ is the space of continuous functions $\xi : (-\infty, 0] \rightarrow \mathbb{R}^d$ such that:
    \begin{equation*}
        \lim_{s\rightarrow-\infty}e^{\theta s}\xi(s) \; \text{exists.}
    \end{equation*}
    Endowed with the norm $ \displaystyle | \xi |_{\theta} = \sup_{s \leq 0} \, e^{\theta s}|\xi(s)|$, the space $(\mathcal{C}_{\theta},|\cdot|_{\theta})$ is a Polish space. Throughout the paper, we write $\xi_s$ for $\xi(s)$ for simplicity.
    
    \item ${\rm{UC}}_{\tau}$ (resp. ${\rm{AC}}_{\tau}$) is the space of uniformly continuous (resp. absolutely continuous) functions $\Lambda:[\tau, \infty)\to\mathbb{R}$. We say that $\Lambda$ is in $C^{\infty}$ if it is infinitely differentiable, the derivatives at $\tau$ being understood as right derivatives.
    
\end{itemize}

\vspace{2pt}

Consider a probability space $(\Omega, \mathcal{F}, \mathbb{P})$ and a filtration $\mathbb{F} = (\mathcal{F}_t)_{t \geq \tau}$. Given $T \geq \tau$,
\begin{itemize}
    \item $\mathbb{H}^0_{\tau}(\R^d)$ is the set of $\mathbb{F}$-progressively measurable processes $\psi : \Omega \times [\tau, \infty) \rightarrow \R^d$. In the following, we use the compact notation $\psi_t = \psi(\omega, t)$ for any $\psi \in \mathbb{H}^0_{\tau}(\R^d)$.
    
    \item $\mathbb{H}^2_{\tau, T}(\R^d)$ is the subset of processes $\psi  \in \mathbb{H}^0_{\tau}(\R^d)$ such that $\displaystyle\mathbb{E}\int_{\tau}^T |\psi_t|^2 \, dt < \infty$.

    \vspace{3pt}

    \item $\mathbb{S}^2_{\tau, T}(\R^d)$ is the subset of continuous processes $\psi \in \mathbb{H}^0_{\tau}(\R^d)$ such that $ \displaystyle \mathbb{E}\sup_{t \in [\tau,T]} |\psi_t|^2 < \infty$.

    \item If a process $(\beta_s)_{s \geq \tau} \in \mathbb{H}^0_{\tau}(\R^d)$ is bounded and $W^{\tau}$ is a $\mathbb{P}-$Brownian motion starting at $\tau$, we denote by $\mathbb{P}^{\tau, T,\beta}$ the unique probability distribution such that $W^{\tau, \beta} \coloneqq W^{\tau} - \int_{\tau}^{\cdot} \beta_s \, ds$ is a $\mathbb{P}^{\tau, T,\beta}-$Brownian motion in $[\tau, T]$. The existence and uniqueness of $\mathbb{P}^{\tau, T,\beta}$ are guaranteed by Girsanov's Theorem. We denote $\mathbb{E}^{\tau, T,\beta}$ the expectation with respect to $\mathbb{P}^{\tau, T,\beta}$. In the rest of the paper, we omit the dependence on $\tau$ or $T$ from the notation whenever there is no ambiguity.

    \vspace{3pt}

    \item Given a constant $B>0$, let $\mathcal P_B:=\{\beta\in\mathbb H^0_\tau(\mathbb R^d): |\beta_t|\le B,\ d\mathbb P\otimes dt\text{-a.e.}\}$.
\end{itemize}

\vspace{2pt}

For any set $I$, we finally denote by $\mathbf{E}_{\tau}(I)$ the set of families of processes $(\varphi, \psi) = (\varphi^{i}, \psi^{i})_{i \in I}$ indexed by $I$ satisfying $(\varphi^{i}, \psi^{i}) \in \mathbb{S}^2_{\tau, T}(\R) \times \mathbb{H}^2_{\tau,T}(\R^d)$ for all $T > \tau$ and $i \in I$. 

\subsection{Kernel averaging}\label{sec:kernel_averaging}

In the following we define the notions of \textit{kernel} and \textit{ergodic value}, which will be used in our main results of Section \ref{sec:main_results}. We set the time origin to $\tau = 0$ and we therefore omit it in the notations.

\begin{definition}\label{def:values} 
{\rm (i)} $\K\!:=\!\big\{K\!: \R_+\!\to\!\R_+~\text{non-increasing}, \text{c\`adl\`ag},\ \int_0^\infty\!K(t)\,dt=1 \big\}$ is the set of \textnormal{kernels}; \\

\vspace{-8pt}

\noindent {\rm (ii)} The uniform Kernel $K^0:=\mathbf{1}_{[0,1)}$, the exponential kernel $K^{\rm exp}:=e^{-\cdot}$ and the rescaled kernel $K_\rho:=\rho K(\rho\cdot)$ are elements of $\K$, for all $K\in\K$, $\rho>0$; \\

\vspace{-8pt}

\noindent {\rm (iii)} For $\Lambda \in {\rm{UC}}$ and $K \in \K$, we define the two \textnormal{ergodic values} of $\Lambda$ by:
    \begin{equation*}
    \begin{array}{ccc}
        \displaystyle \lambda_*(K)\coloneqq \liminf_{\rho\downarrow 0} - \int_0^\infty \Lambda(t) \,dK_{\rho}(t)& \text{and} & \displaystyle \lambda^*(K)\coloneqq \limsup_{\rho\downarrow 0} - \int_0^\infty \Lambda(t) \, dK_{\rho}(t) \,.
    \end{array}
    \end{equation*}
If in addition $\Lambda\in{\rm UC}\cap{\rm AC}$, then $\Lambda$ has linear growth and we have by integration by parts:
\begin{equation*}
\begin{array}{ccc}
    \displaystyle \lambda_*(K)\coloneqq \liminf_{\rho\downarrow 0}\int_0^\infty K_{\rho}(t)\Lambda'(t)\,dt & \text{and} & \displaystyle \lambda^*(K)\coloneqq \limsup_{\rho\downarrow 0}\int_0^\infty K_{\rho}(t) \Lambda'(t)\,dt \,,
\end{array}
\end{equation*}
as $\lim_{T\to\infty}TK(T)=0$ due to the non-increase and the integrability of $K$ at $\infty$.
\end{definition}

The non-reference to the map $\Lambda$ in the notations $\lambda_*$ and $\lambda^*$ is motivated by Theorem \ref{thm:existence_uniqueness} below which shows that these ergodic values do not depend on the $\Lambda\in{\rm UC}$ obtained in that context.

If $\Lambda' \equiv \lambda$ is constant, then the ergodic values of Definition \ref{def:values} are all equal: $\int_0^{\infty} \rho K(\rho t) \lambda \, dt = \int_0^{\infty} K(s) \lambda \, ds = \lambda$. In general, the relationship between the ergodic values follow from the well-known Wiener-type Tauberian theorem. Let $\mathscr{F}(g)$ denote the Fourier transform of an integrable function
$g:\mathbb{R}\to\mathbb{R}$, and define $\K_W$ as the set:

\vspace{-3pt}

\begin{equation*}
    \mathcal{K}_W:=\Big\{K\in\mathcal{K}\cap {\rm AC}:\ \mathscr{F}\bigl[e^{-2u}K'(e^{-u})\bigr](q)\neq 0,\ \text{for all } q\in\mathbb{R}\Big\} \, .
\end{equation*}

\begin{proposition} \label{prop:relation_value}
For all $K\in\K$, we have $\lambda_*(K^0) \le \lambda_*(K) \le \lambda^*(K) \le \lambda^*(K^0)$.
Moreover, 
\begin{center}
    \hspace{-5mm}
    {\rm (1)} $\lambda_*(K^0)=\lambda^*(K^0)$ iff {\rm (2)} $\lambda_*(K)=\lambda^*(K)$ for all $K\in\K$  iff {\rm (3)} $\lambda_*(K)=\lambda^*(K)$ for some $K\in\K_W$.
\end{center}
As a consequence, in either of these cases, we have $\lambda_*(K)=\lambda^*(K)=\lambda^*(K^0)$ for all $K\in\K$.
\end{proposition}

The proof is postponed in Appendix \ref{app:ergodic_values}.

\begin{example}
(i) The uniform kernel $K^0\not\in\mathrm{AC}$, and therefore $K^0\notin\mathcal K_W$. On the other hand, the exponential kernel $K^{\mathrm{exp}}\in\mathcal K_W$, since $K^{\mathrm{exp}}\in\mathrm{AC}$ and $
\mathscr{F}\bigl[e^{-2u}(K^{\textnormal{exp}})'(e^{-u})\bigr](q)=-\Gamma(2+iq)\neq 0$. \\

\vspace{-8pt}

\noindent (ii) More generally, if $k:\R_+\to\R_+$ is continuous and satisfies
$\int_0^\infty t\,k(t)\,dt=1$, and $K(t):=\int_t^\infty k(s)\,ds$, then $K\in\mathcal{K}\cap\rm{AC}$ and

\vspace{-3pt}

\[
\mathscr{F}\bigl(e^{-2u}K'(e^{-u})\bigr)(q)
=-\mathscr{F}\bigl(e^{-2u}k(e^{-u})\bigr)(q)
=-\int_0^\infty t^{1+i q} \, k(t)\,dt \, .
\]
Hence $K\in\mathcal{K}_W$ whenever the Mellin transform $\int_0^\infty t^{1+i q} \, k(t)\,dt$ does not vanish on $\mathbb{R}$. \\

\vspace{-8pt}

\noindent (iii) With $k_\nu(t)=\frac{1}{\Gamma(\nu+1)}t^{\nu-1}e^{-t}$, $\nu>0$, we compute that $\mathscr{F}\bigl(e^{-2u}k_\nu (e^{-u})\bigr)(q)
=
\frac{\Gamma(\nu+1+i q)}{\Gamma(\nu+1)}\neq 0$, where $\Gamma$ is the Gamma function.
\end{example}

\begin{remark}\label{inequalities_ergodic_values_rmk}
By Proposition~\ref{prop:relation_value}, for $K\in \K_W$, if the four ergodic values $\lambda_*(K^0)$, $\lambda^*(K^0)$, $\lambda_*(K)$ and $\lambda^*(K)$ are not equal, the only possible configurations are: 
\begin{equation*}
\begin{array}{cc}
     \lambda_*(K^0)=\lambda_*(K)<\lambda^*(K)=\lambda^*(K^0) \, , & \lambda_*(K^0)=\lambda_*(K)<\lambda^*(K)<\lambda^*(K^0) \, , \\[13pt]
     \lambda_*(K^0)<\lambda_*(K)<\lambda^*(K)=\lambda^*(K^0)  \, , & \lambda_*(K^0)<\lambda_*(K)<\lambda^*(K)<\lambda^*(K^0) \, .
\end{array}
\end{equation*}
In particular, if $K = K^{\textnormal{exp}}$, all cases may occur. Examples can be found in \citeauthor{BishopFeinbergZhang2014}~\cite{BishopFeinbergZhang2014} in discrete time. The next example illustrates the last case in our continuous time setting.
\end{remark}

\begin{example}\label{ex:lim_lambda/T}
Let $\Lambda(t) \coloneqq \int_0^t\sin(\ln(1+s))ds$. For the exponential kernel $K^{\textnormal{exp}}\in \K_W,$ we show that $\lambda_*(K^0)<\lambda_*(K^{\textnormal{exp}})<\lambda^*(K^{\textnormal{exp}})<\lambda^*(K^0)$. Note that one has:
\[
\int_0^\infty K^0_{\frac1T}(t)\Lambda'(t)dt=\frac1T\int_{0}^{\ln(1+T)} \hspace{-7pt} e^{u}\sin u\,du
=\frac{1+T}{2T} \big[\sin (\ln(1+T))-\cos (\ln(1+T)) \big]+\frac{1}{2T} \, .
\]
We obtain the limits $\lambda_*(K^0)=-\frac{\sqrt{2}}{2}$ and $\lambda^*(K^0)=\frac{\sqrt{2}}{2}$. Moreover, by integration by parts,
\begin{equation*}
   \hspace{-5pt} \int_0^\infty\!\!\! K^{\exp}_\rho(t)\Lambda'(t)\,dt
    =e^{\rho}\,\text{Im}\Big(\rho^{-i}\!\int_{\rho}^{\infty}\!\!\! e^{-u}u^{i}\,du\Big) = |\Gamma(1+i)|\sin(\varphi-\ln\rho)+o(1),
    ~\mbox{with}~
    e^{i\varphi}:=\frac{\Gamma(1+i)}{|\Gamma(1+i)|},
\end{equation*}
where $o(1)\to 0$ as $\rho \downarrow 0$. Then $\lambda^*(K^{\textnormal{exp}})=-\lambda_*(K^{\textnormal{exp}}) = |\Gamma(1+i)|= \sqrt{\frac{\pi}{\sinh(\pi)}}\in(0,\frac{\sqrt{2}}{2})$.
\end{example}

\subsection{State equation with infinite delay}

Let $(\Omega, \mathcal{F}, \mathbb{P})$ be a complete probability space, $\tau \in \R$ and $W^{\tau} = (W_t^{\tau})_{t \geq \tau}$ be a $\mathbb{R}^d$-valued Brownian motion on this space. Denote by $\mathbb{F}^{\tau}  = (\mathcal{F}_t^{\tau})_{t \geq \tau}$ the $\mathbb{P}$-completion of the filtration generated by $W^{\tau}$. Fix $\theta > 0$ and let $b : \mathbb{R} \times \mathcal{C}_{\theta} \rightarrow \mathbb{R}^d$ be a measurable function. For convenience, we did not borrow the standard notations from the literature on functional stochastic differential equations. Our notations are inspired from the paper of \citeauthor{HJB_delay} \cite{HJB_delay}. Given $\tau \in \mathbb{R}$ and $\xi \in \mathcal{C}_{\theta}$, we define the state process $X$ through the following functional stochastic differential equation, with coefficients depending on the path process $(\mathcal{X}_t^{\tau, \xi})_{t \geq \tau}$ defined by $\mathcal{X}_t^{\tau, \xi}(s) = X_{t + s}^{\tau, \xi}$ for $s \leq 0$:
\begin{equation}\label{Functional_SDE}
    \begin{array}{cccc}
        \displaystyle dX_t^{\tau, \xi} = b(t, \mathcal{X}^{\tau, \xi}_t) dt + d W_t^{\tau}  & \text{for } \, t \geq \tau \, ,   & \text{and} & \mathcal{X}^{\tau, \xi}_{\tau} = \xi \,.
    \end{array}
\end{equation}
We shall see in Proposition \ref{prop:FSDE_existence} below that $\X$ is a Markov process on $\C_{\theta}$, and we therefore denote by $P_{s, t}$ the transition semi-group $P_{s, t} [\phi] (\xi) = \mathbb{E}[\phi(\mathcal{X}^{s, \xi}_t)]$ for all bounded measurable functions $\phi : \mathcal{C}_{\theta} \to \mathbb{R}$.

\begin{definition}
    We say that a family of measures $(m_t)_{t \in \R}$ on $\C_{\theta}$ is an \textnormal{evolution system of measures} for the transition semi-group $P_{s,t}$ if $m_s P_{s, t} = m_t$, \textnormal{i.e.} 
    \begin{equation*}
        \begin{array}{cc}
             \displaystyle \int_{\mathcal{C}_{\theta}} P_{s, t} [\phi](\xi) \, m_s(d\xi) = \int_{\mathcal{C}_{\theta}} \phi(\xi) \, m_t(d\xi) \, , & \text{for all } \, s \leq t \in \mathbb{R}~\mbox{and}~\phi~\mbox{bounded continuous}.
        \end{array}
    \end{equation*}
\end{definition}

\begin{assumption} \label{dissipativity_assumption}
\setcounter{assumptiontag}{0}
    \begin{itemize}
        \item[]

        \refstepcounter{assumptiontag}
        \item[] \hspace{-30pt} (\theassumptiontag) \label{Lpz_condition_lip_growth_condition} \hspace{2pt} \textnormal{Lipschitz and linear growth:} {\it There exists $L_b>0$ such that for any $t \in \mathbb{R}$ and $\xi, \xi'\in\mathcal C_{\theta}$,
        \begin{equation*}
            \begin{array}{ccc}
                \big| b(t, \xi) - b(t, \xi') \big| \leq L_b \, | \xi - \xi'|_{\theta} & \text{and} &
                |b(t, \xi)| \leq L_b \big(1 + | \xi |_{\theta} \big) \,.
            \end{array}
        \end{equation*}}

        \refstepcounter{assumptiontag}
        \item[] \hspace{-26pt}(\theassumptiontag) \label{dissipativity_condition} \hspace{1pt} \textnormal{Dissipativity:} {\it There exist $\theta_0 > \theta$ and $\eta_1, \eta_2 > 0$ such that $\eta_1 >2\theta - \frac{\eta_2}{2(\theta_0-\theta)}$ and
        \begin{equation*}
        \begin{array}{cc}
             \displaystyle 2\bigl(\xi_0 - \xi'_0\bigr)\cdot\bigl(b(t, \xi)-b(t, \xi')\bigr) \le -\eta_1 \,  \bigl|\xi_0 - \xi'_0\bigr|^2 +\eta_2 \int_{-\infty}^{0} e^{2\theta_0 s}\bigl| \xi_{s} - \xi'_{s}\bigr|^2\,ds,
             ~t \in \mathbb{R},~\xi, \xi'\in\mathcal C_{\theta} \, .
        \end{array}
        \end{equation*}}
    \end{itemize}
\end{assumption} 

\begin{remark}\label{rmk:finite_delay_ass}
A particular case of the above framework is when the drift $b$ only depends on the recent history on a finite interval $[-r,0]$, for some $r>0$. More precisely, there exists a measurable map $\bar b:\R\times C([-r,0];\R^d)\to\R^d$ such that $b(t,\xi)=\bar b\bigl(t,\xi|_{[-r,0]}\bigr)$ for all $t \ge 0$. This case was the most widely studied in the literature, see \citeauthor{Mohammed_FSDE} \cite{Mohammed_FSDE} and \citeauthor{Ergodic_FSDE} \cite{Ergodic_FSDE}.
\end{remark}

\begin{proposition}\label{prop:FSDE_existence}
    Under Assumption~\ref{dissipativity_assumption}, for every $\tau \in \mathbb{R}$ and $\xi\in\mathcal{C}_\theta$, the SDE~\eqref{Functional_SDE} admits a unique strong solution $(X_t^{\tau, \xi})_{t\ge\tau}$. Moreover, the corresponding path process $(\mathcal{X}_t^{\tau, \xi})_{t\ge\tau}$ is a $\mathcal{C}_\theta$-valued time-inhomogeneous Markov process satisfying, for some constant $C>0$,
\begin{equation}\label{eq:estimate_X_Lip}
     \displaystyle |\mathcal{X}_t^{\tau, \xi} - \mathcal{X}_t^{\tau, \xi'}|_\theta \le Ce^{-\theta t}|\xi-\xi'|_\theta \, , \; d\mathbb{P}\otimes dt-\text{a.e. and} ~
     \displaystyle \sup_{\beta\in\mathcal{P}_B}\sup_{t\ge \tau} \, \E^\beta\big[|\mathcal{X}_t^{\tau, \xi}|_\theta\big]<\infty,~B>0.
\end{equation}
Moreover, there exists a evolution system of measures $(m_t)_{t \in \mathbb{R}}$ for the semi-group $P_{s, t}$.
\end{proposition}

\begin{proof}
    The existence, uniqueness, continuity and Markov property for the delayed SDE \eqref{Functional_SDE} was first studied by \citeauthor{FSDE_infinite_delay} \cite{FSDE_infinite_delay} in the time-homogeneous case, and then by \citeauthor{Infinite_delay_nonhomogeneous} \cite{Infinite_delay_nonhomogeneous} for the time-inhomegenous case. The first estimate of \eqref{eq:estimate_X_Lip} follows from Proposition~\ref{prop:forward_lip} below, while the second is implied by Proposition~\ref{prop:forward_holder} below. The existence of the family $(m_t)_{t \in \mathbb{R}}$ is showed in \cite{Infinite_delay_nonhomogeneous}.
\end{proof}

\begin{remark} The \emph{dissipativity} condition of
Assumption~\eqref{dissipativity_condition} is satisfied for example by an Ornstein-Uhlenbeck process, corresponding to $b(t, \xi) = - \eta_0 \, \xi_0$ for some $\eta_0 > 0$. For time-homogeneous controlled Markov processes, the dissipativity condition was used in \citeauthor{Goldys_Maslowski} \cite{Goldys_Maslowski} and later in \citeauthor{furhman_ergodic}~\cite{furhman_ergodic}. For equations with finite delay, an equivalent condition can be found in \citeauthor{Ergodic_FSDE} \cite{Ergodic_FSDE}. In the general case, the essential role of the constraint $\eta_1>2\theta+\frac{\eta_2}{2(\theta_0-\theta)}$ will become apparent in the proof of Lemma~\ref{lem:exp_decreasing} below. Note that under this condition the function $V(\xi) = | \xi|_{\theta}^2$ is a Lyapunov function for $X$, see \citeauthor{Ergodicity_neutral_SDE} \cite[Proposition 1.2]{Ergodicity_neutral_SDE}. We finally point out that all of the following results hold if $\eta_1, \eta_2 : \mathbb{R} \to \mathbb{R}$ are two nonnegative bounded functions satisfying $\eta_1(t) -2\theta+\frac{\eta_2(t)}{2(\theta_0-\theta)} \geq \eta_0$ for some constant $\eta_0 > 0$.
\end{remark}

\begin{example}\label{ex:OU_evolution_system}
    Consider a time-inhomogeneous Ornstein-Uhlenbeck process with dynamics with initial condition $(\tau, x) \in \mathbb{R} \times \mathbb{R}$:
    \vspace{-5pt}
    \begin{equation} \label{eq:OU_non_homog}
        \begin{array}{cc}
             d X^{\tau, x}_t = - \eta(t) X_t^{\tau, x} dt  + d W_t^{\tau} \, , & X^{\tau, x}_{\tau} = x \, ,
        \end{array}
    \end{equation}
    where $\eta(t) \geq \eta_0 > 0$ is a measurable function. In \cite{OU_periodic}, \citeauthor{OU_periodic} exhibited an evolution system of measures for $X$. First note that:
    \begin{equation*}
        \begin{array}{cc}
            \displaystyle \delta_x P_{\tau, t} = \mathcal{N}(x \, e^{-S_{\tau, t}}, \Sigma_{\tau, t}) \, , & \text{where } \, \displaystyle S_{\tau, t} \coloneqq \int_{\tau}^t \eta(r) \, dr  \, \text{ and } \, \Sigma_{\tau, t} \coloneqq \int_{\tau}^t e^{-2 S_{r, t}} \, dr \, .
        \end{array}
    \end{equation*}
    Then letting $m_t \coloneqq \mathcal{N}(0, v(t))$, where $v\geq0$ satisfies $v' = - 2 \eta v + 1$ on $\mathbb{R}$,
    we obtain that $m_s P_{s, t} = m_t$ for all $s \leq t \in \mathbb{R}$. Note that the solutions of this ordinary differential equation are not unique, so there is no uniqueness for the evolution systems of measures in general. In \cite{OU_periodic}, uniqueness is obtained under the assumption that $\eta$ is time-periodic. Note that $v(t) = \Sigma_{-\infty, t}$ is the unique solution which is bounded on $\R$. For a more general model, we can obtain uniqueness of the evolution system of measures under the condition that $\sup_{t \in \R} \int_{\R^d} |x|^2 m_t(dx) < \infty$, see Proposition \ref{prop:ergodicity} and Appendix \ref{app:mixing_prop_sec}.
\end{example}

\section{Main results} \label{sec:main_results}

In this section we present our main results and postpone the proofs to the remaining of the paper.
Fix $\theta > 0$, let 
$b : \mathbb{R} \times \mathcal{C}_{\theta} \rightarrow \mathbb{R}^d$ satisfy Assumption~\ref{dissipativity_assumption} and $H : \mathbb{R} \times \mathcal{C}_{\theta} \times \mathbb{R}^d \rightarrow \mathbb{R}$ be a measurable function. For simplicity, we set the time origin to $\tau = 0$ throughout this section and we therefore omit it in the notations. Denote by $X^\xi$ the solution of~\eqref{Functional_SDE} and $\mathcal{X}_t^{\xi}(s) = X_{t + s}^{\xi}$ for $s \leq 0$.

\subsection{Path-dependent ergodic backward SDE}

We study an infinite-horizon backward SDE with unknowns $(Y^{\xi}, Z^{\xi},\Lambda)$ of the form:
\begin{equation} \label{EBackwards SDE}
    \begin{array}{cc}
        \displaystyle Y_t^{\xi} - \Lambda(t) = Y_T^{\xi} - \Lambda(T) + \int_t^T H(s, \mathcal{X}^{\xi}_s, Z_s^{\xi}) \, ds - \int_t^T Z_s^{\xi}\cdot  dW_s \, , & \text{for all } \,  0 \leq t \leq T  \, .
    \end{array}
\end{equation}

\begin{definition}\label{def_ergodic_BSDE}
    We say that $(Y, Z, \Lambda) \in \mathbf{E}(\C_{\theta}) \times {\rm{UC}}$ is a solution of the \textnormal{ergodic backward SDE} if:
    \begin{itemize}
        \item[{\rm (i)}] For all $\xi \in \mathcal{C}_{\theta}$, the triple $(Y^{\xi}, Z^{\xi}, \Lambda)$ solves the infinite-horizon backward SDE \eqref{EBackwards SDE}.

        \item[{\rm (ii)}] There exists a constant $C > 0$ such that:
        \vspace{-6pt}
    \begin{equation}\label{Y_decompsition}
    \begin{array}{cc}
         | Y^\xi_t | \leq C ( 1 + | \mathcal{X}^{\xi}_t |_{\theta}) \,  & \text{for all } \,  t \geq 0 \, .
    \end{array}
    \end{equation}
    \end{itemize}
\end{definition}

This is our extension of the ergodic backward SDE \eqref{Ergodic_BSDE_intro} of \citeauthor{furhman_ergodic} \cite{furhman_ergodic} in the homogeneous case (\textit{i.e.} $b$ and $H$ do not depend on the time). Indeed, if $(Y, Z, \lambda) \in \mathbf{E}(\R^d) \times \mathbb{R}$ satisfies $|Y_t^{x}| \leq C(1 + |X_t^{x}|)$ and solves \eqref{Ergodic_BSDE_intro} for all $x \in \mathbb{R}^d$, then $(Y,Z,\Lambda)$ is a solution of the ergodic backward SDE, where $\Lambda(t) \coloneqq \lambda t$. 

\begin{remark}\label{rmk:smooth_repre}
In fact, if $(Y,Z,\Lambda)\in \mathbf{E}(\C_\theta) \times {\rm{UC}}$ is a solution of the ergodic backward SDE and $\overline{\Lambda}\in {\rm{UC}}$, such that $|\overline{\Lambda}-\Lambda|_\infty < \infty,$ $(\overline{Y},Z,\overline{\Lambda})$ is also a solution of the ergodic backward SDE, where $\overline{Y} = Y+(\overline{\Lambda}-\Lambda)$. Theorem \ref{thm:existence_uniqueness} below shows that this characterize all possible $\Lambda$. Moreover, by Lemma~\ref{lem:mollification}, $\overline{\Lambda}$ can be chosen $C^\infty$ with bounded derivatives. In that case, $\overline{Y}^{\xi}$ is a semimartingale, and we may write
\eqref{EBackwards SDE} in differential form, abusing notation by replacing
$(\overline{Y},Z,\overline{\Lambda})$ with $(Y,Z,\Lambda)$:
\begin{equation*}
\begin{array}{cc}
     \displaystyle d Y_t^{\xi} = d \Lambda(t) - H(t, \mathcal{X}^{\xi}_t, Z_t^{\xi}) \, dt + Z_t^{\xi}  \cdot dW_t \, , & t \ge 0 \ .
\end{array}
\end{equation*}
\end{remark}

Our main well-posedness result requires the following additional conditions.

\begin{assumption} \label{hamiltonian_assumption}{\it
\setcounter{assumptiontag}{0}
$H$ is measurable and there exists a constant $L_H>0$ such that
$$
    |H(t,\xi,0)|\le L_H,
    ~\mbox{and}~~
    |H(t,\xi,z)-H(t,\xi',z')|
            \le L_H\big(|\xi-\xi'|_{\theta}+|z-z'|\big),
~t\in \R,~\xi,\xi'\in\mathcal{C}_\theta, z,z'\in\R^d.
$$}
\end{assumption}

\begin{theorem}[Well-posedness]\label{thm:existence_uniqueness}
Under Assumptions~\ref{dissipativity_assumption} and \ref{hamiltonian_assumption}, there exists a solution $(Y, Z,\Lambda) \in \mathbf{E}(\C_{\theta}) \times \rm{UC}$ of the ergodic backward SDE~\eqref{EBackwards SDE}, and a continuous function $v :\R_+ \times \mathcal{C}_{\theta} \to \mathbb{R}$ such that $Y^{\xi}_t = v(t, \mathcal{X}^{\xi}_t)$ for all $(t,\xi)\in \R_+\times \mathcal{C}_{\theta}$. Finally, if $(\overline Y, \overline Z,\overline\Lambda)$ is another solution, then:

\vspace{-5pt}

\begin{equation}\label{eq:Lambda_unique_up_to_constant}
\sup_{t\ge 0} |\Lambda(t)-\overline\Lambda(t) | < \infty \, .
\end{equation}
In particular, $\Lambda$ can be chosen in $C^\infty$, with all derivatives bounded and $\Lambda(0)=0$.
\end{theorem}

\vspace{2pt}

The proof of Theorem \ref{thm:existence_uniqueness} is given in Section~\ref{sec:ebsde}. In the rest of this paper, we write $\lambda_*(K)$ and $\lambda^*(K)$ as in Definition~\ref{def:values}, where $\Lambda$ is part of a solution to the ergodic backward SDE~\eqref{EBackwards SDE}. Moreover, by \eqref{eq:Lambda_unique_up_to_constant}, these quantities are independent of the choice of $\Lambda$.

\begin{remark}\label{rmk:uniqueness}
Uniqueness of $\Lambda$ only holds in the sense of~\eqref{eq:Lambda_unique_up_to_constant}, in line with Remark~\ref{rmk:smooth_repre}. In the homogeneous Markov case, \textit{i.e.} where there is a canonical representation $\Lambda(t) = \lambda t$, we recover the uniqueness of the ergodic constant $\lambda$, see \citeauthor{Arisawa_Lions} \cite{Arisawa_Lions} and \citeauthor{furhman_ergodic} \cite{furhman_ergodic}.
\end{remark}

We finally present a stability result for the ergodic values of Definition \ref{def:values} with respect to the coefficients $b$ and $H$. Let $(b, H)$ and $(\tilde{b}, \tilde{H})$ be two sets of coefficients satisfying Assumptions \ref{dissipativity_assumption} and \ref{hamiltonian_assumption}. We denote by $(\lambda_*(K), \lambda^*(K))$ and $(\tilde{\lambda}_*(K), \tilde{\lambda}^*(K))$, the ergodic values corresponding to the systems $(b, H)$ and $(\tilde{b},\tilde{H})$ respectively. We prove stability results for the ergodic values with respect to measurable perturbations $\delta b$ and $\delta H$ satisfying:
\begin{equation}\label{eq:delta_bH}
    \begin{array}{ccc}
       \displaystyle \delta b(t)\geq\sup_{\xi\in \mathcal{C}_{\theta}}|b(t,\xi)-\tilde b(t,\xi)|& \text{and} & \displaystyle \delta H(t)\geq \hspace{-5pt} \sup_{\xi\in\C_\theta, \, z\in\R^d}|H(t,\xi,z)-\tilde H(t,\xi,z)|\, .
    \end{array}
\end{equation}

\begin{theorem}\label{thm:stability_values}
Let $(b,H), (\tilde b,\tilde H)$ be coefficients satisfying  Assumptions~\ref{dissipativity_assumption},~\ref{hamiltonian_assumption}, together with \eqref{eq:delta_bH} for some measurable maps $\delta b,\delta H:\R_+\rightarrow\R_+$. Then,
there exists a constant $C>0$, depending only on $\eta_1$, $\eta_2$, $\theta_0$, $\theta$, and $L_H$, such that for all $K\in\K$:
\begin{equation*}\label{eq:stability_b,H}
|\lambda_*(K)-\tilde{\lambda}_*(K)|+|\lambda^*(K)-\tilde{\lambda}^*(K)|\leq \, \limsup_{\rho\downarrow 0}\Big\{ C \,  \Big(\int_0^\infty K_{\rho}(t)(\delta b(t))^2dt\Big)^{\frac{1}{2}} + \int_0^\infty K_{\rho}(t) \delta H(t)dt \Big \} \, .    
\end{equation*}

\end{theorem} 

The proof is reported in Section~\ref{sec:stability}.

\subsection{Path-dependent ergodic optimal control}

To relate the ergodic backward SDE \eqref{EBackwards SDE} to an ergodic optimal control problem, we assume:
$$
H(t,\xi,z)\coloneqq \inf_{a\in A}\{f(t,\xi,a)+a\cdot z\} \,,
$$
where $f :\R \times \mathcal{C}_{\theta} \times A \rightarrow \mathbb{R}$ is the running cost, and $A \subset \mathbb{R}^d$ is the state space for the control process. 

\begin{assumption} \label{control_assumption}
\setcounter{assumptiontag}{0}

    \begin{itemize}
        \item[]

        \refstepcounter{assumptiontag}
        \item[] \hspace{-30pt} (\theassumptiontag) \label{compact_control} \hspace{2pt} {\it $A$ is a compact set of $\R^d$, $f$ is bounded measurable and $H$ is measurable.} 

        \refstepcounter{assumptiontag}
        \item[] \hspace{-30pt} (\theassumptiontag) \label{cost_Lpz} \hspace{1pt} {\it There exists $L_f > 0$ such that $|f(t, \xi,a)-f(t, \xi',a)|\leq L_f \, |\xi-\xi'|_{\theta},~t \in \R  ,  \; \xi, \xi' \in \C_{\theta} , \; a \in A.
        $ }
    \end{itemize}
\end{assumption}

\vspace{5pt}

In particular, under Assumption~\ref{control_assumption}, $H$ satisfies Assumption~\ref{hamiltonian_assumption}. Given $\xi\in \C_{\theta}$ and $K\in\K$, we consider the stochastic optimal control problem defined by the average cost
\begin{equation*}
    \begin{array}{ccc}
        \hspace{-10pt} \displaystyle J^{K}(\xi,\alpha) \coloneqq \E^{\alpha}\int_0^{\infty}  K(t)\,  f(t, \X^{\xi}_t,\alpha_t) \, dt \, ,
    \end{array}
    \end{equation*}
where the control process $\alpha$ lies in the set $\mathcal{A}$ of $\mathbb{F}-$progressively measurable processes $\alpha : \Omega \times\R_+\rightarrow A$. 

\begin{remark}\label{rmk:cost_lip}
(i) The notation $J^{K}$ is slightly abusive: if $\alpha\in\mathcal{A}$, the probability measure $\mathbb{P}^\alpha$ is well defined only on $\mathcal{F}_T$ in general. Thus, $J^{K}$ should be understood as $\int_0^\infty K(t)\,\E^\alpha[f(t,X_t^\xi,\alpha_t)]\,dt$. \\

\vspace{-8pt}

\noindent (ii) Note that given $K\in\K,\rho>0$, $K_\rho=\rho K(\rho\cdot)\in\K$. Thus the classical time-average cost and exponentially discounted cost functionals are recovered as special cases:
$J^{K^0_\rho}(\xi,\alpha)=\rho\E^\alpha\int_0^\frac1\rho f(t,X_t^\xi,\alpha_t)\,dt$, and
$J^{K^{\exp}_\rho}(\xi,\alpha)=\rho\E^\alpha\int_0^\infty  e^{-\rho t}f(t,X_t^\xi,\alpha_t)\,dt$. \\

\vspace{-8pt}

\noindent (iii) Moreover, by Proposition \ref{prop:FSDE_existence} and Assumption~\eqref{cost_Lpz}, for all $\xi,\xi'\in \C_{\theta}$, $\alpha\in\mathcal{A}$ and $K\in\K$,

\vspace{-5pt}

\begin{equation*}
|J^{K_\rho}(\xi,\alpha)-J^{K_\rho}(\xi',\alpha)|\leq C|\xi-\xi'|_\theta\int_0^\infty K_\rho(t)e^{-\theta t}dt\xrightarrow[\rho \downarrow 0]{}0 \, .
\end{equation*}
\end{remark}

We introduce the following notations to characterize the optimal control problem.

\begin{definition}\label{def:quantity}
Let $\xi\in \C_{\theta}$, $\alpha\in\mathcal A$, and $K\in\K$.
We define:
\begin{equation*}
\begin{array}{ccc}
\displaystyle 
\displaystyle \underline{J}(K,\alpha)\coloneqq \liminf_{\rho\downarrow 0} J^{K_\rho}(\xi,\alpha)  & \text{and}
& \displaystyle \overline{J}(K,\alpha)\coloneqq \limsup_{\rho\downarrow 0} J^{K_\rho}(\xi,\alpha) \, , 
\end{array}
\end{equation*}
where these quantities are independent of $\xi$ by Remark \ref{rmk:cost_lip} (iii). We say that $\hat{\alpha}\in\mathcal{A}$ is $K$-optimal if:
\[
\underline{J}(K,\hat{\alpha})=\inf_{\alpha\in\mathcal{A}}\underline{J}(K,\alpha) \qquad \text{and} 
\qquad
\overline{J}(K,\hat{\alpha})=\inf_{\alpha\in\mathcal{A}}\overline{J}(K,\alpha)  \, .
\]
Finally, if $(Y,Z,\Lambda)$ is a solution of the ergodic backward SDE, we define $\hat{\mathcal{A}}$ as the subset of $\mathcal{A}$ such that for all $\alpha \in \hat{\mathcal{A}}$, there exists $\xi \in \C_{\theta}$ and a nonnegative function $\varepsilon\in L^1_{loc}(\R_+)$ satisfying:
\begin{equation*}
    \begin{array}{ccc}
         \displaystyle f (t, \mathcal{X}^{\xi}_t , \, \alpha_t \big) + \alpha_t \cdot Z^{\xi}_t - H (t, \X^\xi_t , Z^{\xi}_t ) \leq \varepsilon(t) \,, \; d\mathbb{P} \otimes dt-\text{a.e.} & \text{and} & \displaystyle \frac1T\int_0^T\varepsilon(t)dt
    \xrightarrow[T \rightarrow \infty]{}0.
    \end{array}
\end{equation*}
\end{definition}

\begin{theorem}[Verification argument]  \label{verif_argument} 
Under Assumptions~\ref{dissipativity_assumption} and~\ref{control_assumption}, let $(Y,Z,\Lambda)$ be a solution of the ergodic backward SDE. Then, for all $K\in\K,\xi\in\C_\theta$, and $\lambda_*(K),\lambda^*(K)$ as in Section~\ref{sec:kernel_averaging}, we have:
\begin{align*}
\inf_{\alpha\in\mathcal{A}}\underline{J}(K,\alpha) = \liminf_{\rho\downarrow0}\inf_{\alpha\in\mathcal{A}}J^{K_\rho}(\xi,\alpha) =\lambda_*(K) \, ,\quad
\inf_{\alpha\in\mathcal{A}}\overline{J}(K,\alpha) = \limsup_{\rho\downarrow0}\inf_{\alpha\in\mathcal{A}}J^{K_\rho}(\xi,\alpha) =\lambda^*(K) \,. 
\label{eq:cost_value}
\end{align*}
Moreover, $\hat{\mathcal{A}} \neq \emptyset$, and any
$\hat{\alpha} \in \hat{\mathcal{A}}$ is $K$-optimal for all $K\in\K$.
\end{theorem}

\begin{proof}
By Remark \ref{rmk:smooth_repre}, we can assume that $\Lambda \in C^\infty$ with bounded derivative, and that $Y$ is a semi-martingale without modifying $Z$. By Measurable Selection Theorem, for any $\xi\in\C_\theta$ and strictly positive continuous function $\varepsilon :\R_+ \to \mathbb{R}$ satisfying $\varepsilon(t) \rightarrow 0$ as $t \rightarrow \infty$, there exists $\alpha \in \mathcal{A}$ such that:
\begin{equation}\label{eq:optimal_control_epsilon}
    H (t, \X^\xi_t , Z^{\xi}_t )
    \geq f (t, \mathcal{X}^{\xi}_t , \, \alpha_t \big) + \alpha_t \cdot Z^{\xi}_t - \varepsilon (t) \, ,
    \qquad d\mathbb P \otimes dt\text{-a.s.}
\end{equation}
Thus $\hat{\mathcal{A}} \neq \emptyset$. Let $\hat{\alpha}\in\hat{\mathcal A}$ and $K\in\mathcal K$; we show that $\hat{\alpha}$ is a $K$-optimal control. 
Fix the corresponding $\xi$ and $\varepsilon(t)$ associated with $\hat{\alpha}$, and let $T,\rho>0$. Notice that, by Remark~\ref{rmk:cost_lip}~(iii), the quantities
$\liminf_{\rho\downarrow0}\inf_{\alpha\in\mathcal A}J^{K_\rho}(\xi,\alpha)$
and
$\limsup_{\rho\downarrow0}\inf_{\alpha\in\mathcal A}J^{K_\rho}(\xi,\alpha)$
are independent of $\xi$. Therefore, it suffices to consider this fixed initial condition $\xi$.  Since $K$ is non-increasing, $K_\rho$ is also non-increasing, thus has bounded variation on $[0,T]$. Applying the Itô formula to $K_\rho Y^\xi$ on $[0,T]$, we obtain:
\begin{align*}
K_\rho( T)Y^\xi_T-K_\rho(0)Y^\xi_0
&=\int_0^T K_\rho(s)\,dY^\xi_s+\int_{(0,T]}Y^\xi_s\,dK_\rho(s) \\
&=\int_0^T K_\rho(s)\bigl(\Lambda'(s)-H(s,\mathcal X^\xi_s,Z^\xi_s)\bigr)\,ds
+\int_0^T K_\rho(s)Z^\xi_s\cdot dW_s
+\int_{(0,T]}Y^\xi_s\,dK_\rho(s) \, .
\end{align*}
Let $\alpha\in\mathcal A$. Since we have $H(s,\mathcal X^\xi_s,Z^\xi_s)\le f(s,\mathcal X^\xi_s,\alpha_s)+\alpha_s\cdot Z^\xi_s$, we deduce by Lemma \ref{lem:girsanov_intZ_is_mg} that:
\begin{align}
K_\rho(0)Y^\xi_0-\E^\alpha[K_\rho(T)Y^\xi_T]-\E^\alpha\int_{(0,T]}Y^\xi_s\,dK_\rho(s)
\le I^{\rho,T}(\xi,\alpha) \, , \label{I^rho_ineq}
\end{align}
where we define $I^{\rho,T}(\xi,\alpha)\coloneqq \E^\alpha\int_0^T K_\rho(s)[f(s,\mathcal X^\xi_s,\alpha_s)-\Lambda'(s)]\,ds$. By Proposition \ref{prop:FSDE_existence} and the growth condition $|Y^{\xi}|\le C(1+|\X^{\xi}|_{\theta})$ in \eqref{Y_decompsition}, we have:
\[
C_{\xi} \coloneqq \sup_{\alpha\in\mathcal A}\sup_{s\ge0}\E^\alpha|Y^\xi_s|<\infty \, .
\]

Thus, $|K_\rho(0)Y^\xi_0|\le C_\xi K_\rho(0)$ and $|\E^\alpha[K_\rho(T)Y^\xi_T]|\le C_\xi K_\rho(0)$ and since $K$ is non-increasing,
\begin{align*}
\Big|\E^\alpha\int_{(0,T]}Y^\xi_s\,dK_\rho(s)\Big|
&\le \E^\alpha\int_{(0,T]}|Y^\xi_s|\,d(-K_\rho(s)) \\
&\le C_\xi \int_{(0,T]}d(-K_\rho(s)) \, = \, C_\xi\bigl(K_\rho(0)-K_\rho(T)\bigr) \, \le \,  C_\xi K_\rho(0) \, .
\end{align*}
Combining with \eqref{I^rho_ineq}, we obtain $-3C_\xi K_\rho(0) \le I^{\rho,T}(\xi,\alpha)$. Similarly, we can write, using that $f(s,\mathcal X^\xi_s,\hat{\alpha}_s)+\hat{\alpha}_s\cdot Z^\xi_s
-H(s,\mathcal X^\xi_s,Z^\xi_s) \le  
\varepsilon(s)$,

\vspace{0pt}

\begin{equation*}\label{eq:optimal_alpha_mono}
-3C_\xi K_\rho(0)
\, \le \, 
I^{\rho,T}(\xi,\hat{\alpha})
\, \le \, 
3C_\xi K_\rho(0)+\int_0^T K_\rho(s)\varepsilon(s)\,ds \, .
\end{equation*}
Letting $T\to\infty$, we obtain the following inequalities:
\begin{equation*}
\begin{array}{l}
\displaystyle
-3C_\xi K_\rho(0)
\le
J^{K_\rho}(\xi,\alpha)-\int_0^\infty K_\rho( s)\Lambda'(s)\,ds \, , \\[10pt]
\displaystyle
-3C_\xi K_\rho(0)
\le
J^{K_\rho}(\xi,\hat{\alpha})-\int_0^\infty  K_\rho( s)\Lambda'(s)\,ds
\le
3C_\xi K_\rho(0)+\int_0^\infty K_\rho(s)\varepsilon(s)\,ds \, .
\end{array}
\end{equation*}
Notice that $K_\rho(0)=\rho K(0)\xrightarrow[\rho\downarrow0]{}0$ and
$\frac{1}{T}\int_0^T\varepsilon(t)\,dt\xrightarrow[T\to\infty]{}0$.
Then, applying Lemma~\ref{lemma:upper-sandwich} to $\varepsilon$  implies that
$\int_0^\infty K_\rho(s)\varepsilon(s)\,ds
\xrightarrow[\rho\downarrow 0]{}0$. We deduce the required identities by letting $\rho \downarrow 0$.
\end{proof}

\begin{example}
All inequalities of the ergodic values in Remark \ref{inequalities_ergodic_values_rmk} are possible. Let us illustrate this in the inhomogeneous Markov case. Take $A = \{ 0 \}$, $f(t, x, a) = c(t) e^{-x^2 / 2}$ for $(t, x, a) \in \R_+\times \mathbb{R} \times A$, then $H(t,x,z)=f(t,x,0)=c(t)e^{-x^2/2}$. Keeping the notations of Example \ref{ex:OU_evolution_system}, let $X^{0, x}$ be the solution of the inhomogeneous Ornstein-Uhlenbeck SDE \eqref{eq:OU_non_homog} with initial conditions $(0, x)$, the law of $X^{0,x}_t$ is the Gaussian $\mathcal{N}(x e^{- S_{0,t}}, \Sigma_{0, t})$. Thus $(0,0,\Lambda)$ is a solution of ergodic backward SDE, where:
\begin{align}
     \Lambda'(t)=
     \E f(t, X_t, \alpha_t) \,  
    = c(t) \, \mathbb{E} [e^{- (X_t^{0, x})^2 / 2}] \,= \frac{c(t) \, \Sigma_{0, t}}{\sqrt{1 + \Sigma_{0, t}}} \, \exp \Big[ \frac{e^{-S_{0,t}}}{1 + \Sigma_{0, t}} \big(1 - \frac{e^{-S_{0,t}}}{1 + \Sigma_{0, t}} \big) \Big] \, . \nonumber
\end{align}
We may choose the functions $c(t)$ and $\eta(t)$ so that the ergodic values of Definition \ref{def:values} are equal or lie in one of the cases of Remark \ref{inequalities_ergodic_values_rmk}.
\end{example}

\vspace{5pt}

Similarly to Theorem~\ref{thm:stability_values}, we can deduce a stability result for the control problem. Let \(f,\tilde{f}\) be cost functions satisfying Assumption~\ref{control_assumption}, and we define $\tilde{J}^{K_\rho}(\xi,\alpha) \coloneqq \E^{\alpha}\int_0^{\infty} K_{\rho}(t)\,  \tilde{f}(t, \tilde{\X}^{\xi}_t,\alpha_t) \, dt$, where $\tilde{\X}^{\xi}$ is the path process associated with $\tilde{b}$.
Similarly to \eqref{eq:delta_bH}, we consider $\delta b,\delta f$ such that
\begin{equation}\label{eq:delta_f}
\begin{array}{ccc}
   \displaystyle \delta b(t)\geq\sup_{\xi\in \mathcal{C}_{\theta}}|b(t,\xi)-\tilde b(t,\xi)\bigr| & \text{and} & \displaystyle  \, \displaystyle \delta f(t)\geq \hspace{-7pt} \sup_{\xi\in \C_\theta, \,  a\in A} \, \big|f(t,\xi,a)-\tilde{f}(t,\xi,a)\big|.
\end{array}
\end{equation}

\begin{theorem}
[Stability of cost]\label{thm:stability_control}
Under Assumptions~\ref{dissipativity_assumption} and~\ref{control_assumption}, there exists a constant $C>0$ depending only on $\eta_1$, $\eta_2$, $\theta_0$, $\theta$, and $L_f$ such that if $\delta b,\delta f:\R_+\rightarrow\R_+$ are measurable and satisfy~\eqref{eq:delta_f}, then for $\xi\in \C_{\theta}$, $\alpha\in\mathcal A$, $K\in\K$ and $\rho>0$:
\begin{equation*}\label{eq:stability_control}
\big|J^{K_\rho}(\xi,\alpha)-\tilde{J}^{K_\rho}(\xi,\alpha)\big|\leq  \,  C \,  \Big(\int_0^\infty K_{\rho}(t)(\delta b(t))^2dt\Big)^{\frac{1}{2}} + \int_0^\infty K_{\rho}(t) \delta f(t)dt \, .  
\end{equation*}
In particular, if   $\lim_{\rho\downarrow 0}\int_0^\infty K_{\rho}(t)(\delta b(t))^2dt =\lim_{\rho\downarrow 0}\int_0^\infty K_{\rho}(t)\delta f(t)dt =0$, then $\lambda_*(K)=\tilde{\lambda}_*
(K)$, $\lambda^*(K)=\tilde{\lambda}^*
(K)$, and any $K$-optimal control for the coefficients $(\tilde{b}, \tilde{f})$ is also $K$-optimal for the coefficients $(b, f)$.
\end{theorem}

\vspace{5pt}

The proof is postponed in Section~\ref{sec:stability}, and does not leverage the arguments of Theorem \ref{thm:stability_values}.

\vspace{5pt}

\begin{remark}
The stability results of Theorems~\ref{thm:stability_values} and~\ref{thm:stability_control} allow us to reduce the complexity of the optimal ergodic control in some cases. For example, consider the time-inhomogeneous Markovian case, \textit{i.e.}, when $b$ and $f$ depend on $\xi$ only through $\xi_0$, rather than on the whole path. Assume that $(b(t,\cdot),f(t,\cdot))$ converges uniformly as $t\to\infty$ to a pair $(b_\infty(\cdot),f_\infty(\cdot))$ satisfying Assumptions \ref{dissipativity_assumption} and \ref{control_assumption}. By \citeauthor{furhman_ergodic} \cite{furhman_ergodic}, the system $(b_\infty,f_\infty)$ has a unique ergodic value $\lambda_{\infty}$, and we then obtain $\lambda_*(K) = \lambda^*(K) = \lambda_{\infty}$ for all $K \in \K$. Moreover, any $K$-optimal control for the system $(b_\infty,f_\infty)$ is $K$-optimal for $(b,f)$.
\end{remark}

\subsection{Further characterization results}

In Section \ref{sec:special_cases}, we shall report further results in the following special settings:
\begin{itemize}
    \item When the system $(b,H)$ is time-periodic, we prove that the ergodic values coincide:
$\lambda_*(K)=\lambda^*(K)=\lambda$ for every kernel $K\in\mathcal K$, where $\lambda\in\mathbb R$ is independent of $K$. This is consistent with the work of \citeauthor{cohen_periodic_EBSDE} \cite{cohen_periodic_EBSDE}.

    \item In the time-inhomogeneous Markov setting, \textit{i.e.} $b$ and $H$ depend on the time and the current value of the state, we show the uniqueness of the solution $(Y, Z, \Lambda)$ in a suitable sense. 

    \item When $b$ and $H$ depend on a finite window of the past, we obtain, under additional assumptions, an explicit representation of the function $\Lambda$, extending the corresponding result of \citeauthor{furhman_ergodic} \cite[Corollary~5.9]{furhman_ergodic}.
\end{itemize}

\section{Regularity of the forward flow}\label{sec:forward equation}

In this section, Assumption~\ref{dissipativity_assumption} is assumed to hold. Let $\tau\in\R$ and $\xi \in \mathcal{C}_{\theta}$. Recall that we denote by $X^{\tau,\xi}$ the unique solution of the forward equation \eqref{Functional_SDE} with initial data $(\tau,\xi)$, and that the path process $(\mathcal{X}_t^{\tau, \xi})_{t \geq \tau}$ is defined by $\mathcal{X}_t^{\tau, \xi}(s) = X_{t + s}^{\tau, \xi}$ for $s \leq 0$. In the proofs below, we only treat the case $\tau =0$ for simplicity, so we write $W$ instead of $W^{\tau}$. We identify some key properties of $\X^{\tau,\xi}_t$, namely stability results with respect to the initial data $\xi$ and time $t$. The following technical result will be used extensively below.

\begin{lemma}\label{lem:exp_decreasing}
Let $(U_t)_{t\in \R}$ be a nonnegative It\^o process with the decomposition $dU_t=\nu_t\,dt+\sigma_t\cdot dW_t$ for $t \geq 0$, where $\nu$ and $\sigma$ are $\mathbb{F}$-progressively measurable processes. For $t \ge 0$, define:
\begin{equation}\label{def_V_t}
\begin{array}{cc}
     \displaystyle V_t
:=\nu_t+\gamma_1U_t-\gamma_2\int_{-\infty}^{0} e^{2\theta_0 s}\,U_{t+s}\,ds \, , & \displaystyle \text{for some } \, \gamma_2 > 0 \, \text{ and } \, \gamma_1\geq2\theta+\frac{\gamma_2}{2(\theta_0-\theta)} \, .
\end{array}
\end{equation}
Then, one has, for all $t \geq 0$:
\begin{equation*}\label{general_estimate}
e^{2\theta t}U_t \leq
U_0
+\frac{\gamma_2}{2(\theta_0-\theta)}\int_{-\infty}^{0} e^{2\theta_0 s} \, U_s\,ds
+\int_{0}^{t} e^{2\theta s}\big( V_s\,ds
+ \sigma_s\cdot dW_s \big) \, .
\end{equation*}
\end{lemma}

\begin{proof}
First note that, by Tonelli's Theorem, we have: 

\begin{align*}
    \int_0^t e^{2\theta s} \int_{-\infty}^s e^{2\theta_0(r-s)} U_r \,dr\,ds  &=  \int_{-\infty}^t  e^{2\theta_0 r} U_r  \int_{0\vee r}^{t} e^{2(\theta-\theta_0)s}\,ds\,dr \\
    & \leq \frac{1}{2(\theta_0-\theta)} \Big(\int_{-\infty}^0 e^{2\theta_0 r} U_r \,dr
    + \int_{0}^t e^{2\theta r} U_r \,dr \Big).
\end{align*}
Then, applying It\^o's formula to $e^{2\theta t}U_t$, we obtain:
\begin{align*}
    e^{2\theta t}U_t &= U_0+\int_0^t e^{2\theta s} (2\theta U_s+\nu_s)\,ds+\int_0^t e^{2\theta s}\sigma_s\,dW_s \\ 
    &\leq U_0+\int_0^t e^{2\theta s}\Bigl((2\theta-\gamma_1)U_s+\gamma_2\int_{-\infty}^s e^{2\theta_0(r-s)} U_r \,dr+V_s\Bigr)\,ds
      +\int_0^t e^{2\theta s}\sigma_s\,dW_s \\
    &\leq  U_0+\frac{\gamma_2}{2(\theta_0-\theta)}\int_{-\infty}^0 e^{2\theta_0 s} U_s\,ds
      +\int_0^t e^{2\theta s}\Bigl( \big(2\theta-\gamma_1+\frac{\gamma_2}{2(\theta_0-\theta)}\big)U_s + V_s \Bigr) \,ds
      +\int_0^t e^{2\theta s} \sigma_s\,dW_s \, . \nonumber
\end{align*}
We conclude by the assumption on $\gamma_1$ and $\gamma_2$.
\end{proof}

%\vspace{5pt}

Our first result concerns the stability of the path process $\mathcal{X}^{\xi,\tau}_t$ with respect to the initial data $\xi$. This is mostly due to the dissipativity condition. The proof can also be found in \citeauthor{Infinite_delay_nonhomogeneous} \cite{Infinite_delay_nonhomogeneous} and is based on standard arguments, see \citeauthor{Ergodic_FSDE} \cite[Theorem~1.5]{Ergodic_FSDE} and \citeauthor{FSDE_infinite_delay} \cite[Theorem~3.2]{FSDE_infinite_delay}.

\begin{proposition}[Stability in initial data]\label{prop:forward_lip}
There exists a constant $C>0$ depending only on $\eta_1,\eta_2,\theta_0,\theta$ such that for all $t \geq 0$ and $\xi,\xi'\in \mathcal{C}_{\theta}$:
\begin{equation}\label{stability_prop_ineq}
\begin{array}{cc}
     \displaystyle \displaystyle | \mathcal{X}^{\tau, \xi}_t - \mathcal{X}^{\tau, \xi'}_t|_{\theta} \leq Ce^{-\theta (t-\tau)} |\xi-\xi'|_{\theta} \, , & \mathbb{P}-\text{a.s.}
\end{array}
\end{equation}
\end{proposition}

\begin{proof}
For notational simplicity, we only treat the case $\tau = 0$ and omit the initial time from the notation below. Applying It\^o's formula to the process $U_t = |X^\xi_t-X^{\xi'}_t|^2$ for $t\geq0$, it follows that:
\begin{equation*}
    \begin{array}{cc}
        \displaystyle U_t = U_0 + \int_0^t \nu_s \, ds \, , & \text{where } \, \nu_t \coloneqq 2 (X^\xi_t-X^{\xi'}_t) \cdot \big(b(t, \mathcal{X}^\xi_t)-b(t, \mathcal{X}^{\xi'}_t)\big) \, .
    \end{array}
\end{equation*}
Then, for $\gamma_1 \coloneqq \eta_1$ and $\gamma_2 \coloneqq \eta_2$, by Assumption~\eqref{dissipativity_condition}, the process $V$ defined in \eqref{def_V_t} satisfies $V_t \leq 0$. Applying Lemma \ref{lem:exp_decreasing} and the fact that $U_s = |\xi_s-\xi'_s|^2 \leq e^{- 2\theta s} |\xi - \xi'|_{\theta}^2$ for all $s \leq 0$, we obtain:
\begin{align}
e^{2\theta t} U_t \leq |\xi-\xi'|_{\theta}^2 \, \Big(1 + \frac{\gamma_2}{2(\theta_0-\theta)} \int_{-\infty}^0 e^{(2\theta_0-2\theta) s}\,ds \Big) = |\xi-\xi'|_{\theta}^2 \,  \Big(1+\frac{\gamma_2}{4(\theta_0-\theta)^2} \Big) \nonumber \, .
\end{align}
Then \eqref{stability_prop_ineq} is a consequence of the following inequality:
\begin{equation*}
\begin{array}{cc}
     \displaystyle |\mathcal{X}^\xi_t-\mathcal{X}^{\xi'}_t|^2_{\theta} \,  = \sup_{s\in(-\infty, t]} e^{2\theta(s-t)}U_s \, \leq \, e^{-2\theta t} |\xi-\xi'|^2_{\theta} \Big(1+\frac{\gamma_2}{4(\theta_0-\theta)^2}\Big) \, , & t \geq 0 \, .
\end{array}
\end{equation*}
\end{proof}

%\vspace{5pt}

We now establish a quantitative bound on the time-continuity of $\mathcal{X}^{\tau,\xi}$, extending the corresponding estimate for solutions of Markovian stochastic differential equations to the path-dependent setting.

\begin{definition}
Let $\xi\in \mathcal{C}_{\theta}$ and $\delta>0$. We call the modulus of continuity of $\xi$ on $(-\infty,0]$ of decayed dependence $\theta$ the real number defined as:
\[
\omega_{\theta,\xi}(\delta)
:=\sup\Big\{\,e^{\theta t}|\xi_s-\xi_{s'}|:\ t\leq 0 \, \text{ and }\, s, s' \in[t,(t+\delta)\wedge 0]\,\Big\} \, .
\]
\end{definition}

\begin{lemma}\label{lem:property_mc}
For all $\xi\in \mathcal{C}_{\theta}$, we have $\displaystyle \,  \sup_{\delta > 0} \, \omega_{\theta,\xi}(\delta)\leq 2 |\xi |_{\theta}$ and $\displaystyle \lim_{\delta \downarrow 0}\omega_{\theta,\xi}(\delta)=0$.
\end{lemma}
\begin{proof}
The first statement follows from the fact that if $t \leq 0$ and $s, s' \in [t,(t+\delta)\wedge 0]$, then $e^{\theta t} |\xi_s - \xi_{s'}| \leq 2  |\xi |_{\theta}$. For the second statement, first assume that $e^{\theta t} \xi_t \rightarrow 0$ as $t \rightarrow - \infty$. Fix $\varepsilon > 0$, then there exists $t_{\varepsilon} < 0$ such that $e^{\theta t } |\xi_t| \leq \varepsilon$ for all $t \leq t_{\varepsilon} $. If $t\leq t_{\varepsilon}-\delta$ and $s, s' \in [t,t+\delta]$, then one has: $e^{\theta t}|\xi_s - \xi_{s'}| \leq 2 \varepsilon$. Therefore, we obtain:
\begin{equation*}
    \omega_{\theta,\xi} (\delta) \leq \max \Bigg\{ 2 \varepsilon \, , \, \sup_{\substack{
         s, s' \in [t_{\varepsilon} - \delta, 0] \\
         |s-s' | \leq \delta }} \big| \xi_s - \xi_{s'} \big| \Bigg\} \, ,
\end{equation*}
and the right-hand side converges to $2 \varepsilon$ when $\delta \downarrow 0$. We conclude by taking $\varepsilon \rightarrow 0$. In the general case, there exists $c \in \mathbb{R}$ such that $e^{\theta t} \xi_t \rightarrow c$ as $t \rightarrow - \infty$. Denote $\psi_t \coloneqq \xi_t - c \, e^{-\theta t}$. Then $\psi \in \mathcal{C}_{\theta}$ and $e^{\theta t} \psi_t \rightarrow 0$ when $t \rightarrow -\infty$. Moreover, we have for $t \leq 0$ and $s,s' \in [t, (t+\delta)\wedge 0]$,
\begin{equation*}
    e^{\theta t}\big|\xi_s - \xi_{s'} \big| \leq e^{\theta t} \big| \psi_s - \psi_{s'} \big| + |c| \big( 1 - e^{\theta \delta} \big) \, ,
\end{equation*}
which implies that $\omega_{\theta,\xi}(\delta)\le \omega_{\theta,\psi}(\delta)+|c|\bigl(1-e^{\theta\delta}\bigr)\to 0$ as $\delta \downarrow 0$.
\end{proof}

\begin{proposition}[Stability and boundedness in time]\label{prop:forward_holder}
Given $B>0$, there exists a constant $C$ only depending on $\eta_1$, $\eta_2$, $\theta_0$, $\theta$, $L_b$, and $B$, such that for any $\tau\leq t\leq T$ and bounded process $\beta \in \mathcal{P}_B$:
\begin{equation}\label{eq:forward_time_stability}
    \begin{array}{cc}
         \displaystyle \mathbb{E}^\beta |\mathcal{X}^{\tau, \xi}_t-\xi|^2_{\theta} \leq C \big[ (1+|\xi|^2_{\theta}) \, ((t-\tau)\wedge \frac{1}{\theta}) + \omega_{\xi,\theta}^2(t-\tau) \big].
    \end{array}
\end{equation}
\end{proposition}

\noindent \textit{Proof.} Let $T\ge 0$, $t\in[0,T]$ and $\beta \in \mathcal{P}_B$. Once again we treat only the case $\tau = 0$ and omit it in the notations. We want to estimate the following:
\begin{equation*}
    \mathbb{E}^\beta |\mathcal{X}^\xi_t -\xi|^2_{\theta} \, = \mathbb{E}^{\beta} \sup_{s \leq 0} \, e^{2\theta s}|X^\xi_{t+s}-\xi_{s}|^2 \, .
\end{equation*}
Let $s \leq 0$. If $t+s \leq 0$, then $ e^{\theta s}|X^\xi_{t+s}-\xi_{s}|= e^{\theta s}|\xi_{t+s}-\xi_{s}|\leq \omega_{\theta,\xi}(t)$. In the rest of the proof we assume $t+s > 0$. As $|s|=-s<t$, so that $\omega_{\theta,\xi}(|s|)\le \omega_{\theta,\xi}(t)$, the triangle inequality yields 
\begin{align}
    e^{\theta s}|X^\xi_{t+s}-\xi_{s}| \, \leq \,  e^{\theta s}|X^\xi_{t+s}-\xi_0|+ e^{\theta s}|\xi_0-\xi_s| \, \leq \, e^{\theta s}|X^\xi_{t+s}-\xi_0|+\omega_{\theta,\xi}(t). \nonumber
\end{align}
Let $r \coloneqq t+s$ and define the process $ U_r := |X_r^\xi-\xi_0|^2 $. Then $U_0 = 0$ and applying Itô's formula under the distribution $\mathbb{P}^{\beta}$, we obtain $d U_r = \nu_r \, dr + 2(X^\xi_{r}-\xi_0) \cdot dW_{r}^{\beta}$, where $\nu_{r} \coloneqq  2(X^\xi_{r}-\xi_0) \cdot (b(r, \mathcal{X}^\xi_r) + \beta_{r}) + 1 $. First, by the dissipativity condition, for any $\varepsilon>0$,
\begin{align}\label{eq:basic_inequality}
\nu_r-1
&= 2(X^\xi_r-\xi_0) \cdot \big(b(r,\mathcal{X}^{\xi}_r)-b(r,\xi)\big)
   + 2(X^\xi_r-\xi_0)\cdot \big(b(r,\xi)+\beta_r\big) \nonumber\\
&\le -\eta_1|X^\xi_r-\xi_0|^2
   + \eta_2\int_{-\infty}^0 e^{2\theta_0u}|X^\xi_{r+u}-\xi_u|^2\,du
   + 2(X^\xi_r-\xi_0)\cdot \big(b(r,\xi)+\beta_r\big) \\
&\le -(\eta_1-\varepsilon)|X^\xi_r-\xi_0|^2
   + \eta_2\int_{-\infty}^0 e^{2\theta_0u}|X^\xi_{r+u}-\xi_u|^2\,du \nonumber+ \frac{2}{\varepsilon}(K^2+B^2)\big(1+|\xi|_\theta\big)^2,
\end{align}
by the Young inequality and the linear growth condition \eqref{Lpz_condition_lip_growth_condition}. Applying Young and the triangle inequalities, there exists a positive constant $\tilde{C} = \tilde{C}(\eta_1, \eta_2, \theta_0, \theta, \varepsilon)$ such that: 
\begin{align}
    \int_{-\infty}^0 e^{2\theta_0 u}|X^\xi_{r+u}-\xi_{u}|^2 \, du &= \int_{-\infty}^{-r} e^{2\theta_0 {u}}|\xi_{r+u}-\xi_{u}|^2 \, du + \int_{-r}^0 e^{2\theta_0 u}|X^\xi_{r+u}-\xi_{u}|^2 \, du \nonumber \\
    & \leq \tilde{C} |\xi|^2_{\theta} + (1+\varepsilon) \int_{-r}^0 e^{2\theta_0 u}|X^\xi_{r+u}-\xi_0|^2 \, du \, . \label{ineq_int_X_stability_times_prop}
\end{align}
Combining the two inequalities \eqref{eq:basic_inequality} and \eqref{ineq_int_X_stability_times_prop}, we can write:
\begin{equation*}\label{eq:estimate_drift}
\nu_r \leq -(\eta_1-\varepsilon)|X^\xi_r-\xi_0|^2 + \eta_2(1+\varepsilon) \int_{-r}^0 e^{2\theta_0 s}|X^\xi_{r+s}-\xi_0|^2ds + \overline{C} (1+|\xi|^2_{\theta}) \, ,  
\end{equation*}
where $\overline{C} = \frac{4}{\varepsilon}(K^2 + B^2) + \eta_2 \tilde{C}+1$. Note that the process $V$ defined in \eqref{def_V_t} satisfies $V_r \leq \overline{C} (1 + |\xi|^2_{\theta})$ if $\gamma_1 \coloneqq  \eta_1 -\varepsilon$ and $\gamma_2 \coloneqq (1+\varepsilon) \eta_2 $. Since $\eta_1-2\theta+\frac{\eta_2}{2(\theta_0-\theta)} > 0$, we can choose an $\varepsilon>0$ small enough such that $\gamma_1\geq2\theta+\frac{\gamma_2}{2(\theta_0-\theta)}$. Applying Lemma~\ref{lem:exp_decreasing}, we have:
\begin{align}\label{ineq_final_stability_time}
    e^{2\theta r} U_r \, \leq \,  \overline{C}(1+|\xi|^2_{\theta}) \frac{e^{2\theta r}-1}{2\theta} + 2\int_0^r e^{2\theta u}(X^\xi_{u}-\xi_0) \cdot dW_{u}^{\beta}
     \, .
\end{align}
Define the stopping time $T_n \coloneqq \inf \{ u \geq 0 : |X^\xi_{u}-\xi_0|^2 \geq n \}$. Then by Fatou's Lemma:
\begin{align}
    \mathbb{E}^{\beta} |X^\xi_r-\xi_0|^2 \, \leq \,  \liminf_{n\rightarrow \infty} \, \mathbb{E}^{\beta} |X^\xi_{r \wedge T_n}-\xi_0|^2  \, \leq \, \overline{C}(1+|\xi|^2_{\theta}) \frac{1-e^{-2\theta r}}{2\theta} \, \leq \,  \overline{C}(1+|\xi|^2_{\theta}) (r\wedge\frac{1}{\theta}) \, . \nonumber
\end{align}
Finally, by the Burkhölder-Davis-Gundy inequality, 
\begin{align}
\mathbb{E}^{\beta} \sup_{r\in[0,t]} \Big|\int_0^r e^{2\theta u}(X^\xi_{u}-\xi_0) \cdot dW_{u}^{\beta} \Big| 
& \leq  \Big(\mathbb{E}^{\beta} \int_0^t e^{4\theta u} |X^\xi_{u}-\xi_0|^2 \, du \Big)^{1/2} \nonumber \\ 
&\leq \sqrt{\overline{C}(1+|\xi|^2_{\theta})} \,  \Big(\int_0^t e^{4\theta u} \Big( u \wedge\frac{1}{\theta} \Big) \, d u\Big)^{1/2} 
\nonumber \\ &
\leq \sqrt{\overline{C}(1+|\xi|^2_{\theta})} \, e^{2\theta t} \Big( t\wedge\frac{1}{\theta} \Big) \, . \label{BDG_ineq_stability_time} \nonumber
\end{align}
Thus by~\eqref{ineq_final_stability_time} and  $\sup_{r\in[0,t]} \frac{e^{2\theta (r-t)}-1}{2\theta}\leq t\wedge\frac{1}{\theta}$, we deduce the following inequality, which implies \eqref{eq:forward_time_stability}:
\begin{equation*}
    \mathbb{E}^{\beta} \sup_{r \in [0, t]} e^{2\theta(r-t)}|X^\xi_{r}-\xi_0|^2 \, \leq \,  \overline{C}(1+|\xi|^2_{\theta}) (t\wedge\frac{1}{\theta})+2\sqrt{\overline{C}(1+|\xi|^2_{\theta})} \, \Big( t\wedge\frac{1}{\theta} \Big) \, . 
\end{equation*}

\vspace{-5mm}
\qed

\section{Well-posedness of the ergodic backward SDE}\label{sec:ebsde}
Throughout this section, we work under Assumptions~\ref{dissipativity_assumption}
and~\ref{hamiltonian_assumption}. Unless otherwise stated,
$\tau\le T$ are real numbers. Under this convention, we study
the flow versions of the finite-horizon backward SDE~\eqref{BSDE_finite_horizon}
and the ergodic backward SDE~\eqref{eq:ebsde_flow_path}, and prove one of our
main results, Theorem~\ref{thm:existence_uniqueness}.

\subsection{Finite-horizon backward SDE}

 We consider the finite-horizon backward SDE with coefficient $H$ and data $(\tau, T,\xi)$:
\begin{equation}\label{BSDE_finite_horizon}
    \begin{array}{ccc}
       \displaystyle dY_t^{\tau,T,\xi}
       = -H\big(t, \mathcal{X}^{\tau,\xi}_t, Z_t^{\tau,T,\xi}\big)\,dt
       + Z_t^{\tau,T,\xi} \cdot dW_t^{\tau},
       \quad \tau\le t\le T,
       & \text{and} &
       Y^{\tau,T,\xi}_T = 0.
    \end{array}
\end{equation}

The next results focus on the continuity of $Y^{\tau, T ,\xi}$ in the initial data $\xi$ and in time.

\begin{proposition} \label{prop:Markov_rep}
    There exists a unique solution $(Y^{\tau,T,\xi}, Z^{\tau,T,\xi})$ to the backward SDE \eqref{BSDE_finite_horizon} and:
    \begin{equation*}
    \begin{array}{cc}
         \displaystyle (Y_t^{\tau, T,\xi}, Z_t^{\tau, T, \xi}) = (v^{T}, \zeta^T)(t, \mathcal{X}^{\tau, \xi}_t) \, , & \text{for all } \,  \tau\leq t\leq T \, ,
    \end{array}
    \end{equation*}
    for some measurable functions $v^{T} :(-\infty,T]\times \mathcal{C}_{\theta} \rightarrow \R$ and $\zeta^T : (-\infty,T] \times \mathcal{C}_{\theta} \rightarrow \mathbb{R}^d$. Moreover, the functions $v^T$ are Lipschitz continuous in $\xi$ and uniformly continuous in $t$ (uniformly in $T$), \textit{i.e.}, there exists a constant $C>0$, depending on $\eta_1$, $\eta_2$, $\theta_0$, $\theta$, and $L_H$, such that for all $t,t'\leq T$, and $\xi,\xi'\in \mathcal{C}_{\theta}$,
    \begin{eqnarray}
        \big|v^{T}(t,\xi)-v^{T}(t,\xi')\big|
        &\le&
        C\,|\xi-\xi'|_{\theta},
        \label{eq:backward_lpz}
    \\
        \big|v^T(t',\xi)-v^T(t,\xi)\big|
        &\le&
        C \big[\,(1+|\xi|_{\theta})\,|t'-t|^{1/2}+|t-t'|+\omega_{\theta,\xi}(|t'-t|) \big].
        \label{eq:backward_holder}
    \end{eqnarray}
\end{proposition}

\begin{proof} Throughout the proof, $C$ denotes a positive constant which may change from line to line.
    The well-posedness of \eqref{BSDE_finite_horizon} is standard, see \textit{e.g.} the book of \citeauthor{zhang_BSDE} \cite[Theorem~4.3.1]{zhang_BSDE}. The Markov representation is a direct extension of \cite[Theorem~5.1.3]{zhang_BSDE} and follows from the fact that $(\mathcal{X}^{\tau, \xi}_t)_{t \geq \tau}$ is a Markov process (see Proposition \ref{prop:FSDE_existence}). We start by showing the Lipschitz continuity \eqref{eq:backward_lpz}. First, we can write by Girsanov theorem:
\begin{align}
    v^T(t,\xi) - v^T(t,\xi')
    = \int_t^T \Big\{ H\big(s, \mathcal{X}^{t,\xi}_s, Z_s^{t,T,\xi}\big) - H\big(s, \mathcal{X}^{t,\xi'}_s, Z_s^{t,T,\xi}\big) \Big\}\,ds \nonumber - \int_t^T \big(Z_s^{t,T,\xi} - Z_s^{t,T,\xi'}\big)\cdot dW_s^{t, \beta} \, ,
\end{align}
where
$
\beta_s := \frac{H(s, \mathcal{X}^{t,\xi'}_s, Z_s^{t,T,\xi}) - H(s, \mathcal{X}^{t,\xi'}_s, Z_s^{t,T,\xi'})}{|Z_s^{t,T,\xi}-Z_s^{t,T,\xi'}|^2}
\big(Z_s^{t,T,\xi}-Z_s^{t,T,\xi'}\big)\mathbf{1}_{\{Z_s^{t,T,\xi} \neq Z_s^{t,T,\xi'}\}}
$
is a bounded process by Assumption \ref{hamiltonian_assumption}. By Lemma \ref{lem:girsanov_intZ_is_mg}, the triangle inequality and Proposition~\ref{prop:forward_lip}, we obtain:
\begin{align}
    \big|v^T(t,\xi) - v^T(t, \xi') \big| &\leq \mathbb{E}^{\beta} \int_t^T \big|  H \big(s, \mathcal{X}^{t, \xi}_s, Z_s^{t, T,\xi}\big) - H\big(s, \mathcal{X}^{t, \xi'}_s,Z_s^{t, T,\xi}\big) \big| \, ds \nonumber \\
    & \leq L_H \, \mathbb{E}^{\beta} \int_t^T \big| \mathcal{X}^{t, \xi}_s - \mathcal{X}^{t, \xi'}_s \big|_{\theta} \, ds \, \leq \,  \frac{L_H C}{\theta} | \xi - \xi'|_{\theta} \, . \nonumber
\end{align} We now prove the uniform continuity in time \eqref{eq:backward_holder}. For $t \leq t'\leq T$, we can decompose:
\begin{align}
    v^T(t', \xi) - v^T(t, \xi) - \big( v^T(t', \xi) - v^T(t', \mathcal{X}^{t, \xi}_{t'}) \big) &= Y_{t'}^{t, T, \xi} - Y_{t}^{t, T, \xi}  \nonumber \\
    &= - \int_t^{t'} H \big(s, \mathcal{X}^{t, \xi}_s, Z_s^{t, T, \xi} \big) \, ds + \int_t^{t'} Z_s^{t, T, \xi} \, dW^{t}_s \nonumber \\
    &= - \int_t^{t'}  H\big(s, \mathcal{X}^{t,\xi}_s, 0 \big) \, ds + \int_t^{t'} Z_s^{t,T, \xi} \cdot dW^{t, \beta}_s \nonumber \, ,
\end{align}
where $\beta_s := \frac{H(s, \mathcal{X}^{t, \xi}_s,Z_s^{t, T,\xi})-H(s, \mathcal{X}^{t, \xi}_s,0)}{|Z^{t, T,\xi}_s |^2} Z^{t, T,\xi}_s \mathbf{1}_{\{Z^{t, T,\xi}_s \neq 0\}}$ is bounded by Assumption \ref{hamiltonian_assumption}. By Lemma \ref{lem:girsanov_intZ_is_mg} and Assumption~\ref{hamiltonian_assumption}, we obtain the inequality $\mathbb{E}^{t,\beta} \big|Y_{t'}^{t, T, \xi} - Y_t^{t, T, \xi} \big| \leq L_H|t'-t| \, $. Moreover, by \eqref{eq:backward_lpz} and Proposition~\ref{prop:forward_holder}, we have:
\begin{align}
    \big|\E ^{t, \beta} \big[v^T(t',\xi) - v^T(t', \X^{t,\xi}_{t'}) \big] \big| \, \leq \,  C \, \E^{t,\beta} \big|\X^{t,\xi}_{t'}-\xi \big|_{\theta} \, \leq \, C \big( (1+|\xi|_{\theta})\sqrt{t'-t}+\omega_{\theta,\xi}(|t'-t|) \big) \nonumber \, .
\end{align}
By the triangle inequality, we obtain:
\begin{equation*}
    |v^T(t',\xi) - v^T(t, \xi)|\leq C \big((1+|\xi|_{\theta})\sqrt{t'-t}+\omega_{\theta,\xi}(|t'-t|) \big)+ L_H|t'-t|\,  \, . 
\end{equation*}

\vspace{-20pt}

\end{proof}

\begin{remark}
    This argument for the Lipschitz continuity \eqref{eq:backward_lpz} is very similar to the one of \citeauthor{furhman_ergodic} \cite[Lemma~4.3]{furhman_ergodic}. In case the problem is time-homogeneous, \textit{i.e.} $b$ and $H$ do not depend on $t$, the delay does not affect much the proof of existence of solutions to the ergodic backward SDE \eqref{EBackwards SDE}. Therefore, introducing time-inhomogeneity is the true novelty of this paper.
\end{remark}

\subsection{Proofs of existence and uniqueness}

Let $\mathbf{0}$ be the constant zero function in $\mathcal{C}_{\theta}$. In this section, we lay the final stage towards the proof of existence for the ergodic backward SDE \eqref{EBackwards SDE}, which relies on a compactness argument, similarly to \citeauthor{Arisawa_Lions} \cite{Arisawa_Lions} and~\citeauthor{furhman_ergodic}~\cite{furhman_ergodic}. 

\begin{proposition}\label{prop:limit_decomposition}
    There exists a sequence $T_n \rightarrow \infty$ and a  function $\overline{v} :\R\times \mathcal{C}_{\theta} \rightarrow \mathbb{R}$, such that:
    \begin{equation}\label{eq:convergence_finite_infinite}
    \begin{array}{cc}
         \displaystyle  v^{T_n}(t, \xi) - v^{T_n}(0, \textnormal{\textbf{0}}) \xrightarrow[T_n \rightarrow \infty]{} \overline{v}(t, \xi) \,, & \text{for all } \, t\in\R \, \text{ and } \,  \xi \in \mathcal{C}_{\theta} \, .
    \end{array}
    \end{equation}
    Moreover, there exists a constant $C > 0$ such that, for all $t, t'\in\R$ and $\xi, \xi' \in \mathcal{C}_{\theta}$,
    \begin{equation}\label{vbar_continuity_estimates}
            \begin{array}{l}
                 \displaystyle \big|\overline{v}(t, \xi) - \overline{v}(t, \xi') \big| \leq C | \xi - \xi' |_{\theta} \, , \\[5pt]
                 \displaystyle \big|\overline v(t',\xi)-\overline v(t,\xi)\big|\le C \big(\,(1+|\xi|_{\theta})\,|t'-t|^{1/2}+|t-t'|+\omega_{\theta,\xi}(|t'-t|) \big) \, .
            \end{array}
    \end{equation}
\end{proposition}

\begin{proof} We define $u^T(t,\xi)\coloneqq v^T(t,\xi)-v^T(0,\textbf{0})$. Due to the dependence on $\xi$ of the time regularity estimate~\eqref{eq:backward_holder}, which is not uniform on bounded sets of $\C_\theta$, we proceed in two steps. \\

\vspace{-8pt}

\noindent \underline{Step 1:} Let $\xi \in \mathcal{C}_{\theta}$ and $(T_n)_{n \in\mathbb{N}}$ be a sequence converging to $\infty$. We prove that there is a subsequence $T_{n_k}\rightarrow\infty$, such that $u^{T_{n_k}}(\cdot,\xi)$ converges pointwise on $\R$. By Proposition \ref{prop:Markov_rep},  the family $\{u^{T_n}(\cdot,\xi)\}_{n \in\mathbb{N}}$ is uniformly bounded and equicontinuous on every compact subset of $\R$. Applying Arzelà–Ascoli Theorem, we obtain a subsequence $T_{n_k}\rightarrow\infty$ such that $u^{T_{n_k}}(\cdot,\xi)$ converges pointwise in $\R$.  \\

\vspace{-8pt}

\noindent \underline{Step 2:} Recall that the space $\mathcal{C}_{\theta}$ is separable, so there exists a countable subset $D = (\xi_i)_{i\in\mathbb{N}}$ dense in $\mathcal{C}_{\theta}$. By Step 1 and a diagonal argument, we can extract a sequence $T_n \rightarrow \infty$ such that $u^{T_n}(\cdot,\xi_i)$ converges pointwise in $\R$ for all $i\in\mathbb{N}$. Let $\varepsilon > 0$ and $(t,\xi) \in\R \times \C_{\theta}$. We can choose $\xi_i \in D$ such that $|\xi - \xi_i|_{\theta} < \varepsilon$. By the Lipschitz continuity of $u^T$ following \eqref{eq:backward_lpz}, and the triangle inequality, we have for all $n, m \geq 0$,
    \begin{align}
        |u^{T_m}(t,\xi)-u^{T_n}(t,\xi)| &\leq |u^{T_m}(t,\xi)-u^{T_m}(t,\xi_i)|+ |u^{T_m}(t,\xi_i)-u^{T_n}(t,\xi_i)|+|u^{T_n}(t,\xi_i)-u^{T_n}(t,\xi)| \nonumber \\
        &\leq 2C\varepsilon + |u^{T_m}(t,\xi_i)-u^{T_n}(t,\xi_i)| \, . \nonumber
    \end{align}
    Then there exists $n_0 \geq 0$ such that for all $n, m \geq n_0$, $|u^{T_m}(t,\xi)-u^{T_n}(t,\xi)|\leq (2C + 1) \varepsilon$. It follows that $u^{T_n}(t,\xi)$ is a Cauchy sequence and hence converges. Then there exists a  function $\bar{v}$ satisfying
\eqref{eq:convergence_finite_infinite}. The estimates of \eqref{vbar_continuity_estimates} are inherited from the continuity \eqref{eq:backward_lpz} and \eqref{eq:backward_holder} of $v^{T_n}(t,\xi)$  which is uniform in $T_n$. 
\end{proof}

%\vspace{5pt}

Next, we prove Theorem~\ref{thm:existence_uniqueness}, following the idea of
\citeauthor{furhman_ergodic}~\cite{furhman_ergodic}.
\begin{proof}[Proof of Theorem \ref{thm:existence_uniqueness}] In this proof, we set $\tau=0$ and omit the dependence on $\tau$ in the notation. Throughout the proof, $C$ denotes a positive constant, which may change from one occurrence
to another, and depends only on $\eta_1,\eta_2,\theta_0,\theta$ and $L_H$.  We proceed in three steps. \\

\vspace{-8pt}

\noindent \underline{Step 1:} We start by constructing the triple $(Y, Z, \Lambda)$. Let $\overline{v}$ be the function given in Proposition~\ref{prop:limit_decomposition}. We define $v(t, \xi) \coloneqq \overline{v}(t, \xi) - \overline{v}(t, \mathbf{0})$, then a function $\Lambda :\R_+ \rightarrow \mathbb{R}$ and a process $Y^{\xi} : \Omega \times\R_+ \rightarrow \mathbb{R}$ as:
\begin{equation}\label{Lambda_Y_def}
    \begin{array}{ccc}
         \Lambda(t) \coloneqq -\overline{v}(t, \textbf{0}) & \text{and} & Y_t^{\xi} \coloneqq v(t, \mathcal{X}^{\xi}_t) \, .
    \end{array}
\end{equation}
Note that $\Lambda$ is uniformly continuous by~\eqref{vbar_continuity_estimates} and the fact that $\omega_{\theta, \textbf{0}} \equiv 0$. Moreover, $Y^\xi\in\mathbb S^2_T(\mathbb R)$ for every $T\in\mathbb R_+$, by~\eqref{vbar_continuity_estimates} and Proposition~\ref{prop:forward_holder}. Let $T_n$ be the sequence defined in Proposition~\ref{prop:limit_decomposition}. In the following we exhibit a process $Z^\xi$ such that $Z^{T_n,\xi} \rightarrow Z^{\xi}$ in $\mathbb{H}^2_T(\R^d)$ for all $T > 0$. Let $T > 0$ and $T_n,T_m \geq T$. Define the processes $Y^{n,m}_s \coloneqq (Y_s^{T_n,\xi}-v^{T_n}(0,\mathbf{0}))-(Y_s^{T_m,\xi}-v^{T_m}(0,\mathbf{0}))$ and $Z^{n,m}_s \coloneqq Z^{T_n,\xi}_s - Z^{T_m,\xi}_s.$
 Applying Itô’s formula to $(Y^{n,m})^2$, we can write:
 \begin{align}
     (Y^{n,m}_0)^2 + \mathbb{E} \int_0^T \big|Z^{n,m}_s \big|^2 ds &= \mathbb{E}(Y^{n,m}_T)^2  + 2 \, \mathbb{E}\int_0^T  Y^{n,m}_s \Big[H(s, \mathcal{X}^\xi_s,Z^{T_n,\xi}_s)-H(s, \mathcal{X}^\xi_s,Z^{T_m,\xi}_s) \Big] ds  \nonumber \\
     & \leq \mathbb{E}(Y^{n,m}_T)^2 + 2 L_H \, \mathbb{E} \int_0^T Y^{n,m}_s |Z_s^{n,m}| \, ds \, . \nonumber
 \end{align}
By Young's inequality and Assumption \ref{hamiltonian_assumption}, we obtain:
\begin{equation} \label{eq:z_n,m_boundedness}
    \mathbb{E} \int_0^T \big|Z^{n,m}_s \big|^2 ds \leq C \Big( \mathbb{E}(Y^{n,m}_T)^2 + \mathbb{E}\int_0^T  \big|Y^{n,m}_s \big|^2 ds \Big).
\end{equation}
By the triangle inequality and Proposition~\ref{prop:Markov_rep}, 
\begin{align}
    |Y^{T,\xi}_s - v^T(0, \textbf{0})| &\leq |v^T(s, \mathcal{X}^{\xi}_s)-v^T(s, \textbf{0})|+|v^T(s, \textbf{0})-v^T(0, \textbf{0})| \nonumber 
    %\\[5pt] &
    \leq C(|\mathcal{X}^\xi_s|_{\theta}+\sqrt{s}+s) \, .\nonumber
\end{align}
By Proposition \ref{prop:forward_holder}, we have $\mathbb{E} \sup_{t \in [0, T]} |\mathcal{X}^\xi_t|_{\theta}^2 < \infty$. Together with the Lipschitz continuity of \eqref{vbar_continuity_estimates}, this implies that $\mathbb{E} \sup_{t \in [0, T]} \big[ \sup_{n,m} |Y^{n,m}_t|^2 \big] < \infty$. Therefore, as $\lim_{n,m\to\infty} Y^{n,m}_s=0$, it follows from the Dominated Convergence Theorem that:
\begin{equation*}
    \begin{array}{ccc}
        \displaystyle \mathbb{E} \int_0^T \big| Y^{n,m}_s \big|^2 ds \rightarrow 0 & \text{and} & \mathbb{E} \big| Y^{n,m}_T \big|^2 \rightarrow 0 \, ,
    \end{array}
\end{equation*}
as $n ,m \rightarrow \infty$. We deduce from \eqref{eq:z_n,m_boundedness} that $(Z^{T_n, \xi})$ is a Cauchy sequence in $\mathbb{H}^2_T(\R^d)$, so it converges to a process $Z^{\xi}$. This holds for every $T\in\R_+$, and thus $Z^\xi\in \mathbb{H}^0(\R^d)$ and $Z^\xi\in \HH_T(\R^d)$ for all $T\geq 0$. \\

\vspace{-8pt}

\noindent\underline{Step 2:}
We next verify that $(Y, Z, \Lambda)\in \mathbf{E}(\C_{\theta}) \times \rm{UC}$ solves the ergodic backward SDE. Let $0 \leq t \leq T < \infty$ and note that $(Y^{\xi}, Z^{\xi}) \in \mathbb{S}^2_T(\R) \times \mathbb{H}^2_T(\R^d)$ and $\Lambda$ is a continuous function. Then for $T_n > T$, 
\begin{equation*}
    Y^{T_n, \xi}_t - v^{T_n}(0, \textbf{0}) = Y^{T_n, \xi}_T - v^{T_n}(0, \textbf{0}) + \int_t^T H\big(s, \mathcal{X}^{\xi}_s, Z^{T_n, \xi} \big) \, ds - \int_t^T Z^{T_n, \xi} \cdot dW_s \, .
\end{equation*}
By Proposition~\ref{prop:limit_decomposition}, $Y^{T_n, \xi}_s - v^{T_n}(0, \textbf{0}) \rightarrow Y^\xi_s-\Lambda(s)$ for all $s \geq 0$. Moreover, we have:
\begin{equation*}
    \Big|\int_t^T H\big(s, \mathcal{X}^{\xi}_s,Z_s^{T_n, \xi}\big)\,ds-\int_t^T H\big(s, \mathcal{X}^{\xi}_s,Z_s^{\xi}\big)\,ds\Big| \leq L_H \int_t^T |Z_s^{T_n, \xi}-Z_s^{\xi} |\, ds \longrightarrow 0, ~\mbox{as}~n\to\infty,
\end{equation*}
by Step 1. The estimate \eqref{Y_decompsition} is a consequence of the definition \eqref{Lambda_Y_def} of $\Lambda$ and $Y^{\xi}$ and the fact that $\overline{v}$ is Lipschitz continuous in the second variable. Finally, by Lemma~\ref{lem:mollification} and Remark~\ref{rmk:smooth_repre}, $\Lambda$ can be chosen in $C^\infty$, with bounded derivatives of all orders, and such that $\Lambda(0)=0$. \\

\vspace{-8pt}

\noindent\underline{Step 3:} Let $(\overline{Y}, \overline{Z}, \overline{\Lambda}) \in \mathbf{E}(\C_\theta) \times {\rm{UC}}$ be another solution of the ergodic backward SDE and let us prove that $\Lambda - \overline{\Lambda}$ is bounded. Since $\Lambda$ and $\overline{\Lambda}$ do not depend on the initial data $\xi \in \mathcal{C}_{\theta}$, we fix $\xi=\textbf{0}$ and remove the upper index in our notations. For $T\geq 0$, one has:
\begin{align}
    Y_0 - \overline{Y}_0-(\Lambda(0)-\overline{\Lambda}(0)) = \big(Y_T - \overline{Y}_T \big) - \big(\Lambda(T) - \overline{\Lambda}(T) \big) - \int_0^T \big( Z_s - \overline{Z}_s \big) \cdot dW^\beta_s  \, , \nonumber
\end{align}
where $\beta_s := \frac{H(s, \mathcal{X}_s,Z_s)-H(s, \mathcal{X}_s,\overline{Z}_s)}
{|Z_s-\overline{Z}_s|^2}\,
(Z_s-\overline{Z}_s) \mathbf{1}_{\{Z_s \neq \overline{Z}_s\}}$ is bounded by Assumption~\ref{hamiltonian_assumption}. We obtain by Lemma~\ref{lem:girsanov_intZ_is_mg} and the growth condition \eqref{Y_decompsition}:
\[
|\Lambda(T)-\overline\Lambda(T)| \, = \, \big| \mathbb{E}^\beta\!\big[\,Y_T -\overline{Y}_T\,\big] - \big(Y_0 - \overline Y_0 \big)-(\Lambda(0)-\overline{\Lambda}(0)) \big| \, \leq \,   | Y_0 - \overline Y_0 |+|\Lambda(0)-\overline{\Lambda}(0)| + 2C \, ,
\]
where we used that $\E^\beta|\mathcal{X}_T|_{\theta}^2 \leq C$ for some $C > 0$ by Proposition \ref{prop:FSDE_existence}. We obtain the estimate \eqref{eq:Lambda_unique_up_to_constant}. 
\end{proof}

\vspace{5pt}

\begin{remark}\label{rmk:ebsde_flow_path}
By the same argument as in the proof of Theorem~\ref{thm:existence_uniqueness}, we may consider the case of an arbitrary time origin $\tau$, and we may conclude that there exist continuous functions $v:\R\times\C_\theta\mapsto\R$, $\Lambda:\R\mapsto\R$, and  
$(Y^\tau, Z^\tau) \in \mathbf{E}_\tau(\C_{\theta})$ for every $\tau\in\R$ such that $Y^{\tau,\xi}_t = v(t, \mathcal{X}^{\tau,\xi}_t)$ for all $t\geq \tau,\xi\in  \mathcal{C}_{\theta}$, $\Lambda\in \rm{UC}_\tau$, $\Lambda(0)=0$, $Z^{\tau,T_n,\xi}\rightarrow Z^{\tau,\xi}$ in $\HH_{\tau,T}(\R^d)$ for $\tau\leq T$, and

\begin{equation}\label{eq:ebsde_flow_path}
     \displaystyle Y_t^{\tau,\xi} - \Lambda(t)
    =
    Y_T^{\tau,\xi} - \Lambda(T)
    + \int_t^T H\big(s,\X_s^{\tau,\xi},Z_s^{\tau,\xi}\big)\,ds
    - \int_t^T Z_s^{\tau,\xi}\cdot dW^{\tau}_s \, ,  \tau \leq t \leq T\, .
\end{equation}
\end{remark}

\section{Stability of the ergodic values and costs}\label{sec:stability}

In this section we prove Theorems~\ref{thm:stability_values} and~\ref{thm:stability_control} under Assumptions~\ref{dissipativity_assumption},
\ref{hamiltonian_assumption}, and~\ref{control_assumption} for $b,H,f$ and $\tilde{b},\tilde{H},\tilde{f}$. For simplicity we set the time origin to be $\tau = 0$ and we omit it in the notations. We denote by $X^\xi$ the unique solution of \eqref{Functional_SDE} and by $(\mathcal X_t^\xi)_{t\geq 0}$ the associated path process, defined by $\mathcal X_t^\xi(s)=X_{t+s}^\xi$ for $s\leq 0$. Similarly we define $\tilde X^\xi$ and $\tilde{\mathcal X}^\xi$ for the drift $\tilde b$. We consider measurable functions $\delta b$, $\delta H$, and $\delta f$ satisfying
\eqref{eq:delta_bH} and~\eqref{eq:delta_f}. We start with the forward equation.

\begin{proposition}[Stability in $b$]\label{prop:stability_b}
There exists a constant $C>0$ depending only on $\eta_1,\eta_2,\theta_0,\theta$ such that, for any $\xi\in\mathcal C_{\theta}^0$ and $K\in\K,$
\begin{equation*}
\begin{array}{cc}
     \displaystyle \int_0^{\infty} K(s)|\mathcal{X}^{\xi}_s- \tilde{\mathcal{X}}^{\xi}_s|_{\theta} \, ds\leq \Big(\int_0^{\infty} K(s)|\mathcal{X}^{\xi}_s- \tilde{\mathcal{X}}^{\xi}_s|_{\theta} \, ds\Big)^{\frac{1}{2}}\leq  C  \,  \Big(\int_0^{\infty} K(s)(\delta b(s))^2 \,ds\Big)^{\frac{1}{2}} \, , & \mathbb{P}-\text{a.s.}
\end{array}
\end{equation*}
\end{proposition}

\proof
We denote $\delta \hspace{-1pt} X^\xi_s = X^\xi_s-\tilde{X}^\xi_s$ for $s\in\R$. Let $\eta_0 \coloneqq \eta_1 - 2 \theta - \frac{\eta_2}{2(\theta_0 - \theta)} > 0$ and $\nu_t = 2 \, \delta \hspace{-1pt} X^\xi_t \cdot (b(t, \mathcal{X}^{\xi}_t)-\tilde{b}(t, \tilde{\mathcal{X}}^{\xi}_t))$. By Young's inequality, we have for all $t \geq 0$,
\begin{align}\label{eq:stability_dissipativity_estimate}
    \nu_t &= 2 \, \delta \hspace{-1pt} X^\xi_t \cdot \big(b(t, \mathcal{X}^{\xi}_t)-b(t, \tilde{\mathcal{X}}^{\xi}_t)\big)+2 \, \delta \hspace{-1pt} X^\xi_t \cdot \big(b(t, \tilde{\mathcal{X}}^{\xi}_t)-\tilde{b}(t, \tilde{\mathcal{X}}^{\xi}_t)\big)\notag\\
    &\leq -\eta_1 |\delta \hspace{-1pt} X^\xi_t|^2
    +\eta_2\int_{-\infty}^t e^{2\theta_0(s-t)}|\delta \hspace{-1pt} X^\xi_s|^2\,ds
    +  \eta_0 |\delta X^\xi_t|^2  \, +\frac{(\delta b(t))^2}{\eta_0} \notag\\
    &= -(2\theta+\frac{\eta_2}{2(\theta_0-\theta)})|\delta \hspace{-1pt} X^\xi_t|^2
    +\eta_2 \int_{-\infty}^t e^{2\theta_0(s-t)}|\delta \hspace{-1pt} X^\xi_s|^2\,ds
    +\frac{(\delta b(t))^2}{\eta_0} \nonumber \, .
\end{align}
Now define $U_t \coloneqq |\delta \hspace{-1pt} X^\xi_t|^2$ and note that $U_t=0$ for $t\leq0$ and $dU_t = \nu_t dt$ for $t > 0$. Then the process $V$ defined in \eqref{def_V_t} satisfies $V_s \leq \frac{(\delta b(s))^2}{\eta_0}$, with $\gamma_1 = 2\theta+\frac{\eta_2}{2\theta_0-2\theta}$ and $\gamma_2 = \eta_2$. Applying Lemma~\ref{lem:exp_decreasing} and taking the supremum over $[0, t]$, we obtain:
\begin{align}
    |\mathcal{X}^{\xi}_t - \tilde{\mathcal{X}}^{\xi}_t |_{\theta}  =\sup_{s\in [0,t]}e^{\theta(s-t)} |\delta \hspace{-1pt} X^\xi_s|\leq \Big(\int_0^t e^{2\theta (s-t)} \frac{(\delta b(s))^2}{\eta_0} \, ds \Big)^{1/2}\nonumber \, 
\end{align}
Finally, the left inequality follows from the Cauchy--Schwarz inequality
and the fact that $K\in\K$. For the right inequality,
using Fubini's theorem and the monotonicity of $K$, we obtain
\begin{align*} 
\int_0^{\infty} K( s)\big| \mathcal{X}^{\xi}_s- \tilde{\mathcal{X}}^{\xi}_s \big|_{\theta}^2\,ds \le 
\int_0^{\infty} \int_0^s K(r)e^{2\theta (r-s)} \frac{(\delta b(r))^2}{\eta_0}\,dr\,ds\le 
\frac{1}{2 \eta_0 \theta}\int_0^{\infty} K(r)(\delta b(r))^2\,dr \, .
\end{align*}

\vspace{-10mm}
\qed

\noindent \textit{Proof of Theorem \ref{thm:stability_values}.} In this proof, we denote by $\X$ (resp. $\tilde{\X}$) the solution of the forward equation associated with $b$ (resp. $\tilde b$), with initial time $\tau=0$, initial data $\xi = \mathbf{0}$. Let $(Y,Z,\Lambda)$ and $(\tilde Y,\tilde Z,\tilde\Lambda)$ denote a solution of the ergodic backward SDE for the system $(b,H)$ and $(\tilde b,\tilde H)$ respectively with initial condition $\xi=\mathbf{0}$. By Theorem~\ref{thm:existence_uniqueness}, we can assume that $\Lambda,\tilde{\Lambda}$ are continuously differentiable with bounded derivatives. Let $K\in\K$, and $T,\rho>0$. Since $t\mapsto K_\rho(t)=\rho K(\rho t)$ is deterministic and of finite variation, applying the integration-by-parts formula to $K_\rho(Y-\tilde Y)$ on $[0,T]$ yields:
\begin{align*}
K_\rho(T)(Y_T-\tilde Y_T)-K_\rho(0)(Y_0-\tilde Y_0)
&= \int_0^T K_\rho(s)(Z_s-\tilde Z_s)\cdot dW^\beta_s + \int_{(0,T]} (Y_s-\tilde Y_s)\,dK_\rho(s) \\
&+\int_0^T K_\rho(s)\Big(-H(s,\mathcal X_s,Z_s)+ \tilde{H}(s,\tilde{\mathcal X}_s,Z_s)+\Lambda'(s)-\tilde\Lambda'(s)\Big)\,ds \, .
\end{align*}
where $\beta_s:=\frac{\tilde H(s,\tilde{\mathcal X}_s,Z_s)-\tilde H(s,\tilde{\mathcal X}_s,\tilde Z_s)}{|Z_s-\tilde Z_s|^2}(Z_s-\tilde Z_s)\mathbf 1_{\{Z_s\ne\tilde Z_s\}}$ is bounded by Assumption \ref{hamiltonian_assumption}. By Lemma~\ref{lem:girsanov_intZ_is_mg}, taking expectation under $\E^\beta$ gives
\begin{align}
K_\rho(T)&\E^\beta[Y_T-\tilde Y_T]-K_\rho(0)(Y_0-\tilde Y_0)  = \int_{(0,T]} \E^\beta[Y_s-\tilde Y_s]\,dK_\rho(s)  \label{eq:stability_ito_monotoneK} \\
&+\int_0^T K_\rho(s)\Big(-\E^\beta\big[H(s,\mathcal X_s,Z_s)- \tilde{H}(s,\tilde{\mathcal X}_s,Z_s)\big]+\Lambda'(s)-\tilde\Lambda'(s)\Big)\,ds \nonumber \, .
\end{align}
Recall that $\sup_{s\ge0}\E^\beta|Y_s-\tilde Y_s|<\infty$, by Proposition~\ref{prop:FSDE_existence} and the growth condition~\eqref{Y_decompsition}. Since $K$ is non-increasing and integrable, we have $K_\rho(T)\to0$ as $T\to\infty$. Passing to the limit in \eqref{eq:stability_ito_monotoneK}, we get:
$$
K_\rho(0)(Y_0\!-\!\tilde Y_0) =\!\int_0^\infty\!\!\!K_\rho(s)\Big(\E^\beta\big[H(s,\mathcal X_s,Z_s)+\tilde H(s,\tilde{\mathcal X}_s,Z_s)\big]+\Lambda'(s)-\tilde\Lambda'(s)\Big)\,ds 
%\\&\quad 
-\int_0^\infty\!\!\! \E^\beta[Y_s\!-\!\tilde Y_s]\,dK_\rho(s) \, .
$$
Notice that:
\(
|H(s,\mathcal X_s,Z_s)-\tilde H(s,\tilde{\mathcal X}_s,Z_s)|
\le
|H(s,\mathcal X_s,Z_s)-H(s,\tilde{\mathcal X}_s,Z_s)|
+
|H(s,\tilde{\mathcal X}_s,Z_s)-\tilde H(s,\tilde{\mathcal X}_s,Z_s)| \, .
\)
Apply Proposition~\ref{prop:stability_b} to $K_\rho\in\K$ gives:
\begin{align*}
\Big|\int_0^\infty\!\! & K_{\rho}(s)\,\E^\beta\big[H(s,\mathcal X_s,Z_s)-\tilde H(s,\tilde{\mathcal X}_s,Z_s)\big]\,ds\Big| \le C L_H \Big(\int_0^\infty\!\! K_{\rho}(s)(\delta b(s))^2\,ds\Big)^{1/2}
+\int_0^\infty\!\! K_{\rho}(s)\delta H(s)\,ds .
\end{align*}
Finally, since $s\mapsto \E^\beta|Y_s-\tilde Y_s|$ is bounded and $K$ is non-increasing,
\[
\Big|\int_0^\infty \E^\beta[Y_s-\tilde Y_s]\,dK_\rho(s)\Big|
\le \sup_{s\ge0}\E^\beta|Y_s-\tilde Y_s| \int_0^\infty d(-K_\rho(s))
\le K_\rho(0)\sup_{s\ge0}\E^\beta|Y_s-\tilde Y_s| \, ,
\]
which converges to $0$ as $\rho\downarrow0$ since $K_\rho(0)=\rho K(0)$. We obtain the desired estimate, since:
\begin{align*}
\limsup_{\rho\downarrow0}\Big|\int_0^\infty\!\! K_{\rho}(s)(\Lambda'(s)\!-\!\tilde\Lambda'(s))\,ds\Big|
\le
\limsup_{\rho\downarrow0}\Big(
CL_H \Big(\int_0^\infty\!\! K_{\rho}(s)(\delta b(s))^2\,ds\Big)^{1/2}
+\int_0^\infty\!\! K_{\rho}(s)\delta H(s)\,ds
\Big) \, .
\end{align*}

\vspace{-3mm}
\qed

%\vspace{10pt}

\noindent {\it Proof of Theorem~\ref{thm:stability_control}.}
    Let $\alpha\in \mathcal{A},\rho>0$. The claim follows from Proposition \ref{prop:stability_b} and the inequalities:
\begin{align*}\label{eq:stability_control_equation}
 \Big|\E^{\alpha}&\int_{0}^\infty K_{\rho}(t)(f(t, \X^\xi_t, \alpha_t)-\tilde{f}(t, \tilde{\X}^\xi_t, \alpha_t))\,  dt \Big| \nonumber\\  
 & \leq \Big|\E^{\alpha}\int_{0}^\infty K_{\rho}(t) (f(t, \X^\xi_t, \alpha_t)-f(t, \tilde{\X}^\xi_t, \alpha_t)) \,  dt\Big| +\Big|\E^{\alpha}\int_{0}^\infty K_{\rho}(t)(f(t, \tilde{\X}^\xi_t, \alpha_t)-\tilde{f}(t, \tilde{\X}^\xi_t, \alpha_t)) \, dt \Big|\nonumber \\
&\leq L_f \, \E^{\alpha}\int_{0}^\infty K_{\rho}(t)|\X^\xi_t-\tilde{\X}^\xi_t |_{\theta} \, dt +\int_{0}^\infty K_{\rho}(t) \,  \delta \hspace{-1pt} f(t)\, dt\, . \nonumber
\end{align*}

\vspace{-8mm}\qed

\section{Further results in special cases}\label{sec:special_cases}

In this section, we present three cases in which we can study the solution $(Y, Z,\Lambda)$ in further details. When the system is time-periodic, we observe that it behaves as if it were time-homogeneous Markov and the ergodic values are equal. In the time-inhomogeneous Markov case, we derive a uniqueness result for $(Y, Z)$. Finally, if the dynamics and the Hamiltonian are path-dependent with a finite delay, we obtain an explicit representation of $\Lambda$ in terms of the evolution system of measures of $\mathcal{X}$.

\subsection{Time-periodic systems}

In this section we assume that Assumptions~\ref{dissipativity_assumption} and~\ref{hamiltonian_assumption} hold, and $(b,H)$ is time-periodic. The next result shows that we obtain a unique ergodic cost in that case. This is consistent with the work of \citeauthor{cohen_periodic_EBSDE} \cite{cohen_periodic_EBSDE} for time-periodic Ornstein--Uhlenbeck dynamics, and \citeauthor{wu2026ergodiclinearquadraticoptimalcontrol} \cite{wu2026ergodiclinearquadraticoptimalcontrol} for a linear--quadratic with time-periodic coefficients. In the following we set the time origin to $\tau=0$ and omit it in the notations.

\vspace{5pt}

\begin{proposition}\label{prop:periodic_case}
Assume that the system $(b,H)$ is periodic in time. Then there exists a triple $(Y,Z,\lambda)\in\mathbf{E}(\C_{\theta})\times\R$ such that $(Y,Z, \Lambda)\in\mathbf{E}(\C_{\theta}) \times {\rm{UC}}$ is a solution of the ergodic backward SDE with $\Lambda(t) = \lambda t$. As a consequence, $\lambda_*(K)=\lambda^*(K)=\lambda$ for all $K\in\K$.
\end{proposition}

\begin{proof}
Let $T_0$ be a period of $(b,H)$, and let us proceed in three steps. \\

\vspace{-8pt}

\noindent\underline{Step 1:} Let $T, T' > 0$ and $(\X,Y^{T},Z^{T})$ and $(\X,Y^{T'},Z^{T'})$ be the unique solutions of the finite-horizon forward-backward SDE~\eqref{Functional_SDE}--\eqref{BSDE_finite_horizon}
with data $(0,T,\mathbf{0})$ and $(0,T',\mathbf{0})$, respectively. We denote by $v^T$ the function defined in Proposition \ref{prop:Markov_rep}. We first prove that:
\begin{equation}\label{eq:T,T'}
\sup_{\substack{0\le t\le T \wedge T'}}
\big|\big(v^{T}(0,\mathbf{0})-v^{T}(t,\mathbf{0}) \big)-\big(v^{T'}(0,\mathbf{0})-v^{T'}(t,\mathbf{0}) \big)\big| < \infty \, .    
\end{equation}
By Assumption \ref{hamiltonian_assumption}, $\beta_s := \frac{H( s, \mathcal{X}_s , Z_s^{T}) - H(s, \mathcal{X}_s,Z_s^{T'})}{|Z^{T}_s-Z^{T'}_s |^2}(Z^{T}_s-Z^{T'}_s) \mathbf{1}_{\{Z^{T}_s\neq Z^{T'}_s \}} \mathbf{1}_{\{ s\leq T\}}$ is a bounded process. By Proposition~\ref{prop:Markov_rep} and Lemma~\ref{lem:girsanov_intZ_is_mg},
\begin{align}
    \E^\beta \big[ \big(v^{T}(0,\mathbf{0})-v^{T}(t,\mathcal{X}_t) \big) - \big( v^{T'}(0,\mathbf{0}) - v^{T'}(t,\mathcal{X}_t) \big) \big]&=\E^\beta[(Y^T_0-Y^T_t) - (Y^{T'}_0-Y^{T'}_t)] \nonumber \\
    &=-\E^\beta \int_0^t (Z^{T}_s-Z^{T'}_s)\cdot dW^\beta_s \, = \,  0 \, . \nonumber
\end{align}
Moreover, by Propositions~\ref{prop:FSDE_existence} and~\ref{prop:Markov_rep},
$$\sup_{0\leq t\leq T \wedge T'} \big|\E^\beta \big[ \big(v^{T}(t,\mathbf{0})-v^{T}(t,\mathcal{X}_t) \big) - \big( v^{T'}(t,\mathbf{0}) - v^{T'}(t,\mathcal{X}_t) \big) \big] \big| \, \leq \,  2C \, \sup_{t \geq 0 }\E^\beta|\mathcal{X}_t|_\theta \, < \, \infty \, ,$$
and \eqref{eq:T,T'} follows from the triangle inequality. \\

\vspace{-8pt}

\noindent \underline{Step 2:} Recall that we may choose $(T_n)_{n\in\mathbb{N}}$ as in the Step 2 of the Proof of Proposition~\ref{prop:limit_decomposition} and  $\Lambda(t) \coloneqq -\lim_{T_n\rightarrow\infty}(v^{T_n}(t,\mathbf{0})-v^{T_n}(0,\mathbf{0}))$ as in the Step 1 of the proof of Theorem~\ref{thm:existence_uniqueness}, and let us establish that:

\vspace{-10pt}

\begin{equation}\label{eq:k_1,k_2}
\sup_{\substack{k_1,k_2\in\mathbb{N}}}
\big|\Lambda((k_1+k_2)T_0)- \Lambda(k_1T_0)-\Lambda(k_2T_0)\big| < \infty \, .    
\end{equation}
Given $k_1,k_2\in\mathbb{N}$ and for $n$ large enough such that $T_n\geq k_2T_0$, we take $t=k_1T_0$, $T=T_n$ and $T'=T_n-k_2T_0$ in~\eqref{eq:T,T'}, and obtain that:

$$\sup_{k_1,k_2\in\mathbb{N}, \, T_n\geq k_2T_0}
\big|v^{T_n}(0,\mathbf{0})-v^{T_n}(k_1T_0,\mathbf{0})-v^{T_n-k_2T_0}(0,\mathbf{0})+v^{T_n-k_2T_0}(k_1T_0,\mathbf{0})\big|<\infty \, .
$$
Since $b$ and $H$ have a period $T_0$, $v^{T_n-k_2T_0}(0,\mathbf{0})-v^{T_n-k_2T_0}(k_1T_0,\mathbf{0})=v^{T_n}(k_2T_0,\mathbf{0})-v^{T_n}((k_1+k_2)T_0,\mathbf{0})$. Thus we have:

\vspace{-10pt}

$$\sup_{k_1,k_2\in\mathbb{N}, \, T_n\geq k_2T_0}
\big|v^{T_n}(0,\mathbf{0})-v^{T_n}(k_1T_0,\mathbf{0})-v^{T_n}(k_2T_0,\mathbf{0})+v^{T_n}((k_1+k_2)T_0,\mathbf{0})\big|<\infty \, .$$
Then we obtain~\eqref{eq:k_1,k_2} by passing to the limit $n\rightarrow\infty$. \\

\vspace{-8pt}

\noindent \underline{Step 3:} By~\eqref{eq:k_1,k_2} and Lemma~\ref{lem:cauchy_equation}, whose proof follows the idea of \citeauthor{Hyers1941}~\cite{Hyers1941} and can be viewed as a variant of Fekete's lemma, there exists $\lambda\in\R$ such that $\sup_{k\in\mathbb N}\bigl|\Lambda(kT_0)-\lambda kT_0\bigr|<\infty$. Since $\Lambda\in {\rm{UC}}$, it follows that $\sup_{t\ge0}\bigl|\Lambda(t)-\lambda t\bigr| < \infty$. Moreover, by Steps 1--3 in the proof of Theorem~\ref{thm:existence_uniqueness}, there exists $(Y,Z)\in\mathbf E$ such that $(Y,Z,\Lambda)$ is a solution of the ergodic backward SDE~\eqref{EBackwards SDE}. Consequently, defining $\overline{\Lambda}(t)=\lambda t$, then $(Y+\overline{\Lambda
}-\Lambda,Z,\overline{\Lambda})$ is also a solution of the ergodic backward SDE~\eqref{EBackwards SDE}.
\end{proof}

\subsection{Uniqueness in the time-inhomogeneous Markov setting}\label{sec:time_inhom}

In this section we assume that Assumptions \ref{dissipativity_assumption} and \ref{hamiltonian_assumption} hold and the system to be Markov time-inhomogeneous, \textit{i.e.} there exist two functions $\overline{b} : \mathbb{R} \times \mathbb{R}^d\to\R$ and $\overline{H}: \mathbb{R} \times \mathbb{R}^d \times \mathbb{R}^d\to \R$ such that $b(t, \xi) = \overline{b}(t, \xi_0)$ and $H(t, \xi, z) = \overline{H}(t, \xi_0, z)$. To simplify the notations, we will write $b$ and $H$ instead of $\overline{b}$ and $\overline{H}$, and $x\in\R^d$ instead of $\xi\in\C_\theta$. We first reformulate Assumption \ref{dissipativity_assumption} in the present context.

\begin{assumption} \label{time_inhom_assumption}
\setcounter{assumptiontag}{0}
    \begin{itemize}
        \item[]

        \refstepcounter{assumptiontag}
        \item[] \hspace{-30pt} (\theassumptiontag) \label{Lpz_condition_lip_inhom} \hspace{2pt} \textnormal{Lipschitz and linear growth:} {\it There exists $L_b>0$ such that for any $t \in \mathbb{R}$ and $x, x' \in \mathbb{R}^d$,
        \begin{equation*}
            \begin{array}{ccc}
                \big| b(t, x) - b(t, x') \big| \leq L_b \, | x-x'| & \text{and} &
                |b(t, x)| \leq L_b \big(1 + | x | \big) \,.
            \end{array}
        \end{equation*}}

        \refstepcounter{assumptiontag}
        \item[] \hspace{-26pt}(\theassumptiontag) \label{dissipativity_condition_inhom} \hspace{1pt} \textnormal{Dissipativity:} {\it There exists $\eta_0 > 0$ such that:
        \begin{equation*}
            \begin{array}{cc}
            \displaystyle 2(x - x') \cdot \big(b(t, x) - b(t, x')\big) \leq - \eta_0 |x-x'|^2 \, , & \text{for all } \, x, x' \in \mathbb{R}^d \, , \, t \in \mathbb{R} \, .
            \end{array}
        \end{equation*}}
    \end{itemize}
\end{assumption} 

All conclusions from the previous sections hold in the present case. In particular, in Proposition~\ref{prop:forward_holder}, the modulus of continuity term $\omega_{\xi,\theta}$ vanishes. In Proposition~\ref{prop:ebsde_flow_uniqueness} below, we prove the existence of a decoupling field for the flow version of the ergodic backward SDE. We prove its uniqueness in Theorem~\ref{thm:uniqueness_v_zeta} below, in a suitable sense. By Theorem~\ref{thm:existence_uniqueness} and Remark~\ref{rmk:ebsde_flow_path}, there exists a solution $(Y^\tau,Z^\tau,\Lambda)$ to the flow version of the ergodic backward SDE~\eqref{eq:ebsde_flow_path}. Moreover, there exists a Markovian representation function $v$ for $Y$ such that $Y^{\tau,\xi}_t = v(t,\X^{\tau,\xi}_t)$. In the inhomogeneous Markov case, we can also prove the existence of a Markovian representation function $\zeta$ for $Z$.

\vspace{5pt}

\begin{proposition}\label{prop:ebsde_flow_uniqueness}
There exists a measurable function $\zeta : \mathbb{R} \times \R^d \to \mathbb{R}^d$ such that
$Z_t^{\tau,x} = \zeta(t,X_t^{\tau,x})$, $d\mathbb{P}\otimes dt$-a.e., for all $\tau\in\R$ and $x \in \R^d$, where $(Y^\tau,Z^\tau,\Lambda)$ is the solution of ergodic backward SDE~\eqref{eq:ebsde_flow_path} as constructed in Theorem~\ref{thm:existence_uniqueness} and Remark~\ref{rmk:ebsde_flow_path}.
\end{proposition}

\begin{proof}
Throughout the proof, $C$ denotes a positive constant, which may change from one occurrence to another. We consider the sequence $(T_n)_{n\in\mathbb{N}}$ constructed in Proposition~\ref{prop:limit_decomposition}. Recall that by the same argument as in the proof of Theorem~\ref{thm:existence_uniqueness}, $Z^{\tau,T_n,x}$ converges to $Z^{\tau,x}$ in $\mathbb{H}^2_{\tau,T}(\R^d)$ for all $T\geq \tau$. We claim that there exists a
subsequence $T'_n$ of $T_n$, with $T'_n\geq n$ such that $\|Z^{\tau,T'_n,x}-Z^{\tau,x}\|_{\HH_{\tau,n}}\leq C(1+n^2)2^{-n}$ for all $x\in\mathbb{R}^d$, $\tau\leq T'_n$, and integers $n\geq |x|+|\tau|$. By Propositions \ref{prop:forward_lip}, \ref{BSDE_finite_horizon} and \ref{prop:limit_decomposition}, and the same argument which lead to \eqref{eq:z_n,m_boundedness}, we obtain that for all $\tau\leq T\leq T'$ and $x,x'\in\R^d$,
\begin{equation} \label{eq:z_x,x'_boundedness}
   \hspace{-5pt} \|Z^{\tau,T',x}-Z^{\tau,T',x'}\|_{\HH_{\tau,T}} \, \leq \,  \|Z^{\tau,T',x}-Z^{\tau,T',x'}\|_{\HH_{\tau,T'}} \, \leq \, C  \|Y^{\tau,T',x}-Y^{\tau,T',x'}\|_{\HH_{\tau,T'}}\, \leq \,  C|x-x'|^2 \, .
\end{equation}
Similarly, for all $\tau\le \tau'\leq T\leq T'$, since $Z^{\tau',T,x}$ is unique in law, we may replace $W^{\tau'}$ by the shifted Brownian motion
$W^\tau_t-W^\tau_{\tau'}$, $t\ge \tau'$, without changing the law of $Z^{\tau',T,\xi}$. By Proposition~\ref{prop:forward_holder} and the same argument as before, we have: 
\begin{equation}
\label{eq:z_tau_boundedness}
    \|Z^{\tau',T',x}-Z^{\tau,T',x}\|_{\HH_{\tau',T}} \, \leq \,   C\E|X^{\tau',x}_\tau-x|^2 \, \leq \, C(1+|x|^2)(\tau'-\tau) \, .
\end{equation}
Then, passing to the limit $T'=T_n\to\infty$ in~\eqref{eq:z_x,x'_boundedness} and~\eqref{eq:z_tau_boundedness}, we obtain, for $\tau\leq\tau'\leq T$,
\begin{equation}\label{eq:z_x,tau}
\begin{array}{ccc}
     \|Z^{\tau,,x}-Z^{\tau,x'}\|_{\HH_{\tau,T}} \leq  C|x-x'|^2 & \text{and} & \|Z^{\tau',x}-Z^{\tau,x}\|_{\HH_{\tau',T}} \leq C(1+|x|^2)(\tau'-\tau) \, .   
\end{array}
\end{equation}
Given $n\in\mathbb{N}$, since $E_n=[-n,n]^{d+1}$ is compact, there exists a finite set  $(\tau^n_i,x^n_i)_i\subset E_n$, such that $E_n\subseteq \cup_i B((\tau^n_i,x^n_i),2^{-n})$. Thus there exists a subsequence  $(T'_n)$ of $(T_n)$ such that $T'_n\geq n$ and \begin{equation}\label{eq:Z_partition}
\|Z^{\tau^n_i,T'_n,x^n_i}-Z^{\tau^n_i,x^n_i}\|_{\HH_{\tau^n_i,n}} \leq C2^{-n} \, .    
\end{equation} 
Now if $n\geq |x|+|\tau|$, then $(\tau,x)\in E_n$, there exists $i$ such that $|x-x^n_i|,|\tau-\tau^n_i|\leq 2^{-n}.$ Thus by~\eqref{eq:z_x,x'_boundedness},~\eqref{eq:z_tau_boundedness},~\eqref{eq:z_x,tau},~\eqref{eq:Z_partition} and triangle inequality, we have that for $T\leq n$: 
$$\|Z^{\tau,T'_n,x}-Z^{\tau,x}\|_{\HH_{\tau,T}}\leq\|Z^{\tau,T'_n,x}-Z^{\tau,x}\|_{\HH_{\tau,n}}\leq C(1+n^2)2^{-n},$$
and then $Z^{\tau,T'_n,x}\to Z^{\tau,x}$,
$d\mathbb P\otimes dt$-a.e.\hspace{-3pt} by the Borel-Cantelli Lemma. We now define the measurable map $\displaystyle\zeta(t,x)\!:=\! \lim_{n\rightarrow\infty} \zeta^{T'_n} (t,x)$ if the limit exists, and $0$ otherwise, so that $Z^{\tau,x}_t\!=\!\zeta(t,X^{\tau,x}_t)$, $d\mathbb{P} \otimes dt-$a.e. 
\end{proof}

\begin{remark}\label{rmk:sigma_compact}
This proof relies on the fact that $\R^d$ is $\sigma$-compact, i.e., a countable union of compact sets. However, an infinite-dimensional Banach space is not $\sigma$-compact. Therefore, this argument applies only to $\R^d$, and not to $\C_\theta$.
\end{remark}

\begin{definition}
Let $\mathcal{P}_2$ be the subset of probability distributions $\mu$ on $\mathbb{R}^d$ satisfying $\mu(|\cdot|^2) \coloneqq
\int_{\R^d} |x|^2\mu(dx)<\infty$. For any $\mu, \nu \in \mathcal{P}_2$, the total variation distance with weight $1+|\cdot|^2$ is defined by:
\begin{equation*}
         \displaystyle d_{\rm{TV},2}(\mu, \nu):=\sup_{\left|\frac{\phi(\cdot)}{1+|\cdot|^2}\right|_\infty\leq 1 }\Big|\int_{\R^d} \phi(x) (\mu-\nu)(dx)\Big| = \int_{\R^d} (1+|x|^2)\, |\mu-\nu|(dx),
\end{equation*}
where $|\mu-\nu|$ denotes the total variation measure associated with the signed measure $\mu-\nu$.
\end{definition}

In the next result, we show that a variant of the state process satisfies a mixing property. 

\begin{proposition}\label{prop:ergodicity}
Let $\Gamma : \mathbb{R} \times \mathbb{R}^d \to \mathbb{R}^d$ be a bounded and measurable function and suppose that $b$ satisfies Assumption \ref{time_inhom_assumption}. Then the equation:

    \vspace{-8pt}
    
    \begin{equation}\label{SDE_variant_bounded_measurable}
    \begin{array}{cc}
         \displaystyle dX^{\tau, x}_t=(b+\Gamma)(t,X_t^{\tau, x})\,dt + dW^{\tau}_t \, , & \displaystyle X_{\tau}^{\tau, x}=x \, ,
    \end{array}
    \end{equation}
    admits a weak solution unique in law for all $\tau \in \mathbb{R}$. Denote by $P_{s,t}$ its transition semi-group. Then $P$ satisfies the mixing property, \textit{i.e.} there exist $C, \eta > 0$ such that, for all $\mu, \nu \in \mathcal{P}_2$, 
    
    \vspace{-5pt}

    \begin{equation}\label{feller_inhom}
        \begin{array}{cc}
        d_{\rm{TV},2}(\mu P_{s, t} \, , \, \nu P_{s,t})\le C \, e^{-\eta (t-s)} \, d_{\rm{TV},2}(\mu, \nu) \, , & s \leq t \in \mathbb{R} \, ,
        \end{array}
    \end{equation}
    Moreover, there is a unique evolution system of measures $(m_t)_{t \in \R}$ for $P$ satisfying $\sup_{t\in\R} m_t(|\cdot|^2)<\infty$. 
    \end{proposition}

The proof is reported in Appendix \ref{app:mixing_prop_sec}. Our main result of this section states that the solution $(Y, Z, \Lambda)$ of the ergodic backward SDE~\eqref{eq:Lambda_unique_up_to_constant} and~\eqref{eq:ebsde_flow_path} is unique in the class of solutions having a Markov representation. The proof adapts standard arguments from homogeneous ergodic backward SDEs, see \textit{e.g.} \citeauthor{furhman_ergodic} \cite{furhman_ergodic}, \citeauthor{hu_lemonier} \cite{hu_lemonier}.

\begin{theorem}\label{thm:uniqueness_v_zeta}
    Let Assumptions \ref{hamiltonian_assumption} and \ref{time_inhom_assumption} hold, and let $v,\tilde v:\mathbb R\times\mathbb R^d\to\mathbb R$, 
$\zeta,\tilde\zeta:\mathbb R\times\mathbb R^d\to\mathbb R^d$, $\Lambda,\tilde\Lambda : \mathbb{R} \to \mathbb{R}$ be measurable maps, with $v,\tilde v$ at most linearly growing in $x$ uniformly in $t$, and define:
   \begin{equation*}
        \begin{array}{ccc}
             \displaystyle (Y^{\tau, x}_t, Z_t^{\tau, x}) \coloneqq (v(t,X_t^{\tau,x}),\zeta(t,X_t^{\tau,x})) & \text{and} & \displaystyle (\tilde{Y}^{\tau, x}_t, \tilde{Z}_t^{\tau, x}) \coloneqq (\tilde{v}(t,X_t^{\tau,x}), \tilde{\zeta}(t,X_t^{\tau,x})) \, .
        \end{array}
    \end{equation*}
Assume that $(Y^{\tau}, Z^{\tau}, \Lambda)_{t\geq\tau}$, $(\tilde{Y}^{\tau}, \tilde{Z}^{\tau}, \tilde{\Lambda})_{t\geq\tau}\in\mathbf{E}_{\tau}(\R^d) \times {\rm{UC}}_{\tau}$ are solutions of the ergodic backward SDE \eqref{eq:ebsde_flow_path} for all $\tau \in\R$.
 
Then, $Z^{\tau, x} = \tilde{Z}^{\tau, x}$, $d\mathbb{P}\otimes dt-$a.e. \hspace{-7pt} and $v-\Lambda=\tilde v-\tilde \Lambda + c$ for some $c \in \mathbb{R}$.
\end{theorem}

\begin{proof}
In this proof, $C$ denotes a positive constant which may change from line to line. Let $t\in\R$ and $x,y\in\R^d$. We first prove that $v(t,x)-v(t,y)=\tilde{v}(t,x)-\tilde{v}(t,y)$. As the measurable function 
$\Gamma(s, x) \coloneqq \frac{H(s,x,\zeta(s,x))-H(s,x,\tilde{\zeta}(s,x))}{|\zeta(s,x)-\tilde{\zeta}(s,x)|^2}(\zeta(s,x)-\tilde{\zeta}(s,x))\mathbf{1}_{\{\zeta(s,x)-\tilde{\zeta}(s,x)\neq 0\}} 
$
is bounded by Assumption \ref{hamiltonian_assumption}, 
\begin{align} \label{v-v'}
    (v-\tilde{v})(t,x)  - (\Lambda- \tilde{\Lambda})(t)  = (v-\tilde{v})(T,X^{t,x}_T) -  (\Lambda- \tilde{\Lambda})(T)  - \int_t^T  (\zeta-\tilde{\zeta})(s, X^{t, x}_s)  \cdot d W_s^{t, \beta^{t, x}} ,
\end{align}
where $\beta^{t,x}_s:=\Gamma(s,X^{t, x}_s)$. Under $\mathbb{P}^{t, \beta^{t,x}}$, $X^{t,x}$ is the unique weak solution of \eqref{SDE_variant_bounded_measurable} with initial data $(t, x)$. Since $v,\tilde{v}$ have at most linear growth, the estimate $|(v - \tilde{v})(t,x)| \leq C(1 + |x|^2)$ holds for some constant $C$. Taking expectations in \eqref{v-v'} and using Proposition~\ref{prop:ergodicity} and Lemma~\ref{lem:girsanov_intZ_is_mg}, we obtain:
\begin{align*}
|(v-\tilde{v})(t,x)-(v-\tilde{v})(t,y)| &=\big|\E^{t, \beta^{t,x}}[(v-\tilde{v})(T,X^{t,x}_T)]-\E^{t, \beta^{t,y}}[(v-\tilde{v})(T,X^{t,y}_T)]\big| \\[5pt]
&\leq C \, d_{\rm{TV},2}(\delta_x P_{t, T} \, , \, \delta_y P_{t, T}) 
%\\[5pt] &
\leq C (1+|x|^2+|y|^2)e^{-\eta (T-t)} \xrightarrow[T\to\infty]{} 0.
\end{align*}
This shows that $(v-\tilde{v})(t,x)$ does not depend on $x$. In particular, by \eqref{v-v'}, the Itô integral
$$
    \int_t^T  (\zeta-\tilde{\zeta})(s, X^{t, x}_s) \cdot d W_s^{t, \beta^{t,x}}  = \big[(v-\tilde{v})(.,x) -  (\Lambda- \tilde{\Lambda})\big]_{t}^T 
$$
is deterministic. As $\mathbb{P}^{t, \beta^{t,x}}\!\!\sim\mathbb{P}$ on $[t, T]$, we conclude that $\zeta(t,X^{x}_t)=\tilde{\zeta}(t,X^x_t)$, $d\mathbb{P} \otimes dt-$a.e. It follows that the function $(v-\tilde{v})(t, x)- (\Lambda-\tilde{\Lambda})(t)$ is constant on $\mathbb{R} \times \mathbb{R}^d$, which proves the claim.
\end{proof}

\subsection{Representation of \texorpdfstring{$\Lambda$}{Lambda} in the finite delay setting}

We now give a representation of the function $\Lambda$ in the case where $b$ and $H$ depend on the past with finite memory.  Our proof is a generalization of \citeauthor{furhman_ergodic} \cite[Corollary~5.9]{furhman_ergodic} and leverages the results on nonlinear PDEs with delay studied in \citeauthor{HJB_delay} \cite{HJB_delay}. Unfortunately, our assumption cannot be adapted to the general case with infinite delay, see Remark \ref{rmk:counterex_feller}. We introduce a new set of notations. Given two Banach spaces $E$ and $F$,
\begin{itemize}
    \item $L(E, F)$ is the space of bounded linear operators from $E$ to $F$, endowed with the usual norm. 

    \item $\mathcal{G}^1(E, F)$ is the space of continuous functions $u : E \rightarrow F$ such that 1) $u$ is Gâteaux differentiable on $E$, with Gâteaux differential at point $x \in E$ denoted by $\nabla u(x) \in L(E, F)$, and 2) for every $h \in E$, the map $x \mapsto \nabla u(x)h$ is continuous. 

    \item For an interval $I\subset\R$, $\mathcal{G}^{0,1}(I\times E, F)$ is the space of continuous maps $v : I \times E \rightarrow F$ such that 1) $v(t, \cdot) \in \mathcal{G}^1(E, F)$ for all $t \in I$, and 2) the map $(t, x) \mapsto \nabla v(t,x)h$ is continuous for all $h \in E$. 
\end{itemize}

Let $r > 0$ and recall that $\mathbf{C}_{r} \coloneqq C([-r, 0], \mathbb{R}^d)$ is a Polish space when endowed with the norm $|\xi|_r = \sup_{t \in [-r, 0]} |\xi_t|$. Note that its dual $\mathbf{C}_{r}^*$ is the space of $d$-tuples of finite Borel measures on $[-r, 0]$. Recall that $P_{s, t}$ is the transition semi-group of the path process $\X$ as defined in Proposition \ref{prop:FSDE_existence}. The total variation distance $d_{\rm{TV}, 0}$ between two measures on $\mathbf{C}_{r}$ is defined by:
\begin{equation*}
    d_{\rm{TV}, 0}(m, m') = \sup_{|\phi|_\infty\leq 1} \Big| \int_{\mathbf{C}_{r}} \phi(\chi) \, (m- m') (d\chi) \Big| \, ,
\end{equation*}
where the supremum is taken over all measurable functions $\phi : \mathbf{C}_{r} \rightarrow \mathbb{R}$ bounded by $1$.

\begin{assumption} \label{ass:diff}
\setcounter{assumptiontag}{0}

    \begin{itemize}
        \item[]

        \refstepcounter{assumptiontag}
        \item[] \hspace{-30pt} (\theassumptiontag) \label{diff_ass} \hspace{2pt} There exist two functions $\overline{b} : \mathbb{R} \times \mathbf{C}_{r} \rightarrow \mathbb{R}^d$ and $\overline{H} : \mathbb{R} \times \mathbf{C}_{r} \times \mathbb{R}^d \to \mathbb{R}$ such that $b(t, \xi) = \overline{b}(t, \xi_{|[-r, 0]})$ and $H(t, \xi, z) = \overline{H}(t, \xi_{|[-r, 0]}, z)$. Moreover, $b$ and $H$ satisfy Assumptions \ref{dissipativity_assumption} and \ref{hamiltonian_assumption}.
        
        \refstepcounter{assumptiontag}
        \item[] \hspace{-30pt} (\theassumptiontag) \label{feller_ass} \hspace{1pt} There exists an evolution system of measures $(m_t)_{t \in \R}$ for transition semi-group $P_{s,t}$ of the path process $\X$ and $\eta > 0$ such that, for all $\xi \in \mathbf{C}_{r}$, there is a constant $C_{\xi}$ satisfying:
        \begin{equation*}
            \begin{array}{cc}
                 \displaystyle d_{\rm{TV}, 0} (\delta_{\xi} P_{s, t} \, ,  m_t) \leq C_{\xi} \,  e^{- \eta (t-s)} \, , & \text{for all } \, s \leq t \, .
            \end{array}
        \end{equation*}
    \end{itemize}
\end{assumption}

A discussion of Assumption~\eqref{feller_ass} is given in Remarks~\ref{rmk:counterex_feller} and~\ref{rmk:markov_H52} below. We simply denote $\xi:=\xi_{|[-r, 0]}$ belongs to $\mathbf{C}_{r}$ and we write $b$ and $H$ instead of $\overline{b}$ and $\overline{H}$. Note that $\X$ is now viewed as a Markov process in $\mathbf{C}_{r}$. In what follows, $(Y^{\tau,T,\xi},Z^{\tau,T,\xi})$ denotes the solution of~\eqref{BSDE_finite_horizon}, and the finite-horizon decoupling field $(v^T,\zeta^T)$ is understood to be defined on $(-\infty,T]\times\mathbf{C}_r$. 
The next result is based on the work of \citeauthor{HJB_delay} \cite{HJB_delay}. Note that $\nabla v^T(t, \xi)$ belongs to $\mathbf{C}_{r}^*$ and $\nabla v^T(t, \xi)(\{0\})$ is the mass at zero of $\nabla v^T(t, \xi)$. 

\begin{lemma} \label{lem:diff}
    Under Assumption~\eqref{diff_ass}, there exists a constant $C>0$ such that $|Z^{\tau,T,\xi}_t| \leq C$, $d\mathbb{P} \otimes dt-$a.e. \hspace{-7pt} for any $ \tau\leq T$ and $\xi\in \mathbf{C}_{r}$.
\end{lemma}

\begin{proof}
    Without loss of generality, we only treat the case $\tau = 0$ and omit it in the notations. We first assume that for all $t \in\R$, $b(t, \cdot) \in \mathcal{G}^1(\mathbf{C}_{r})$ and $H(t, \cdot, \cdot) \in \mathcal{G}^1(\mathbf{C}_{r} \times \mathbb{R}^d)$. Then, by \citeauthor{HJB_delay} \cite[Proposition~4.2]{HJB_delay} the map $v^T$ defined in Proposition \ref{prop:Markov_rep} belongs to $\mathcal{G}^{0,1}([0, T]\times \mathbf{C}_{r})$ for all $T \geq 0$, and we have:
    \begin{equation*}
        \begin{array}{cccc}
             \zeta^T(t, \X_t^{\xi}) =  Z_t^{T, \xi} = \nabla v^T(t, \mathcal{X}^{\xi}_t)(\{0\}) \, , & \text{for all } \,  0 \leq t \leq T \, \text{ and } \, \xi \in \mathbf{C}_{r} \, .
        \end{array}
    \end{equation*}
    By the uniform Lipschitz estimate in Proposition~\ref{prop:Markov_rep},
there exists a constant $C>0$, independent of $t,T,\xi$, such that the total mass of the measure
$\nabla v^T(t,\xi)$ is bounded by $C$. In particular,
$|\nabla v^T(t,\xi)(\{0\})| \le C$, thus $Z^{T, \xi}$ is bounded by $C$.
In the general case when $b$ and $H$ are only Lipschitz continuous on $\mathbf{C}_{r}$ and $\mathbf{C}_{r} \times \mathbb{R}^d$ respectively with constant $L=L_b\vee L_H$, for all $\varepsilon>0$, there exist two measurable functions $\tilde{b}$ and $\tilde{H}$ satisfying $\sup_{t,\xi}|b(t,\xi)- \tilde{b}(t,\xi)| + \sup_{t,\xi,z}|H(t,\xi,z) - \tilde{H}(t,\xi,z)| < \varepsilon$, with $\tilde{b}(t,\cdot)$ and $\tilde{H}(t,\cdot,\cdot)$ being $L$-Lipschitz and Gâteaux differentiable on $\mathbf{C}_{r}$ and $\mathbf{C}_{r} \times \mathbb{R}^d$. This can be proved by adapting the arguments of \citeauthor{Lpz_approx} \cite[Theorem 7]{Lpz_approx}, noting that $\mathbf{C}_{r}$ is a separable space. Let $(\tilde{\X}^{\xi}, \tilde{Y}^{T, \xi}, \tilde{Z}^{T, \xi})$ be the solution of the forward-backward system with parameters $(\tilde{b}, \tilde{H})$. Note that we do not need $\tilde{b}$ to satisfy the dissipativity condition for the well-posedness of $(\tilde{Y}^{T, \xi}, \tilde{Z}^{T, \xi})$. For $T \geq 0$, one has:
    \begin{align}
        \tilde{Y}^{T, \xi}_{0} - Y^{T, \xi}_{0} &= \int_{0}^T \hspace{-7pt} \big( \tilde{H}(s, \tilde{\X}_s^{\xi}, Z^{T, \xi}_s) - \tilde{H}(s, \X_s^{\xi}, Z^{T, \xi}_s) + \tilde{H}(s, \X_s^{\xi}, Z^{T, \xi}_s) - H(s, \X_s^{\xi}, Z^{T, \xi}_s) \big) \, ds \nonumber \\
        &\quad \quad \quad - \int_{0}^T (\tilde{Z}^{T, \xi}_s - Z^{T, \xi}_s) \cdot dW^{\beta}_s \, , \nonumber
    \end{align}
    where $\beta_s = \frac{\tilde{H}(s, \tilde{\X}_s^{\xi}, \tilde{Z}^{T, \xi}_s) - \tilde{H}(s, \tilde{\X}_s^{\xi}, Z^{T, \xi}_s)}{|\tilde{Z}^{T, \xi}_s - Z^{T, \xi}_s|^2} \mathbf{1}_{\{ \tilde{Z}^{T, \xi}_s \neq Z^{T, \xi}_s \}}$ is bounded by $L$. Taking the expectation $\mathbb{E}^{\beta}$ by Lemma~\ref{lem:girsanov_intZ_is_mg} and applying Proposition \ref{prop:stability_b} with the kernel $K=\mathbf{1}_{[0,T)}$ and $\delta b=\varepsilon$, we obtain a constant $C_1 > 0$ such that:
    \begin{align}\label{eq:difference_v^T}
        |v^T(0, \xi) - \tilde{v}^T(0, \xi)| \, \leq \,  C_1 \varepsilon (\sqrt{T} + T) \, \leq \,  2C_1\varepsilon \big(T\vee 1 \big) \, . 
    \end{align}
    Following a similar argument as in the proof of Theorem \ref{thm:existence_uniqueness}, we can write by Itô's formula:
    \begin{align}
       \mathbb{E} \int_{0}^T |\tilde{Z}^{T, \xi}_s - Z^{T, \xi}_s|^2 \, ds &\leq 2\mathbb{E} \int_{0}^T (\tilde{Y}^{ T, \xi}_{s} - Y^{T, \xi}_{s}) \big( \tilde{H}(s, \tilde{\X}^{\xi}, \tilde{Z}^{T, \xi}) - \tilde{H}(s, \tilde{\X}^{\xi}, Z^{T, \xi}) \big) \, ds  \nonumber \\
        &  \quad \quad \quad +2 \mathbb{E} \int_{0}^T (\tilde{Y}^{T, \xi}_{s} - Y^{T, \xi}_{s}) \big(\tilde{H}(s, \tilde{\X}_s^{\xi}, Z^{T, \xi}_s) - \tilde{H}(s, \X_s^{\xi}, Z^{T, \xi}_s) \big) \, ds \nonumber \\
        & \quad \quad \quad +2 \mathbb{E} \int_{0}^T (\tilde{Y}^{T, \xi}_{s} - Y^{T, \xi}_{s}) \big(\tilde{H}(s, \X_s^{\xi}, Z^{T, \xi}_s) - H(s, \X_s^{\xi}, Z^{T, \xi}_s) \big) \, ds \nonumber
    \end{align}
    Note that, similarly to \eqref{eq:difference_v^T}, we have $|\tilde{Y}^{T, \xi}_{s} - Y^{T, \xi}_{s}| \leq 2 C_1 \varepsilon (T\vee 1 )$. By Young's inequality and the same computations as above, we obtain a constant $C_2 > 0$ such that:
    \begin{equation*}\label{eq:Z_estimate_Y_finite_delay}
    \mathbb{E} \int_{0}^T \big|\tilde{Z}^{T, \xi}_s - Z^{T, \xi}_s \big|^2 ds \leq C_2 \varepsilon^2 \big(T^3\vee 1 \big) \, .
    \end{equation*}
    Replacing $\varepsilon$ by a sequence $\varepsilon_n \rightarrow 0$, we can construct a family $(Z^n)_{n}$ of random processes bounded by $C$ converging to $Z^{T, \xi}$ in $\mathbb{H}^2_{T}(\R^d)$. We conclude that $Z^{T, \xi}$ is also bounded by $C$.
\end{proof}

\vspace{5pt}
The following result provides a canonical representation of the function $\Lambda$.

\begin{theorem}\label{thm:canonical_representation}
     Assume that there exists a measurable function $\zeta : \R \times \mathbf{C}_r \to \R^d$ such that $Z^{\tau, \xi}_t= \zeta(t, \X^{\tau, \xi}_t)$ for any $t \geq \tau$ and $\xi \in \mathbf{C}_r$, where $(Y^{\tau}, Z^{\tau}, \Lambda)\in\mathbf{E}_{\tau}(\mathbf{C}_r) \times {\rm{UC}}_{\tau}$ is the solution of the ergodic backward SDE constructed in Theorem \ref{thm:existence_uniqueness} and Remark~\ref{rmk:ebsde_flow_path}. Let
     \begin{equation*}\label{eq:lambda_rep}
         \begin{array}{cc}
              \displaystyle \overline{\Lambda}(T) = \int_0^T \int_{\mathbf{C}_{r}} H(t, \chi, \zeta(t, \chi)) \,  m_t(d\chi) dt 
              ~~\mbox{and}~~
              \overline{Y}^\tau_T:=Y^\tau_T-\Lambda(T)+\overline{\Lambda}(T),
              & \, T \in \R \, .
         \end{array}
     \end{equation*}
Then, under Assumption \ref{ass:diff}, $(\overline Y,Z,\overline\Lambda)\in\mathbf{E}_{\tau}(\mathbf{C}_r) \times {\rm{UC}}_{\tau}$ is a solution of the ergodic backward SDE.
\end{theorem}

\begin{proof}
    Following Remark \ref{rmk:smooth_repre}, we only need to prove that $\Lambda - \overline{\Lambda}$ is bounded. By construction, $Z^{\tau, T_n, \xi} \rightarrow Z^{\tau, \xi}$ in $\mathbb{H}^2_{\tau, T}(\mathbb{R}^d)$ for all $T\geq\tau$, so we deduce by Lemma \ref{lem:diff} that $|Z^{\tau, \xi}| \leq C$ for some constant $C >0$. Thus, by Assumption \ref{hamiltonian_assumption}, the process $H(\cdot, \X^{\tau, \xi}_\cdot, \zeta(\cdot, \X^{\tau, \xi}_\cdot))$ is bounded by $L_H(C+1)$, $d\mathbb{P} \otimes dt-$a.e. By the mixing property \eqref{feller_ass}, $\delta_{\xi} P_{\tau, t} \rightarrow m_t$ in total variation as $\tau \to -\infty$, so the function $(t,\chi) \mapsto H(t, \chi, \zeta(t, \chi))$ is bounded by $L_H(C+1)$, $dm_t \otimes dt-$a.e. Moreover, for  $T\geq 0$ and $\xi \in \mathbf{C}_{r}$, $\mathbb{E} [Y_0^{0,\xi} - Y_T^{0,\xi}] - \Delta_{\rm{TH}} = \overline{\Lambda}(T) - \Lambda(T)$, where:
    \begin{equation*}
        \Delta_{\rm{TH}} = \mathbb{E} \int_0^T \Big( H \big(t, \mathcal{X}^{0,\xi}_t, \zeta(t, \mathcal{X}_t^{0,\xi})\big) - \int_{\mathbf{C}_{r}} H(t, \chi, \zeta(t, \chi)) \,  m_t(d\chi) \Big) \, dt \, .
    \end{equation*}
    The term $\mathbb{E} [Y_{0}^{\xi} - Y_T^{\xi}]$ is bounded since $Y$ has linear growth in $\X$ and by Proposition~\ref{prop:forward_holder}. Moreover,
    \begin{align*}
        \Delta_{\rm{TH}} \, \leq \,  \int_0^T\Big|\int_{\mathbf{C}_r} H(t,\chi, \zeta(t, \chi)) \, \delta _{\xi}P_{0,t}(d\chi) - \int_{\mathbf{C}_r} H(t,\chi, \zeta(t, \chi)) \, m_t(d\chi) \Big|  \, dt   \, \leq \,  C'_{\xi} \int_0^T e^{-\eta t} \, dt \,\leq \,  \frac{C'_{\xi}}{\eta}  \, ,
    \end{align*}
    where we used \eqref{feller_ass} and $C'_\xi=C_\xi L_H(C+1)$. For $T\leq0$, a similar proof works if we consider the solution of ergodic backward SDE starting at time $T$ and estimate the integral from $T$ to $0$ instead.
\end{proof}

\begin{remark}
    In Theorem \ref{thm:canonical_representation}, we assume the existence of a Markovian representation function $\zeta$ of $Z$. By Proposition \ref{prop:ebsde_flow_uniqueness} and Remark~\ref{rmk:sigma_compact}, the assumption may not be ensured generally for path-dependent case. However, if for all $t \in\R$, $b(t, \cdot) \in \mathcal{G}^1(\mathbf{C}_{r})$ and $H(t, \cdot, \cdot) \in \mathcal{G}^1(\mathbf{C}_{r} \times \mathbb{R}^d)$, we can show, using similar arguments as in \citeauthor{furhman_ergodic} \cite[Theorem~5.1]{furhman_ergodic}, along with \citeauthor{HJB_delay} \cite[Theorem~3.1]{HJB_delay}, that $v \in \mathcal{G}^{0,1}([\tau, T] \times \mathbf{C}_r)$ for all $\tau \leq T \in \R$, and $Z^{\tau, \xi}_t = \nabla v(t, \X^{\tau, \xi}_t)(\{0\})$, $d\mathbb{P} \otimes dt-$a.e. In this case $\zeta(t,x)=\nabla v(t,x)(\{0\})$.
\end{remark}

\begin{remark}\label{rmk:counterex_feller}
    The existence of an evolution measure family $(m_t)_{t \in \R}$ is given in Proposition \ref{prop:FSDE_existence}. However, it turns out that the mixing property of Assumption \eqref{feller_ass} cannot hold for systems with infinite delay with $\xi \in \C_{\theta}$. If there exists $(m_t)_{t\in\R}$ satisfying Assumption \eqref{feller_ass}, define the function $\phi(\xi) = {\rm{sgn}} (\lim_{s \to - \infty} e^{\theta s} \xi_s)$. Let $\xi(s) = -e^{\theta s}$ and $\xi'(s) = e^{\theta s}$. Then, for all $s \leq t$, noting that $P_{s,t}[\phi](\xi) = \mathbb{E}[\phi(\X_t^{s, \xi})] = -1$ and $P_{s,t}[\phi](\xi') = \mathbb{E}[\phi(\X_t^{s, \xi'})] = 1$, we have:
    \begin{equation*}
    \begin{array}{ccc}
         \displaystyle |1+m_t(\phi) | \, \le \, d_{\rm{TV},0} (\delta_{\xi} P_{s, t},  m_t) \, \le \, C_{\xi} \, e^{-\eta (t-s)},\quad  |1 -m_t(\phi) | \, \le \, d_{\rm{TV},0} (\delta_{\xi'} P_{s, t},  m_t) \, \le \, C_{\xi'} \, e^{-\eta (t-s)} \, .
    \end{array}
    \end{equation*}
    Sending $s$ to $-\infty$, we obtain $m_t(\phi) = -1$ and $m_t(\phi) = 1$ respectively, which is contradiction. Note that the exponential mixing property for a time-homogeneous system with infinite delay was proved for the Wasserstein distance in \citeauthor{Ergodicity_neutral_SDE} \cite{Ergodicity_neutral_SDE}, but applying this result would require $\zeta$ to be at least Lipschitz continuous, which we believe is too strong in this context. For systems with finite delay, the condition \eqref{feller_ass} was studied in \citeauthor{Expo_mixing_finite_delay} \cite{Expo_mixing_finite_delay} under different assumptions. 
\end{remark}

\begin{remark}\label{rmk:markov_H52}
    Assumption~\eqref{feller_ass} always holds when the state dynamics is Markovian. More precisely, let $b(t, \xi) = \overline{b}(t, \xi_0)$, where $\overline{b}$ satisfies Assumption \ref{time_inhom_assumption}. By Proposition \ref{prop:ergodicity}, there exists a unique evolution system of measures $(m_t)_{t \in \R}$ satisfying $\sup_{t \in \R} m_t(|\cdot|^2) < \infty$. Let $\tilde{X}^{\tau}$ be the unique weak solution of \eqref{SDE_variant_bounded_measurable} on $[\tau, \infty)$ with $\Gamma \equiv 0$ and $\tilde{X}^{\tau}_{\tau} \sim m_{\tau}$ and define $\tilde{m}_t = \mathcal{L}(\tilde{\X}^{\tau}_t)$, which is a probability measure on $\mathbf{C}_r$. Then, with a slight abuse of notation for $d_{\rm{TV},0}$, we have, for all $t\geq \tau$,
    \begin{equation*}
        d_{\rm{TV},0}(\mathcal{L}(\X^{\tau,\xi}_{t+r}),\tilde{m}_{t+r}) \, = \, d_{\rm{TV},0}(\mathcal{L}(X^{\tau,\xi}_{t}),m_{t}) \, \leq \, d_{\rm{TV},2}(\mathcal{L}(X^{\tau,\xi}_{t}),m_{t}) \, ,
    \end{equation*}
    The first identity follows because, on the one hand, $\mathcal{L}(X^{\tau,\xi}_{t})$ is the projection of $\mathcal{L}(\X^{\tau,\xi}_{t+r})$ on $\mathcal{P}(\R^d)$ at time $t$ ; on the other hand, $\mathcal{L}(\X^{\tau,\xi}_{t+r})$ is generated by $\mathcal{L}(X^{\tau,\xi}_{t})$ together with the transition kernel of the state dynamics. Finally, by Proposition~\ref{prop:ergodicity} and the fact that $m_{\tau} P_{\tau, t} = m_t$, we have $d_{\rm{TV},2}(\mathcal{L}(X^{\tau,\xi}_{t}),m_{t})\leq C_\xi e^{-\eta(t-\tau)}$. We provide an explicit example below.
\end{remark}

\begin{example}\label{ex:OU_mixing}
    Let $X^{s, x}$ be a time-inhomogeneous Ornstein-Uhlenbeck process satisfying the dynamics \eqref{eq:OU_non_homog} with initial condition $(s, x) \in \mathbb{R} \times \mathbb{R}$. Suppose that $\eta(t) \geq \eta_0$ for some $\eta_0 > 0$. By Example \ref{ex:OU_evolution_system}, we have $\delta_x P_{s, t} = \mathcal{N}(x \, e^{-S_{s, t}}, \Sigma_{s, t})$ and $m_t = \mathcal{N}(0, \Sigma_{-\infty, t})$ for $s \leq t$. By Pinsker's inequality, one has $d_{\rm{TV}, 0}(\delta_x P_{s, t} \, ,   m_t) \leq \sqrt{2 \, \text{D}_{\text{KL}}(\delta_x P_{s, t} \, ,  m_t)}$, where $\text{D}_{\text{KL}}(P, P') \coloneqq \mathbb{E}^{P}[\ln(\frac{dP}{dP'})]$ is the Kullback–Leibler divergence of $P$ with respect to $P'$. By Gaussian calculus, we obtain:

    \vspace{-8pt}
    
    \begin{equation*}
        2 \, \text{D}_{\text{KL}}(\delta_x P_{s, t} \, , m_t) = \log \big(\frac{\Sigma_{-\infty, t}}{\Sigma_{s, t}} \big) + \frac{\Sigma_{s, t} + |x|^2 e^{-2 S_{s, t}}}{ \Sigma_{-\infty, t}} - 1 \, .
    \end{equation*}
    Note that $S_{s, r} \geq \eta_0 (r-s)$ if $r \geq s$ and $S_{s, r} \leq \eta_0(r-s)$ if $r \leq s$. On one hand, because $S_{r, t} = S_{s, t} - S_{s, r}$, we have:
    \begin{equation*}
        \frac{\Sigma_{-\infty, t}}{\Sigma_{s, t}} \, = \, 1 + \frac{\int_{-\infty}^s e^{-2 S_{r, t}} \, dr}{\int_s^t e^{-2 S_{r, t}} \, dr} \,  = \,  1 + \frac{\int_{-\infty}^s e^{2 S_{s, r}} \, dr}{\int_s^t e^{2 S_{s, r}} \, dr} \, \leq \,  1 + \frac{1}{e^{2 \eta_0 (t-s)} - 1} \, .
    \end{equation*}
    On the other hand, as $\Sigma_{s, t} \leq \Sigma_{-\infty, t}$, we obtain:
    \begin{align}
        \frac{\Sigma_{s, t} + |x|^2 e^{-2 S_{s, t}}}{ \Sigma_{-\infty, t}} - 1 \, \leq \,  \frac{|x|^2 e^{-2 S_{s, t}}}{ \Sigma_{-\infty, t}} \leq \frac{|x|^2}{\int_{-\infty}^t e^{2 S_{s, r}} \, dr} \, \leq \,  \frac{|x|^2}{\int_s^t e^{2 S_{s, r}} \, dr} \, \leq \,  \frac{\eta_0 |x|^2}{e^{2 \eta_0 (t-s)} - 1} \, . \nonumber 
    \end{align}
    Finally, using the inequality $\log(1 + y) \leq y$, and posing $C=\sqrt{\frac{1+\eta_0}{1-e^{-2\eta_0}}},$ we have for $t-s\geq 1$:
    \begin{equation*}
        d_{\rm{TV}, 0}(\delta_x P_{s, t} \, , m_t) \leq C(1 + |x|) e^{-\eta_0 (t-s)} \, .
    \end{equation*}
\end{example}

\appendix
\section{Appendix}
\subsection{Kernels and ergodic values} \label{app:ergodic_values}

In the following, we discuss the relationship between the ergodic values of Definition \ref{def:values}. We start with two intermediate results on the properties of the kernels.

\begin{lemma}\label{lemma:upper-sandwich}
Let $K\in\K$, and let
$\lambda\in L^1_{\mathrm{loc}}(\R_+)$ be bounded from below. Then:
$$\displaystyle
\limsup_{\rho\downarrow0}\int_0^\infty K_{\rho}(t)\lambda(t)\,dt
\, \le \,
\limsup_{T\to\infty}\frac{1}{T}\int_0^T\lambda(t)\,dt \, .$$
\end{lemma}

\begin{proof}
Subtracting a lower bound of $\lambda$ from $\lambda$ shifts both sides by
the same constant. We may therefore assume that $\lambda\geq0$. Set
$A(T):=\frac{1}{T}\int_0^T\lambda(t)\,dt$ and
$L:=\limsup_{T\to\infty}A(T)$. If $L=+\infty$, there is nothing to prove.
Otherwise, for $\varepsilon>0$, there exists $T_\varepsilon>0$ such that
$A(T)\leq L+\varepsilon$ for all $T\geq T_\varepsilon$. Since $K\in\K$, the measure $\mu:=-dK$ is a positive Stieltjes measure on
$\R_+$. Moreover, integration by parts gives
$\int_0^\infty s\,\mu(ds)=\int_0^\infty K(t)\,dt=1$. Thus,
$\nu(ds):=s\,\mu(ds)$ defines a probability measure on $\R_+$. Since
$\lambda\geq0$, Tonelli's theorem yields
$$\displaystyle
\int_0^\infty K_{\rho}(t)\lambda(t)\,dt
=
\int_0^\infty \rho\lambda(t)\int_{(\rho t,\infty)}\mu(ds)\,dt
=
\int_0^\infty A\!\left(\frac{s}{\rho}\right)\nu(ds) \, .$$
Splitting the last integral at $\rho T_\varepsilon$ and using the
monotonicity of $T\mapsto TA(T)=\int_0^T\lambda(t)\,dt$, we obtain
\begin{align*}
\int_0^\infty K_{\rho}(t)\lambda(t)\,dt
\leq
\rho T_\varepsilon A(T_\varepsilon)
\mu\bigl((0,\rho T_\varepsilon]\bigr)
+
(L+\varepsilon)\nu\bigl((\rho T_\varepsilon,\infty)\bigr)\leq
\rho T_\varepsilon A(T_\varepsilon)K(0)+L+\varepsilon \, .
\end{align*}
Letting $\rho\downarrow0$ and then $\varepsilon\downarrow0$ yields
$\displaystyle
\limsup_{\rho\downarrow0}\int_0^\infty K_{\rho}(t)\lambda(t)\,dt
\leq
L \, ,$
which proves the result.
\end{proof}

\vspace{5pt}

\begin{proof}[Proof of Proposition \ref{prop:relation_value}]
 By Theorem~\ref{thm:existence_uniqueness}, we may choose $\Lambda\in C^\infty$ with all derivatives bounded and $\Lambda(0)=0$. Applying Lemma~\ref{lemma:upper-sandwich} to
$\lambda(t)=\Lambda'(t)$ and $\lambda(t)=-\Lambda'(t)$ yields the required inequalities. Therefore, $(1)\Rightarrow(2)$, and $(2)\Rightarrow(3)$ by definition. We now prove that $(3)\Rightarrow(1)$. Define $s(t):= \frac{\Lambda(t)}{t}$ for $t>0$. Then $s$ is bounded by $|\Lambda'|_\infty$ and 
\begin{align*}
|s(t)-s(\rho t)| \, \le \,  \Big|\frac{\Lambda(t)-\Lambda(\rho t)}{t}\Big|+\Big|\frac{\Lambda(\rho t)}{\rho t}\Big|(1-\rho) \, \le \,  |\Lambda'|_\infty(1-\rho)+|s|_\infty(1-\rho)
\xrightarrow[\rho\to 1]{}
0.
\end{align*}
Let $K\in\K_W$ and define $k(t):=-tK'(t)$. Since $K\in \K_W\subset \K\cap {\rm AC}$, the function $K$ is nonincreasing and integrable on $\R_+$. Thus $k$ is nonnegative, and
$
0\leq t K(t)\le 2\int_{t/2}^t K(s)\,ds\to0
$
as $t\to\infty$. By integration by parts, we obtain
$
\int_0^\infty k(t)\,dt = \int_0^\infty K(t)\,dt = 1.
$ The claim $(3)\Rightarrow(1)$ is then a consequence of Theorems 3.4 and 8.1 in \citeauthor{Korevaar2004}~\cite[Chapter~II]{Korevaar2004}. 
\end{proof}

\subsection{Some useful results}\label{appendix}

We prove here three technical results from Stochastic Calculus and Real Analysis which we use in the paper. These results are classical but we recall them for completeness. For simplicity we set the time origin to $\tau =0$ and omit it in the notations, but they can be extended easily to general $\tau\in\R$.

\begin{lemma}\label{lem:girsanov_intZ_is_mg}
Let $\beta\in \mathbb{H}^0(\R^d)$ be a bounded process and $Z\in\mathbb{H}^2_T(\R^d)$ for the probability measure $\mathbb{P}$ and all $T>0$. Then the process $M_t:=\int_0^t Z_s\,dW_s^{\beta}$ is a $\mathbb{P}^{\beta}$-martingale.
\end{lemma}

\begin{proof}
By Girsanov's Theorem, $M$ is a $\mathbb{P}^{\beta}$-local martingale, where $\frac{d\mathbb{P}^\beta}{d\mathbb{P}}=\mathcal{E}_\cdot(\beta) \coloneqq \exp( \int_0^\cdot \beta_s \cdot dW_s - \frac{1}{2} \int_0^\cdot |\beta_s|^2 ds)$. Moreover, $\E[\mathcal{E}_T(\beta)^2]\leq e^{||\beta||_\infty T}<\infty.$
By the Burkholder--Davis--Gundy inequality with $p=1$ and the Cauchy--Schwarz inequality, there exists a constant $C>0$ such that:
$$
\mathbb E^{\beta}\big[\sup_{0\le r\le T}| M_r|\big]
\le C\,\mathbb E^{\beta}\Big[\Big(\int_0^T |Z_s|^2\,ds\Big)^{1/2}\Big]
\le C\,\big(\mathbb E[\mathcal{E}_T(\beta)^2]\big)^{1/2}\Big(\mathbb E\int_0^T |Z_s|^2\,ds\Big)^{1/2}<\infty.  
$$
Hence the family $\{M_\tau:\tau\!\le\!T, \tau \text{ stopping time}\}$ is uniformly integrable, implying the required result.
\end{proof}

\begin{lemma}\label{lem:mollification}
For every $\Lambda\in {\rm{UC}}$, there exists a function $\overline{\Lambda}\in C^\infty$ such that $\overline\Lambda(0)=0$,
$\overline{\Lambda}-\Lambda$ and $\overline{\Lambda}^{(k)}$ are bounded for every $k\ge 1$, where $\overline{\Lambda}^{(k)}$ denotes the $k$-th derivative of $\overline{\Lambda}$.
\end{lemma}

\begin{proof}
By abuse of notation, we still denote by $\Lambda$ the even extension of $\Lambda$ to $\R$, which is also uniformly continuous. 
Let $\eta$ be a mollifier satisfying $\eta\in C^\infty(\R)$, $supp(\eta)\subset[-1,1]$, and $\int_\R \eta(x)\,dx=1$. Define:
\begin{equation*}
    \begin{array}{ccc}
         \displaystyle \overline{\Lambda} \coloneqq \Lambda*\eta=\int_\R \Lambda(\cdot-y)\eta(y)\,dy \, & \text{and} & \displaystyle C_\Lambda \, \coloneqq \hspace{-10pt} \sup_{x,x'\in\R,\ |x-x'|\leq 1}|\Lambda(x)-\Lambda(x')| \, .
    \end{array}
\end{equation*}
By the properties of the convolution, $\overline{\Lambda}\in C^\infty(\R)$. Moreover, since $\int_\R \eta(y)\,dy=1$ and $\int_\R \eta'(y)\,dy=0$, we obtain the two inequalities:
\begin{equation*}
    \begin{array}{l}
         \displaystyle |\Lambda(x)-\overline{\Lambda}(x)| = \Big|\int_\R \big(\Lambda(x)-\Lambda(x-y)\big)\eta(y)\,dy \Big| \leq C_\Lambda\int_\R |\eta(y)|\,dy \, , \\[15pt]
         \displaystyle |\overline{\Lambda}'(x)| = |\Lambda*\eta'(x)| = \Big|\int_\R \big(\Lambda(x)-\Lambda(x-y)\big)\eta'(y)\,dy \Big| \leq C_\Lambda\int_\R |\eta'(y)|\,dy \, .
    \end{array}
\end{equation*}
Similarly, for all $k \ge 2$, $\overline{\Lambda}^{(k)}$ is bounded by $C_\Lambda\int_\R |\eta^{(k)}(y)|\,dy$. Finally, we may replace $\overline{\Lambda}$ by $\overline{\Lambda}-\overline{\Lambda}(0)$ to ensure that $\overline{\Lambda}(0)=0$.
\end{proof}

\begin{lemma}\label{lem:cauchy_equation}
Let $(a_n)_{n\ge1}$ be a sequence in $\mathbb{R}$ satisfying $\sup_{m,n\ge 1}
|a_{n+m}-a_n-a_m|<\infty$.
Then there exists $\lambda\in\mathbb{R}$ such that $\sup_{n\ge1}|a_n-\lambda n|<\infty$. In particular, $\frac{a_n}{n}\xrightarrow[n\to\infty]{}\lambda$.
\end{lemma}

\begin{proof}
The claim is a discrete time version of the classical result of \citeauthor{Hyers1941}~\cite{Hyers1941}. Denoting $C:=\sup_{m,n\ge 1}
|a_{n+m}-a_n-a_m|$, we compute directly that:
\[
|a_{mn}-m a_n| \, =\, \Big|\sum_{j=0}^{m-1}\big(a_{(j+1)n}-a_{jn}-a_n\big)\Big|
\, \le \,  \sum_{j=0}^{m-1}|a_{(j+1)n}-a_{jn}-a_n| \, \le \, (m-1)C \, .
\]
Exchanging the roles of $m$ and $n$, we also have $|a_{mn}-n a_m|\le (n-1)C$, which implies that:
\[
|m a_n-n a_m|
\le |m a_n-a_{mn}|+|a_{mn}-n a_m|
\le (m+n-2)C \, .
\]
Dividing by $mn$, we see that $\left|\frac{a_n}{n}-\frac{a_m}{m}\right|
\le \frac{C}{n}+\frac{C}{m}$, which shows that $\left(\frac{a_n}{n}\right)_{n\ge1}$ is a Cauchy sequence in $\mathbb R$, so there exists $\lambda\in\mathbb R$ such that $\frac{a_n}{n}\to\lambda$. Finally, sending $m\to\infty$ in the last estimate induces the desired result by taking the supremum over $n\ge1$.
\end{proof}

\subsection{Mixing property for time-inhomogeneous Markov processes}\label{app:mixing_prop_sec}

In this section we use the notations introduced in Section \ref{sec:time_inhom} and we consider the stochastic differential equation \eqref{SDE_variant_bounded_measurable}, where $b,\Gamma: \mathbb{R} \times \mathbb{R}^d \to \mathbb{R}^d$ are measurable functions, with $\Gamma$ bounded and $b$ satisfying Assumption \ref{time_inhom_assumption}. In the remaining of the section, we denote by $B_R$ the ball of $\mathbb{R}^d$ of radius $R > 0$ and centered at $0$. The proof of Proposition \ref{prop:ergodicity} is based on the results of \citeauthor{hairer_mattingly_harris_2011}~\cite{hairer_mattingly_harris_2011} for one-step Markov semi-groups, and on the following density estimate for solutions of time-inhomogeneous SDEs, derived by \citeauthor{MENOZZI2021330}~\cite{MENOZZI2021330}. 

\begin{lemma} \label{lem:menozzi}
    Under the assumptions of Proposition \ref{prop:ergodicity}, the stochastic differential equation \eqref{SDE_variant_bounded_measurable} admits a weak solution unique in law for all time origin $\tau \in \mathbb{R}$. Denote by $p(s, x, t, y)$ its transition densities. Then under~\eqref{dissipativity_condition_inhom}, for all $R > 0$, there exist a non-negative measurable function $\varepsilon_R(s,t)$ and a probability measure on $B_R$ with density $q_R$ such that $\inf_{s\in\R} \varepsilon_R(s,s+1)>0$ and \begin{equation*}
        \begin{array}{cc}
             \displaystyle p(s, x, t, \cdot) \geq \varepsilon_R(s,t) \, q_R(\cdot) \, , & \text{for all } \,  x \in B_R \, , \;  s \leq t \in \R \, .
        \end{array}
    \end{equation*}
\end{lemma}

\proof
    Note that the SDE \eqref{SDE_variant_bounded_measurable} with $\Gamma \equiv 0$ has a strong solution which is path-wise unique, thus unique in law. Therefore, since $\Gamma$ is a bounded measurable function, the existence of a weak solution to \eqref{SDE_variant_bounded_measurable} and its uniqueness in law are a direct consequence of Girsanov's Theorem. Let $\rho$ be a nonnegative smooth function with support in $B_1$ satisfying $\int_{\R^d} \rho = 1$ and define $\tilde{b}(t,x) \coloneqq [(b+\Gamma)(t, \cdot) * \rho](x)$. In other words, $\tilde{b}$ is a mollification of the drift $b + \Gamma$ in the space variable, it is smooth and has linear growth. In \cite[Theorem~1.2]{MENOZZI2021330}, \citeauthor{MENOZZI2021330} establish an estimate for the transition densities $p(s,x,t,y)$, namely there exist $\varepsilon, \delta > 0$ such that:
    \begin{equation*}
        \begin{array}{cc}
             \displaystyle p(s, x, t, y) \geq \varepsilon \, g_{\delta}(t-s, \theta_{t,s}(x) -y) \, , & \text{for all } \, x, y \in \R^d \, , \;  t\geq s \, ,
        \end{array}
    \end{equation*}
    where $g_{\delta}(t, x) = t^{-d/2} \exp(- \frac{\delta|x|^2}{t})$ and $\theta$ is a deterministic flow solving:
    \begin{equation*}
        \begin{array}{cc}
             \dot{\theta}_{t,s}(x) = \tilde{b}(t, \theta_{t,s}(x)) \, , \; t \geq s \, , & \theta_{s,s}(x) = x \,.
        \end{array}
    \end{equation*}
    Let $C > 0$ denote a constant which we may adjust if needed. Since $|\tilde{b}(t,x)| \leq C(1+|x|)$, we obtain by Gronwall's Lemma that $|\theta_{t,s}(x)| \leq C(1 + |x|) e^{C(t-s)}$. Therefore, for all $R > 0$, the map $(x,y) \mapsto \theta_{t,t+1}(x)-y$ is bounded for $B_R\times B_R$, uniformly in $t$, by some constant $\theta_R$. Then we deduce that the claim holds with $q_R$ chosen as the density of the uniform distribution on $B_R$, and with
\begin{equation*}\label{eq:eta_q_def}
\varepsilon_R(s,t) \coloneqq \frac{\varepsilon}{|B_R|} g_{\delta}(t-s,\theta_R)
\geq
\frac{\varepsilon}{|B_R|} g_{\delta}(1,\theta_R)\mathbf{1}_{\{t-s=1\}} .
\end{equation*}

\vspace{-8mm}
\qed

\begin{proof}[Proof of Proposition \ref{prop:ergodicity}]
Note that this proof can also be found in a recent paper of \citeauthor{entrance_measure_inhomogeneous} \cite[Corollary~2.5]{entrance_measure_inhomogeneous} in a more general framework. However, given its recent release and the numerous technical assumptions made by the authors, we decided to write our own version to facilitate the readability. We proceed in three steps. First we leverage the dissipativity condition to obtain a Lyapunov function for the semi-group $P$. Then we prove the mixing property and finally we construct an evolution system of measure $(m_t)_{t \in \R}$ and show the uniqueness under the integrability assumption. \\

\vspace{-8pt}

\noindent\underline{Step 1:} We first observe that $|\cdot|^2$ is a Lyapunov function for the semi-group $P$, in the sense that there is a constant $K>0$ such that:

\vspace{-8pt}

\begin{equation}\label{eq:lyapunov_P}
\begin{array}{cc}
    P_{s,t} |\cdot|^2  \le e^{-\frac{\eta_0}{2} (t-s)} |\cdot|^2 + K \, , & s \leq t \in \R \, ,
\end{array}  
\end{equation}
Indeed, denote by $\mathcal{L}_t$ the infinitesimal generator of this transition semi-group. Then the dissipativity condition of Assumption \eqref{dissipativity_condition_inhom} and Young's inequality imply that:
\begin{equation*}
    \mathcal{L}_t(|\cdot|^2)(x) \, = \, 2 \, x \cdot \bigl(b(t,x)+\Gamma(t,x)\bigr) + 1
\, \le \,  -\eta_0 |x|^2 + 2\, x \cdot \bigl(b(t,0)+\Gamma(t,x)\bigr) + 1 \, \le \, -\frac{\eta_0}{2} (|x|^2-K)  \, ,
\end{equation*}
with $K=\frac{4}{\eta_0^2}|b(\cdot,0)+\Gamma(\cdot,\cdot)|_\infty^2
+\frac{2}{\eta_0}$. We finally obtain \eqref{eq:lyapunov_P} after applying Gronwall's Lemma. \\

\vspace{-8pt}

\noindent\underline{Step 2:} Given $\beta>0$, we define another distance between two probability measures:
\begin{equation*}
         \displaystyle d_{\beta}(\mu, \nu):=\sup_{\left|\frac{\phi(\cdot)}{1+\beta|\cdot|^2}\right|_\infty\leq 1 }\Big|\int_{\R^d} \phi(x) (\mu-\nu)(dx)\Big| = \int_{\R^d} (1+\beta|x|^2)\, |\mu-\nu|(dx) \,,
\end{equation*}
The space $\mathcal{P}_2$ is complete when endowed with the distance $d_\beta $ and closed under the semi-group $P$ by \eqref{eq:lyapunov_P}. Let $R > \frac{2K}{1-e^{-\eta_0 /2}}$. By Lemma \ref{lem:menozzi} and a slight adaptation of \citeauthor{hairer_mattingly_harris_2011} \cite[Theorem 3.1]{hairer_mattingly_harris_2011}, see also \citeauthor{entrance_measure_inhomogeneous} \cite[Lemma~2.2]{entrance_measure_inhomogeneous}, we have, for all $\mu, \nu \in \mathcal{P}_2$ and $s < t \in \R$,
$$
        %\hspace{-10pt} 
        \displaystyle d_{\beta}(\mu P_{s,t}, \nu P_{s,t}) \leq \alpha(s,t) d_{\beta}(\mu, \nu),
        ~\text{where}~
        \displaystyle \alpha(s,t) \coloneqq \max \Big\{ 1 \!-\! \varepsilon_R(s,t)\!+\! \beta K \, , \frac{2 + \beta( R e^{- \frac{\eta_0}{2}(t-s)} + 2 K)}{2 + \beta R} \Big\} .
$$
Note that $\alpha(s,t) \leq 1 + \beta K$. Moreover, since $\inf_n \varepsilon_R(n,n+1)>0$ and $R > \frac{2K}{1-e^{-\eta_0 /2}}$, there exists $\beta\in(0,1)$ such that $\eta \coloneqq -\log(\alpha(0, 1)) = -\log(\alpha(n,n+1))>0$, for all $n \in \mathbb{Z}$. We obtain, using the discretization $s < \lfloor s \rfloor + 1 < \lfloor s \rfloor +2 < \dots < \lfloor t \rfloor \leq t$ for $s \leq t \in R$,
\begin{equation}\label{eq:mixing_prop_proof}
    d_{\beta}(\mu P_{s,t} \, , \nu P_{s,t}) \leq (1 + \beta K)^2 \,  e^{-\eta (\lfloor t \rfloor - \lfloor s \rfloor-1)} \, d_{\beta}(\mu, \nu) \mathbf{1}_{\{t-s \geq 1 \}} + (1 + \beta K) \, d_{\beta}(\mu, \nu) \mathbf{1}_{\{ t-s < 1 \}} \, .
\end{equation}
Noting that $\lfloor t \rfloor - \lfloor s \rfloor-1 \geq t-s-2$ and $\beta d_\beta\leq \beta  d_{\rm{TV},2}\leq d_\beta$ since $\beta\in(0,1)$, \eqref{eq:mixing_prop_proof} implies 
the mixing property \eqref{feller_inhom} with $C \coloneqq \frac{e^{2\eta}(1 + \beta K)^2}{\beta}$. \\

\vspace{-8pt}

\noindent\underline{Step 3:} We construct an evolution system of measures $(m_t)_{t \in \R}$ for the semi-group $P$. Let $m_{s,t} \coloneqq \delta_0 P_{s,t}$ for $s \leq t \in \R$ and note that $\sup_{s,t \in \R} m_{s,t}(|\cdot|^2) \leq K$ by \eqref{eq:lyapunov_P}. Then, by \eqref{eq:mixing_prop_proof} and \eqref{eq:lyapunov_P}, for all $n,k \in \mathbb{Z}$ with $n \leq k \leq t$,
\begin{equation*}
    d_{\beta}(m_{n,t} \, , m_{k, t}) \, = \, d_{\beta}(\delta_0 P_{n,k} P_{k, t} \, , \delta_0 P_{k,t}) \, \leq \, C e^{-\eta (t-k)} \, d_{\beta}(\delta_0 P_{n,k} \, , \delta_0) \, \leq \, C e^{-\eta (t-k)} \, (2 + \beta K) \, .
\end{equation*}
Hence, $(m_{k,t})_{k \in \mathbb{Z}}$ is a Cauchy sequence for $d_{\beta}$, and therefore converges to some $m_t \in \mathcal{P}_2$ satisfying $m_t(|\cdot|^2) \leq K$ and $m_s P_{s,t} = m_t$ for all $s \leq t$. Finally, let $(\tilde{m}_t)_{t \in \R}$ be another evolution system of measures for the semi-group $P$ satisfying the integrability condition $\sup_{t \in \R} \tilde{m}_t(|\cdot |^2) < \infty$. Then,
\begin{equation*}
    d_{{\rm{TV}}, 2}( m_t \, , \tilde{m}_t) \, = \, d_{{\rm{TV}}, 2}(m_s P_{s,t} \, , \tilde{m}_s P_{s,t}) \, \leq \, C \big(1 + \sup_{\tau \in \R} m_{\tau}(|\cdot|^2) + \sup_{\tau \in \R} \tilde{m}_{\tau}(|\cdot|^2)\big) \, e^{-\eta (t-s)},~s \leq t \in \R.
\end{equation*}
Sending $s \rightarrow - \infty$, we obtain $d_{{\rm{TV}}, 2}( m_t \, ,  \tilde{m}_t) \leq 0$, which implies that $m_t = \tilde{m}_t$.
\end{proof}

\end{spacing}
\vspace{-5mm}
\begingroup
\printbibliography

@incollection{hairer_mattingly_harris_2011,
  author    = {Hairer, Martin and Mattingly, Jonathan C.},
  title     = {Yet another look at Harris' ergodic theorem for Markov chains},
  booktitle = {Seminar on Stochastic Analysis, Random Fields and Applications VI},
  series    = {Progress in Probability},
  volume    = {63},
  pages     = {109--117},
  publisher = {Birkh{\"a}user/Springer Basel AG},
  address   = {Basel},
  year      = {2011}
}

@article{MENOZZI2021330,
title = {Density and gradient estimates for non degenerate Brownian SDEs with unbounded measurable drift},
journal = {Journal of Differential Equations},
volume = {272},
pages = {330-369},
year = {2021},
issn = {0022-0396},
author = {S. Menozzi and A. Pesce and X. Zhang}
}

@misc{wu2026ergodiclinearquadraticoptimalcontrol,
      title={The Ergodic Linear-Quadratic Optimal Control Problems with Random Periodic Coefficients}, 
      author={Wu, J. and Zhang, Q.},
      year={2026},
      eprint={2601.08672},
      archivePrefix={arXiv},
      primaryClass={math.OC}, 
}

@article{Korevaar2004,
author = {Korevaar, J.},
year = {2004},
pages = {},
title = {Tauberian theory. A century of developments},
journal = {Journal of High Energy Physics}
}

@article{BishopFeinbergZhang2014,
  author  = {Bishop, C. J. and Feinberg, E. A. and Zhang, J.},
  title   = {Examples concerning Abel and Ces{\`a}ro limits},
  journal = {Journal of Mathematical Analysis and Applications},
  year    = {2014},
  volume  = {420},
  number  = {2},
  pages   = {1654--1661},
  issn    = {0022-247X}
}

@article{Hyers1941,
  author  = {Hyers, D. H.},
  title   = {On the Stability of the Linear Functional Equation},
  journal = {Proceedings of the National Academy of Sciences of the United States of America},
  year    = {1941},
  volume  = {27},
  number  = {4},
  pages   = {222--224}
}

@article{Guo_Huang_Zhang,
    author = {Guo, X. and Huang, Y. and Zhang, Y.} ,
    title = {On Average Optimality for Non-Stationary Markov Decision Processes in Borel Spaces} ,
    journal = {Mathematics of Operations Research} ,
    year = {2024}
}

@article{furhman_ergodic,
    author = {Furhman, M. and Hu, Y. and Tessitore, G.} ,
    title = {Ergodic BSDEs and Optimal Ergodic Control in Banach Spaces} ,
    journal = {Siam J. Control Optim..} ,
    volume = {48},
    number = {3},
    pages = {1542-1566},
    year = {2009}
}

@article{ergodic_weak_dissipative,
    author = {Debussche, A. Hu, Y. and Tessitore, G.} ,
    title = {Ergodic BSDEs under weak dissipative assumptions} ,
    journal = {Stochastic Processes and their Applications} ,
    volume = {121},
    pages = {407-426},
    year = {2011}
}

@article{Arisawa_Lions,
    author = {Arisawa, M. and P.-L. Lions} ,
    title = {On ergodic stochastic control} ,
    journal = {Communications in Partial Differential Equations} ,
    pages = {2187-2217},
    volume = {23:11-12},
    year = {1998}
}

@article{Bensoussan_Frehse ,
    author = {Bensoussan, A. and Frehse, J.} ,
    title = {On Bellman Equations of Ergodic Control in $\mathbb{R}^n$} ,
    journal = {Journal für die reine und angewandte Mathematik} ,
    volume = {429},
    pages = {125-160},
    year = {1992}
}

@article{Lasry_Lions_state_constraints,
    author = {Lasry, J.-M. and Lions, P.-L.} ,
    title = {Nonlinear Elliptic Equations with Singular Boundary Conditions and Stochastic Control with State Constraints} ,
    journal = {Math. Ann.} ,
    volume = {283},
    pages = {583-630},
    year = {1989}
}

@article{Benaim_ergodicity_inhomogeneous,
    author = {Benaïm, M. and Bouguet, F. and Cloez, B.} ,
    title = {Ergodicity of Inhomogeneous Markov Chains Through Asymptotic Pseudotrajectories} ,
    journal = {The Annals of Probability} ,
    volume = {27},
    number = {5},
    pages = {3004-3049},
    year = {2017}
}

@article{principal_eigenvalues_parabolic,
    author  = {Berestycki, H. and Nadin, G. and Rossi, L.},
    title   = {Generalized principal eigenvalues for parabolic operators in bounded domains},
    journal = {Annali Scuola Normale Superiore - Classe di Scienze},
    volume  = {38},
    year    = {2025}
}

@book{Ergodic_FSDE,
    author = {Bao, J. and Yin, G. and Yuan, C.} ,
    title = {Asymptotic Analysis for Functional Stochastic Differential Equations} ,
    publisher = {Springer} ,
    year = {2016}
}

@article{LQ_italiennes,
    author = {Guatteri, G. and Masiero, F.} ,
    title = {Infinite Horizon and Ergodic Optimal Quadratic Control for an Affine Equation with Stochastic Coefficients} ,
    journal = {Siam J. Control Optim.} ,
    volume = {48},
    number = {3},
    pages = {1600-1631},
    year = {2009}
}

@book{zhang_BSDE,
    author = {Zhang, J.} ,
    title = {Backward Stochastic Differential Equations} ,
    publisher = {Springer} ,
    year = {2017}
}

@article{ergodic_BSDE_neumann,
    author = {Richou, A.} ,
    title = {Ergodic BSDEs and related PDEs with Neumann boundary conditions} ,
    journal = {Stochastic Processes and their Applications} ,
    volume = {19},
    pages = {2945-2969},
    year = {2009}
}

@article{ergodic_performance_process,
    author = {Liang, G. and Zariphopoulou, T.} ,
    title = {Representation of Homothetic Forward Performance Processes in Stochastic Factor Models via Ergodic and Infinite Horizon BSDE},
    journal = {Siam J. Financial Math.} ,
    volume = {8},
    pages = {344-372},
    year = {2017}
}

@article{ergodic_bsde_superquadratic,
    author = {Jackson, J. and Liang, G.} ,
    title = {A new monotonicity condition for ergodic BSDEs and ergodic control with super-quadratic hamiltonians} ,
    journal = {SIAM J. Control Optim.} ,
    volume = {61},
    number = {3},
    pages = {1273-1296},
    year = {2023}
}

@article{FSDE_infinite_delay,
    author = {Wu, F. and Yin, G. and Mei, H.} ,
    title = {Stochastic functional differential equations with infinite delay: Existence and uniqueness of solutions, solution maps, Markov properties, and ergodicity} ,
    journal = {J. Differential Equations} ,
    volume = {262},
    pages = {1226–1252},
    year = {2017}
}

@article{HJB_delay,
    author = {Furhman, M. and Masiero, F. and Tessitore, G.} ,
    title = {Stochastic Equations with Delay: Optimal Control via BSDEs and Regular Solutions of Hamilton-Jacobi-Bellman Equations} ,
    journal = {Siam J. Control Optim.} ,
    volume = {48},
    number = {7},
    pages = {4624-4651},
    year = {2010}
}

@article{Infinite_delay_nonhomogeneous,
    author = {Pu, Z. and Pan, Z. and Li, D.} ,
    title = {The Existence of Evolution Systems of Measures of Non-autonomous Stochastic Differential Equations with Infinite Delays} ,
    journal = {Qualitative Theory of Dynamical Systems} ,
    volume = {159},
    year = {2022}
}

@article{Goldys_Maslowski,
    author = {Goldys, B. and Maslowski, B.} ,
    title = {Ergodic Control of Semilinear Stochastic Equations and the Hamilton-Jacobi Equation} ,
    journal = {Journal of Mathematical Analysis and Applications} ,
    pages = {592-631},
    year = {1999}
}

@article{OU_periodic,
    author = {Da Prato, G. and Lunardi, A.} ,
    title = {Ornstein–Uhlenbeck operators with time periodic coefficients} ,
    journal = {Journal of Evolution Equations} ,
    volume = {7},
    pages = {587–614},
    year = {2007}
}

@article{Evolution_system_measures,
    author = {Da Prato, G. and Röckner, M} ,
    title = {A note on non autonomous stochastic differential equations} ,
    journal = {Proceedings of the 5th Seminar on Stochastic Analysis, Random Fields and Applications} ,
    year = {2005}
}

@book{Mohammed_FSDE,
    author = {Mohammed, S. E. A.} ,
    title = {Stochastic  Differential Systems with Memory: Theory, Examples and Applications} ,
    publisher = {Decreusefond, L., Øksendal, B., Gjerde, J., Üstünel, A.S. (eds) Stochastic Analysis and Related Topics VI. Progress in Probability} ,
    volume = {42},
    year = {1998}
}

@article{Feo_control_delay,
    author = {De Feo, F. and Federica, S. and Swiech, A.} ,
    title = {Optimal Control of Stochastic Delay Differential Equations and Applications to Path-Dependent Financial and Economic Models} ,
    journal = {Siam J. Control. Optim.} ,
    volume = {62},
    number = {3},
    pages = {1490-1520},
    year = {2024}
}

@article{zero_sum_non_stationary_ergodic,
    author = {Zheng, Z. and Guo, X.} ,
    title = {Zero-Sum Non-Stationary Stochastic Games with the Long-Run Average Criterion} ,
    journal = {Applied Mathematics and Optimization} ,
    volume = {90},
    year = {2024}
}

@book{ergodic_control_diffusion_processes,
    author = {Arapostathis, A. and Borkar, V.S. and Ghosh, M.K.} ,
    title = {Ergodic Control of Diffusion Processes} ,
    publisher = {Cambridge University Press} ,
    year = {2011}
}

@article{Harnack_inequalities,
    author = {H\'{u}ska, J. and Pol\'{a}\v{c}ik, P. and Safonov, M. V.} ,
    title = {Harnack inequalities, exponential separation, and pertubations of principal Floquet bundles for linear parabolic equations} ,
    journal = {Annales de l'Institut Henri Poincaré} ,
    volume = {24},
    pages = {711-739},
    year = {2007}
}

@article{quasitrationary_ergodic,
    author = {Budhiraja, A. and Dupuis, P. and Nyquist, P. and Wu, G.-J.} ,
    title = {Quasistationary Distributions and Ergodic Control Problems} ,
    journal = {Elsevier} ,
    year = {2021}
}

@article{principal_floquet_bundle,
    author = {H\'{u}ska, J. and Pol\'{a}\v{c}ik, P.} ,
    title = {The Principal Floquet Bundle and Exponential Separation for Linear Parabolic Equations} ,
    journal = {Journal of Dynamics and Differential Equations} ,
    volume = {6},
    number = {2},
    year = {2004}
}

@article{cohen_periodic_EBSDE,
    author = {Cohen, S. N. and Fedyashov, V.} ,
    title = {Ergodic BSDEs with jumps and time dependence} ,
    journal = {arXiv:1406.4329v2} ,
    year = {2015}
}

@article{Ergodicity_neutral_SDE,
    author = {Bao, J. and Wang, F.-Y. and Yuan, C.} ,
    title = {Ergodicity for neutral type SDEs with infinite length of memory} ,
    journal = {Mathematische Nachrichten} ,
    volume = {293},
    pages = {1675-1690},
    year = {2020}
}

@article{Lpz_approx,
    author = {Bogachev, V.I. and Shkarin, S.A.} ,
    title = {Differentiable and Lipschitzian mappings of Banach spaces} ,
    journal = {Math. Notes} ,
    volume = {44},
    pages = {790–798},
    year = {1988}
}

@article{Expo_mixing_finite_delay,
    author = {Butkovsky, O. and Scheutzow, M.} ,
    title = {Invariant measures for stochastic functional differential equations} ,
    journal = {Electron. J. Probab.} ,
    number = {98},
    pages = {1-23},
    year = {2017}
}

@article{hu_lemonier,
    author = {Hu, Y. and Lemonnier, F.} ,
    title = {Ergodic BSDE with unbounded and multiplicative underlying diffusion and application to large time behaviour of viscosity solution of HJB equation} ,
    journal = {Stochastic Processes and their Applications} ,
    volume = {129},
    pages = {4009-4050},
    year = {2019}
}

@article{borkar,
    author = {Borkar, V. S. and Ghosh, M. K.} ,
    title = {Ergodic control of multidimensional diffusions I: The existence results} ,
    journal = {SIAM Journal on Control and Optimization} ,
    volume = {26},
    number = {1},
    pages = {112-126},
    year = {1988}
}

@article{entrance_measure_inhomogeneous,
    author = {Feng, C. and Qu, B. and Zhao, H.} ,
    title = {Entrance measures for semigroups of time-inhomogeneous SDEs: possibly degenerate and expanding} ,
    journal = {arXiv:2307.07891v1} ,
    year = {2023}
}
\endgroup

\end{document}